\documentclass[10pt,reqno]{amsart}
\usepackage{amsmath,amssymb,amsthm,graphicx,epstopdf,mathrsfs,url}
\usepackage[usenames,dvipsnames]{color}
\usepackage{dsfont}   
\definecolor{darkred}{rgb}{0.7,0.1,0.1}
\definecolor{darkblue}{rgb}{0.1,0.1,0.4}
\definecolor{darkgrey}{rgb}{0.5,0.5,0.5}
\usepackage[colorlinks=true,linkcolor=darkred,citecolor=Blue]{hyperref}

\numberwithin{equation}{section}
\theoremstyle{plain}
\newtheorem{thm}{Theorem}[section]
\newtheorem{lem}[thm]{Lemma}

\newtheorem{prop}[thm]{Proposition}
\newtheorem{cor}[thm]{Corollary}

\newtheorem{definition}[thm]{Definition}
\theoremstyle{remark}
\newtheorem{remark}[thm]{Remark}

\theoremstyle{plain}

\newcommand{\hyp}[1]{$C^{2}$-hypersurface as in Definition~\ref{definition_hypersurface}}

\DeclareMathOperator\ran{ran}

\newcommand{\dom}{\mathrm{dom}\,}

\begin{document}
\title[]{Self-adjoint {D}irac operators on domains in $\mathbb{R}^3$}
\author[]{}

\author[J. Behrndt]{Jussi Behrndt}
\address{Institut f\"{u}r Angewandte Mathematik\\
Technische Universit\"{a}t Graz\\
 Steyrergasse 30, 8010 Graz, Austria\\
E-mail: {\tt behrndt@tugraz.at}}

\author[M. Holzmann]{Markus Holzmann}
\address{Institut f\"{u}r Angewandte Mathematik\\
Technische Universit\"{a}t Graz\\
 Steyrergasse 30, 8010 Graz, Austria\\
E-mail: {\tt holzmann@math.tugraz.at}}

\author[A.~Mas]{Albert Mas}
\address{
Departament de Matem\`atiques\\
Universitat Polit\`ecnica de Catalunya\\
Campus Diagonal Bes\`os, Edifici A (EEBE), Av. Eduard Maristany 16, 08019
Barcelona, Spain \\
E-mail: {\tt albert.mas.blesa@upc.edu}}

\begin{abstract}
In this paper the spectral and scattering properties of a family of self-adjoint Dirac operators in $L^2(\Omega; \mathbb{C}^4)$, where $\Omega \subset \mathbb{R}^3$ is either a bounded or an unbounded domain 
with a compact $C^2$-smooth boundary, are studied in a systematic way. These operators can be viewed as the natural relativistic counterpart of Laplacians with 
Robin boundary conditions.
Among the Dirac operators treated here is also the so-called MIT bag operator, which has been used by physicists and more recently was discussed in the mathematical literature.
Our approach is based on abstract boundary triple techniques from extension theory of symmetric operators and a thorough study of certain classes of (boundary) integral operators,
that appear in a Krein-type resolvent formula. The analysis of the perturbation term in this formula leads to a description of the spectrum and a Birman-Schwinger principle,
a qualitative understanding of the scattering properties in the case that $\Omega$ is unbounded, and corresponding trace formulas.
\end{abstract}


\maketitle

\section{Introduction} \label{section_introduction}

In recent years the mathematical study of Dirac operators acting on domains $\Omega \subset \mathbb{R}^d$  with special boundary conditions 
that make them self-adjoint gained a lot of attention. 
The motivation for this arises from several aspects: from the mathematical point of view 
they can be seen as the relativistic counterpart of Laplacians
with Robin type boundary conditions. 
From the physical point of view Dirac operators with special boundary conditions are used in 
relativistic quantum mechanics to describe particles that are confined to a predefined area or box. 
One important model in 3D (dimension three) is the MIT bag model suggested in the 1970s 
to study confinement of quarks, see \cite{C75, CJJT74, CJJTW74, DJJK75, J75}.
In the 2D (dimension two) case Dirac operators with special boundary conditions similar to the MIT bag model are used
in the description of graphene; cf. \cite{BFSB17_1, BFSB17_2, CL19}.

To set the stage, let $\Omega \subset \mathbb{R}^3$ be either a bounded or unbounded domain with a compact $C^2$-smooth boundary 
and let $\nu$ be the unit normal vector field at $\partial \Omega$ which points outwards of $\Omega$. Choose units such that the Planck constant $\hbar$ and the speed of light are both equal to one. Moreover, assume that 
$\vartheta: \partial \Omega \rightarrow \mathbb{R}$ is a H\"older continuous function of order $a>\frac{1}{2}$ and consider in $L^2(\Omega; \mathbb{C}^4)$
the operator
\begin{equation} \label{def_op_domain_intro}
  \begin{split}
    A_\vartheta f  &= -i \alpha \cdot \nabla f + m \beta f= -i \sum_{j=1}^3 \alpha_j \cdot \partial_j f + m \beta f, \\
    \dom A_\vartheta 
        &= \big\{ f \in H^1(\Omega; \mathbb{C}^4):
        \vartheta \big(I_4 + i \beta (\alpha \cdot \nu)\big) f|_{\partial \Omega} 
      = \big(I_4 + i \beta (\alpha \cdot \nu)\big) \beta f|_{\partial \Omega} \big\},
  \end{split}
\end{equation}
where $\alpha = (\alpha_1, \alpha_2, \alpha_3)$ and $\beta$ are the $\mathbb{C}^{4 \times 4}$ Dirac matrices defined 
in~\eqref{def_Dirac_matrices} below and $\alpha \cdot x = \alpha_1 x_1 + \alpha_2 x_2 + \alpha_3 x_3$ for $x = (x_1, x_2, x_3)^\top \in \mathbb{R}^3$.
The time-dependent equation with the Hamiltonian given by $A_\vartheta$ models the propagation of a relativistic particle subject to the boundary conditions in $\dom A_\vartheta$
with mass $m > 0$ contained in $\Omega$.

The mathematical literature on such types of Dirac operators contains different approaches. 
In differential geometry there are several articles dealing with self-adjoint Dirac operators on smooth manifolds,
see for instance \cite{BB12, BB16, R06}. In \cite{S95} it was shown that the 2D Dirac operator 
with so-called zigzag boundary conditions (in the massless case) is self-adjoint 
and zero is an eigenvalue of infinite multiplicity, see also \cite{FS14}. More recent related publications 
in the 2D case are \cite{BFSB17_1, BFSB17_2}, where the self-adjointness of Dirac operators in bounded $C^2$-domains $\Omega \subset \mathbb{R}^2$ for a wide class of 
boundary conditions describing quantum dots was shown. Many considerations in \cite{BFSB17_1, BFSB17_2, S95} are based on complex analysis techniques, which are not available in 
the 3D situation. We also refer to \cite{CL19, LTO18, LO18, PB19} for self-adjointness and spectral problems of 2D Dirac operators on different types of domains 
with special boundary conditions. In contrast to the 2D setting the 3D case is less investigated in the mathematical literature, only the MIT bag operator is well studied.
We emphasize \cite{ALTR17, OV17} for the analysis of general properties of the MIT bag operator in 3D and \cite{ALTMR19,BCLS19,B19,MOP18}, where it is shown that the MIT bag boundary conditions can be interpreted as infinite mass boundary conditions (i.e. $\Omega$ is surrounded by a medium with infinite mass).

The main objective of this paper is to develop a systematic approach to the spectral analysis and scattering theory for self-adjoint Dirac operators in the 3D case. 
Here we are particularly interested in boundary conditions as in \eqref{def_op_domain_intro}, since these are the 3D analogue of the 2D boundary conditions in
\cite{BFSB17_1} which can be used to describe graphene quantum dots (cf. Remark~\ref{rem_bc_klar}), and the corresponding Dirac operators can be viewed 
as the relativistic counterpart of Laplace operators with 
Robin boundary conditions. We also note that operators of the form \eqref{def_op_domain_intro} appear in the treatment of Dirac operators 
\begin{equation}\label{bet}
B_{\eta, \tau}  = -i \alpha \cdot \nabla + m \beta + (\eta I_4 + \tau \beta) \delta_{\partial \Omega}
\end{equation}
in $\mathbb R^3$ 
with singular $\delta$-shell potentials supported on $\partial\Omega$ in the confinement (or decoupling) case. In fact, it is well known that for $\eta^2 - \tau^2 = -4$ the operator $B_{\eta, \tau}$ can be written as the orthogonal sum of operators acting in $L^2(\Omega; \mathbb{C}^4)$ and $L^2(\mathbb{R}^3 \setminus \overline{\Omega}; \mathbb{C}^4)$, respectively, and it turns out that these operators are exactly of the form~\eqref{def_op_domain_intro}, see 
Section~\ref{section_confinement} for more details. We refer to the recent contributions 
\cite{AMV14, AMV15, AMV16, BEHL18, BEHL19_1, BH17, BHOP19, HOP17, M17, MP18a, MP18b, OP19, OV17} for a comprehensive study of Dirac operators with singular $\delta$-shell potentials.
Hence, $A_\vartheta$
can be used to describe a relativistic particle actually living in~$\mathbb{R}^3$, but which is confined to $\Omega$ for all time, see \cite[Section 5]{AMV15}.
This is of interest in particle physics.

The mathematical treatment of the operators $A_\vartheta$ in \eqref{def_op_domain_intro} is based on the application of a suitable so-called quasi boundary triple. Quasi boundary triples and their Weyl functions are an 
abstract concept from extension and spectral theory for symmetric and self-adjoint operators which were originally introduced to investigate boundary value problems for 
elliptic partial differential operators in \cite{BL07}, but proved to be useful in many other situations, see, e.g., \cite{BEHL19_2,BLL13_2,BR15_2}.
Quasi boundary triples were also applied more recently in \cite{BEHL18, BH17} to Dirac operators with singular potentials
as in \eqref{bet}.
Once a quasi boundary triple and Weyl function in the present situation are available,
they allow to deduce in an efficient way the spectral properties of $A_\vartheta$ from the properties of certain (boundary) integral operators 
which are induced by the Green's function of the free Dirac operator in $\mathbb{R}^3$. The more demanding issue here is to establish 
the proper mapping properties of these integral operators and, in fact, this analysis covers a great part of the present paper. We would like to point out that this approach is independent of the space dimension.

One of the key features in the quasi boundary triple approach is a Krein-type resolvent formula that relates the resolvent 
of $A_\vartheta$ via a perturbation term to the resolvent of a reference operator, which is in our model the MIT bag operator~$T_{\textup{MIT}}$.  
Using this correspondence we first conclude that $A_\vartheta$ in \eqref{def_op_domain_intro} is self-adjoint in 
$L^2(\Omega; \mathbb{C}^4)$ and we show several spectral and scattering properties of $A_\vartheta$, which are different for bounded and for unbounded domains~$\Omega$. If
$\Omega$ is an unbounded domain with a compact $C^2$-boundary, then the essential spectrum of $A_\vartheta$ is given by $(-\infty, -m] \cup [m, \infty)$, and there are at 
most finitely many discrete eigenvalues in the gap $(-m,m)$ that
can be characterized by a Birman-Schwinger principle.
It also follows that $(A_\vartheta - \lambda)^{-3} 
        - ( T_{\textup{MIT}} - \lambda )^{-3}$ is a trace class operator for any 
        $\lambda \in \mathbb{C} \setminus \mathbb{R}$, which implies the existence and completeness of the wave operators for the scattering pair $\{A_\vartheta,T_{\textup{MIT}}\}$.
If $\Omega$ is a bounded domain with a compact $C^2$-boundary, then the spectrum of $A_\vartheta$ is purely discrete and all eigenvalues of $A_\vartheta$ can be characterized by a modified Birman-Schwinger principle;
cf. Section~\ref{section_def_op} for more details.
The abovementioned properties are proved under the assumption $\vartheta(x)^2 \not= 1$ for all $x\in\partial\Omega$, which we refer to as the {\it non-critical case}.
We expect that in the {\it critical case} $\vartheta(x)^2 = 1$ for some  $x\in\partial\Omega$ the spectral properties of $A_\vartheta$ may significantly differ from the 
non-critical case, e.g.,
essential spectrum may arise also for bounded domains or in the gap $(-m,m)$.  Similar difficulties and effects 
were observed in the 2D situation in \cite{BFSB17_1} and also in the analysis of Dirac operators with singular interactions in \cite{AMV14, BEHL18, BH17, BHOP19, OV17}.

We mention that for some models it is more convenient to consider Dirac operators $A_{[\omega]}$ in $L^2(\Omega; \mathbb{C}^4)$ with boundary conditions of the form
\begin{equation} \label{boundary_condition_B_intro}
  \big(I_4 + i \beta (\alpha \cdot \nu)\big) f|_{\partial \Omega} 
      = \omega \big(I_4 + i \beta (\alpha \cdot \nu)\big) \beta f|_{\partial \Omega},
\end{equation}
where $\omega: \partial \Omega \rightarrow \mathbb{R}$ is a H\"older continuous function. Comparing with the boundary conditions in 
\eqref{def_op_domain_intro} one formally has $\omega=\vartheta^{-1}$. Note that the particularly interesting MIT bag model corresponds to $\omega \equiv 0$. 
Using the abstract quasi boundary triple approach the spectral and scattering properties of $A_{[\omega]}$ can be studied in the same way as those of $A_\vartheta$, and similar results as sketched
above for $A_\vartheta$ follow; cf. Section~\ref{ss bdy MIT}.

Let us give a short overview on the structure of the paper. In Section~\ref{section_boundary_triples} we review the definitions of 
quasi boundary triples and their associated Weyl functions.  Then, in Section~\ref{section_preliminary_ops} we recall some knowledge on 
a minimal and a maximal realization of the Dirac operator in $\Omega$, the MIT bag model, and the properties of several families of integral 
operators associated to the resolvent of the free Dirac operator. Next, in Section~\ref{section_boundary_triples_domain} we introduce and study a 
quasi boundary triple which is suitable to investigate Dirac operators in $\Omega$ with boundary conditions. Section~\ref{section_Dirac_domain} contains 
the main results of the present paper. 
After defining $A_\vartheta$ and proving its self-adjointness in Section~\ref{section_def_op} we use the quasi boundary triple from 
Section~\ref{section_boundary_triples_domain} to conclude various spectral properties. 
Section~\ref{ss bdy MIT} is then devoted to the study of the operator $A_{[\omega]}$ with the boundary conditions~\eqref{boundary_condition_B_intro}, 
while in Section~\ref{section_confinement} we discuss the earlier mentioned connection of $A_\vartheta$ and $A_{[\omega]}$ to the operator 
$B_{\eta, \tau}$ formally given in~\eqref{bet}. Finally, in an appendix we collect some material on integral operators and their mapping properties 
in Sobolev spaces on the boundary $\partial \Omega$, which is applied in the proofs of the main results of this paper.

\subsection*{Notations}
Throughout this paper $m$ is always a positive constant that stands for the mass of a particle.
The Dirac matrices $\alpha_1, \alpha_2, \alpha_3, \beta \in \mathbb{C}^{4 \times 4}$ are defined by
\begin{equation} \label{def_Dirac_matrices}
   \alpha_j := \begin{pmatrix} 0 & \sigma_j \\ \sigma_j & 0 \end{pmatrix}
   \quad \text{and} \quad \beta := \begin{pmatrix} I_2 & 0 \\ 0 & -I_2 
\end{pmatrix},
\end{equation}
where $I_n$ is the $n \times n$-identity matrix and $\sigma_j$, $j \in\{ 1, 2, 3 \}$, are the Pauli spin matrices
\begin{equation} \label{def_Pauli_matrices}
   \sigma_1 := \begin{pmatrix} 0 & 1 \\ 1 & 0 \end{pmatrix}, \qquad
   \sigma_2 := \begin{pmatrix} 0 & -i \\ i & 0 \end{pmatrix}, \qquad
   \sigma_3 := \begin{pmatrix} 1 & 0 \\ 0 & -1 \end{pmatrix}.
\end{equation}
The Dirac matrices satisfy the anti-commutation relations
\begin{equation}\label{eq_commutation}
	\alpha_j\alpha_k + \alpha_k\alpha_j = 2\delta_{jk}I_4,\quad \text{and} \quad \alpha_j \beta + \beta \alpha_j = 0, \qquad
	j,k\in\{1,2,3\}.
\end{equation}
For vectors $x = (x_1, x_2, x_3)^{\top}\in\mathbb{R}^3$ we often use the notation
$\alpha \cdot x := \sum_{j=1}^3 \alpha_j x_j$. 

The upper/lower complex half plane is denoted by $\mathbb{C}_\pm$.
The square root $\sqrt{\cdot}$ is fixed by $\mathrm{Im}\sqrt{\lambda}>0$ for $\lambda\in\mathbb C\setminus [0,\infty)$ and $\sqrt{\lambda}\geq 0$
for $\lambda\geq 0$.
The open ball of radius $r > 0$ centered at $x \in \mathbb{R}^3$ is denoted by $B(x, r)$.
For a $C^2$-domain $\Omega \subset \mathbb{R}^3$ we write ${\partial \Omega}$ for its boundary
and $\sigma$ is the 2-dimensional Hausdorff measure on ${\partial \Omega}$. We shall mostly work with the $L^2$-spaces 
$L^2(\Omega; \mathbb{C}^n)$ and $L^2({\partial \Omega}; \mathbb{C}^n)$
of $\mathbb C^n$-valued square integrable functions,
the corresponding inner products being denoted by $(\cdot, \cdot)_\Omega$ and 
$(\cdot, \cdot)_{\partial \Omega}$, respectively.
We write 
$C_0^\infty(\Omega; \mathbb{C}^n)$ for the space of 
$\mathbb C^n$-valued smooth functions with compact support in $\Omega$ and we set
\begin{equation*}
  C^\infty(\overline{\Omega}; \mathbb{C}^n) := \big\{ f \upharpoonright \Omega: f \in C_0^\infty(\mathbb{R}^3; \mathbb{C}^n) \big\}.
\end{equation*}
We write $H^k(\mathbb R^3; \mathbb{C}^n)$ for the usual 
$L^2(\mathbb R^3; \mathbb{C}^n)$-based Sobolev space of $k$-times 
weakly differentiable 
functions, and similarly $H^k(\Omega; \mathbb{C}^n)$. In addition, $H^1_0(\Omega; \mathbb{C}^n)$ denotes the closure of
$C_0^\infty(\Omega; \mathbb{C}^n)$
in $H^1(\Omega; \mathbb{C}^n)$. Sobolev spaces on $C^l$-surfaces ${\partial \Omega}$, $l \in \mathbb{N}$, are denoted by
$H^s({\partial \Omega}; \mathbb{C}^n)$, $s \in (0, l)$, and the symbol $H^{-s}(\partial \Omega; \mathbb{C}^n)$ is used for their duals. The corresponding norm for $s \in (0, 1)$ is
\begin{equation} \label{def_Sobolev_Slobodeckii}
\|f\|_s^2:=\int_{{\partial \Omega}}|f(x)|^2 \text{d} \sigma(x)
+\int_{{\partial \Omega}}\int_{{\partial \Omega}}
\frac{|f(x)-f(y)|^2}{|x-y|^{2+2s}} \text{d} \sigma(y) \text{d} \sigma(x).
\end{equation}
The trace of a function $f \in H^1(\Omega; \mathbb{C}^n)$, which belongs by the trace theorem
to $H^{1/2}(\partial \Omega; \mathbb{C}^n)$, is denoted by $f|_{\partial \Omega}$. 
Eventually, given $0<a\leq1$ we denote the H\"older continuous functions on 
$\partial\Omega$ of order $a$ by 
\begin{equation*}
\text{Lip}_a({\partial \Omega}):=
\bigl\{f:{\partial \Omega}\to\mathbb{C}:\,|f(x)-f(y)|\leq C|x-y|^a\text{ for all }x,\,y\in{\partial \Omega}\bigr\}.
\end{equation*}

For two Hilbert spaces $\mathcal{G}$ and $\mathcal{H}$ the space $\mathcal B(\mathcal{G}, \mathcal{H})$
is the set of all bounded and everywhere defined operators from $\mathcal{G}$ to $\mathcal{H}$. 
If $\mathcal{G} = \mathcal{H}$, then 
we simply write $\mathcal B(\mathcal{H})$.
We write $\mathfrak{S}_{p, \infty}(\mathcal{G}, \mathcal{H})$ for the weak 
Schatten--von Neumann ideal of order $p>0$;
this is the set of all compact operators $K: \mathcal{G} \rightarrow \mathcal{H}$ for which
there exists a constant $\kappa$ such that the singular values $s_k(K)$ 
of $K$ fulfill $s_k(K) \leq \kappa k^{-1/p}$ for all $k \in \mathbb{N}$, see \cite{GK69} or \cite[Section 2.1]{BLL13_2}. Again we use $\mathfrak{S}_{p, \infty}(\mathcal{G})$
if $\mathcal{G} = \mathcal{H}$ and sometimes we suppress the spaces and just write $\mathfrak{S}_{p, \infty}$.

For a linear operator $T:\mathcal{G}\rightarrow \mathcal{H}$ we denote the domain, range, and kernel by $\dom T$, $\ran T$, 
and $\ker T$, respectively. If $T$ is a self-adjoint operator in $\mathcal{H}$ then its resolvent
set, spectrum, essential spectrum, discrete, and point spectrum are denoted by 
$\rho(T)$, $\sigma(T)$, $\sigma_{\text{ess}}(T)$, $\sigma_{\text{disc}}(T)$,
and $\sigma_{\rm p}(T)$, respectively.
Next, for a Banach space $X$ we use the notation $(\cdot, \cdot)_{X'\times X}$ for the duality product in $X' \times X$ which is linear in its first and anti-linear 
in the second entry. Moreover, for $T \in \mathcal{B}(X, Y)$ we denote by $T' \in \mathcal{B}(Y', X')$ the anti-dual operator, which is uniquely determined by the relation
\begin{equation*}
  (y, T x)_{Y' \times Y} = (T' y, x)_{X' \times X}
\end{equation*}
for all $x \in X$ and $y \in Y'$.

Finally, we call a closed symmetric operator $S$ in a Hilbert space $\mathcal{H}$ simple, if for any orthogonal decomposition
$\mathcal{H} = \mathcal{H}_1 \oplus \mathcal{H}_2$ such that $\mathcal{H}_1$ and $\mathcal{H}_2$ are invariant under $S$
and $S_1 := S \upharpoonright \mathcal{H}_1$ is self-adjoint in $\mathcal{H}_1$ it follows $\mathcal{H}_1 = \{ 0 \}$. 
It was observed in \cite{K49} that a closed symmetric operator $S$ is simple, if and only if 
\begin{equation} \label{equation_check_simple}
  \overline{\text{span} \big\{ \ker(S^* - \lambda) : \lambda \in \mathbb{C} \setminus \mathbb{R} \big\} }
  = \mathcal{H}.
\end{equation}

\section{Quasi boundary triples and their Weyl functions} \label{section_boundary_triples}

This section is devoted to a short introduction to quasi boundary triples and their  
Weyl functions; the presentation is chosen such that the results can be applied directly
in the main part of this paper. For a more detailed exposition and proofs in a general scenario we refer 
to \cite{BL07,BL12}.
Throughout this abstract section $\mathcal{H}$ is always a complex Hilbert space with inner product $(\cdot, \cdot)_\mathcal{H}$;
if no confusion arises, we skip the index in the inner product. 

\begin{definition} \label{definition_boundary_triples}
  Let $S$ be a densely defined, closed, symmetric operator in $\mathcal{H}$ and assume that $T$ is a linear operator in 
  $\mathcal{H}$ such that $\overline{T} = S^*$. Moreover, let $\mathcal{G}$ be a complex Hilbert space and let
  $\Gamma_0, \Gamma_1: \dom T \rightarrow \mathcal{G}$ be linear mappings. Then $\{ \mathcal{G}, \Gamma_0, \Gamma_1 \}$
  is called a \emph{quasi boundary triple} for $T \subset S^*$ if the following conditions are fulfilled:
  \begin{itemize}
    \item[$\textup{(i)}$] For all $f, g \in \dom T$ the abstract Green's identity
    \begin{equation*} 
      (T f, g)_\mathcal{H} - (f, T g)_\mathcal{H} 
         = (\Gamma_1 f, \Gamma_0 g)_\mathcal{G} - (\Gamma_0 f, \Gamma_1 g)_\mathcal{G}
    \end{equation*}
    holds.
    \item[$\textup{(ii)}$] $(\Gamma_0, \Gamma_1)^\top: \dom T \rightarrow \mathcal{G} \times \mathcal{G}$
    has dense range.
    \item[$\textup{(iii)}$] The operator $A_0 := T \upharpoonright \ker \Gamma_0$ is self-adjoint in $\mathcal{H}$.
  \end{itemize}
\end{definition}

The concept of quasi boundary triples is a generalization of ordinary and generalized boundary triples; cf. \cite{BHS19,BGP08, DM91, DM95}.
We note that the operator $T$ in the above definition
is not unique if the dimension of $\mathcal{G}$ is infinite. Moreover, we remark that 
a quasi boundary triple exists if and only if $\dim \ker (S^* - i) = \dim \ker (S^* + i)$,
that is, if and only if $S$ admits self-adjoint extensions in $\mathcal{H}$.

Next, we introduce the $\gamma$-field and the Weyl function associated to a given quasi boundary triple.
Let $\{ \mathcal{G}, \Gamma_0, \Gamma_1 \}$ be a quasi boundary triple for $T \subset S^*$ and let 
$A_0 := T \upharpoonright \ker \Gamma_0$. The definition of the $\gamma$-field and the Weyl function is based on the 
direct sum decomposition
\begin{equation} \label{equation_dom_T_decomposition}
  \dom T = \dom A_0 \dot{+} \ker (T - \lambda) = \ker \Gamma_0 \dot{+} \ker (T - \lambda), \quad \lambda \in \rho(A_0),
\end{equation}
and is formally the same as in the case of ordinary boundary triples, see~\cite{BHS19,DM91}.
Note that~\eqref{equation_dom_T_decomposition} implies, in particular, that 
$\Gamma_0 \upharpoonright \ker (T-\lambda)$ is injective for $\lambda \in \rho(A_0)$.

\begin{definition} \label{definition_gamma_field_Weyl_function_abstract}
  Let $S$ be a densely defined, closed, symmetric operator in $\mathcal{H}$,
  let $T$ be a linear operator such that $\overline{T} = S^*$, and let $\{ \mathcal{G}, \Gamma_0, \Gamma_1 \}$
  be a quasi boundary triple for $T \subset S^*$.
  \begin{itemize}
    \item[$\textup{(i)}$] The {\em $\gamma$-field} associated to $\{ \mathcal{G}, \Gamma_0, \Gamma_1 \}$ is the mapping
    \begin{equation*}
      \rho(A_0) \ni \lambda \mapsto \gamma(\lambda) := \bigl(\Gamma_0 \upharpoonright \ker (T-\lambda)\bigr)^{-1}. 
    \end{equation*}
    \item[$\textup{(ii)}$] The {\em Weyl function} associated to $\{ \mathcal{G}, \Gamma_0, \Gamma_1 \}$ is the mapping
    \begin{equation*}
      \rho(A_0) \ni \lambda \mapsto M(\lambda) := \Gamma_1 \bigl(\Gamma_0 \upharpoonright \ker (T-\lambda)\bigr)^{-1}
          = \Gamma_1 \gamma(\lambda). 
    \end{equation*}
  \end{itemize}
\end{definition}

Let us now assume that $\{ \mathcal{G}, \Gamma_0, \Gamma_1 \}$ is a quasi boundary triple for 
$T \subset S^*$ and set $A_0 := T \upharpoonright \ker \Gamma_0$. In the following we collect several useful properties of the associated $\gamma$-field $\gamma$ and Weyl function $M$; for the proofs see for instance \cite[Proposition~2.6]{BL07} and \cite[Propositions~6.13 and~6.14]{BL12}.
First, for any $\lambda \in \rho(A_0)$ the mapping $\gamma(\lambda)$ is densely defined and
bounded from~$\mathcal{G}$ into $\mathcal{H}$ with $\dom \gamma(\lambda) = \ran \Gamma_0$. 
Using the abstract Green's identity it is not difficult to see that the adjoint $\gamma(\lambda)^*: \mathcal{H} \rightarrow \mathcal{G}$ is given by 
$\gamma(\lambda)^* = \Gamma_1 (A_0 - \overline{\lambda})^{-1}$. This implies, 
in particular, $\gamma(\lambda)^* \in \mathcal{B}(\mathcal{H}, \mathcal{G})$.
In a similar manner we have for any $\lambda \in \rho(A_0)$ that the mapping $M(\lambda)$ is densely defined in $\mathcal{G}$ with $\dom M(\lambda) = \ran \Gamma_0$  
and $\ran M(\lambda) \subset \ran \Gamma_1$. 
By definition we have $M(\lambda) \Gamma_0 f_\lambda = \Gamma_1 f_\lambda$ for $\lambda \in \rho(A_0)$ and $f_\lambda \in \ker (T-\lambda)$. 
Next, for any $\lambda, \mu \in \rho(A_0)$ and $\varphi \in \ran \Gamma_0$
the identity
\begin{equation*} 
  M(\lambda) \varphi = M(\mu)^* \varphi + (\lambda - \overline{\mu}) \gamma(\mu)^* \gamma(\lambda) \varphi
\end{equation*}
holds.
In particular, the operator $M(\lambda)$ is closable, $M(\lambda) \subset M(\overline{\lambda})^*$, and
$M(\lambda)$ is symmetric for $\lambda \in \rho(A_0) \cap \mathbb{R}$.

In the main part of this paper we will use quasi boundary triples to introduce special extensions of a symmetric operator $S$.
Let $\{ \mathcal{G}, \Gamma_0, \Gamma_1 \}$ be a quasi boundary triple for $T \subset S^*$ and let 
$\Theta$ be a symmetric operator in $\mathcal{G}$. Then we define the operator $A_\Theta$ acting in $\mathcal{H}$ by
\begin{equation} \label{def_A_theta_abstract}
  A_\Theta := T \upharpoonright \ker (\Gamma_1 - \Theta \Gamma_0).
\end{equation}
In other words, a vector $f \in \dom T$ belongs to $\dom A_\Theta$ if $\Gamma_0 f \in \dom \Theta$ and if it satisfies the abstract boundary condition
$\Gamma_1 f = \Theta \Gamma_0 f$. It follows immediately from the abstract Green's identity
that $A_\Theta$ is symmetric. Of course, one is typically interested in the self-adjointness
of $A_\Theta$. However, in general, for quasi boundary triples the self-adjointness of 
$\Theta$ in $\mathcal{G}$ does not necessarily imply that 
$A_\Theta$ is self-adjoint in $\mathcal{H}$.
Nevertheless the next theorem provides an explicit Krein-type resolvent formula which allows to deduce several properties of $A_\Theta$ from $\Theta$; for a proof of this result see for instance 
\cite[Theorem~2.8]{BL07}.

\begin{thm} \label{theorem_krein_abstract}
  Let $S$ be a densely defined, closed, symmetric operator in $\mathcal{H}$, 
  let $\{ \mathcal{G}, \Gamma_0, \Gamma_1 \}$ be a quasi boundary triple for $T \subset S^*$,
  set $A_0 := T \upharpoonright \ker \Gamma_0$,
  and let $\gamma$ and $M$ be the associated $\gamma$-field and Weyl function, respectively.
  Moreover, let $\Theta$ be a symmetric operator in $\mathcal{G}$ and let the associated operator 
  $A_\Theta$ be defined by~\eqref{def_A_theta_abstract}. Then the following statements hold
  for $\lambda \in \rho(A_0)$:
  \begin{itemize}
    \item[$\textup{(i)}$] $\lambda \in \sigma_{\textup{p}}(A_\Theta)$ if and only if 
    $0 \in \sigma_{\textup{p}} (\Theta - M(\lambda))$. Furthermore, one has
    \begin{equation*}
      \ker(A_\Theta - \lambda) = \bigl\{ \gamma(\lambda) \varphi: \varphi \in \ker (\Theta - M(\lambda)) \bigr\}.
    \end{equation*}
    \item[$\textup{(ii)}$] If $\lambda \notin \sigma_{\textup{p}}(A_\Theta)$ and $\gamma(\overline{\lambda})^* f \in \ran (\Theta- M(\lambda))$, then $f \in \ran (A_\Theta - \lambda)$.
    \item[$\textup{(iii)}$] If $\lambda \notin \sigma_{\textup{p}}(A_\Theta)$, then 
    \begin{equation*}
      (A_\Theta - \lambda)^{-1} f = (A_0 - \lambda)^{-1} f 
          + \gamma(\lambda) (\Theta - M(\lambda))^{-1} \gamma(\overline{\lambda})^* f
    \end{equation*}
    for all $f \in \ran (A_\Theta - \lambda)$.
  \end{itemize}
\end{thm}

We point out that assertion~(ii) in Theorem~\ref{theorem_krein_abstract} gives an efficient tool
to check the self-adjointness of $A_\Theta$. Since $A_\Theta$ is symmetric by Green's identity,
it suffices to show that $\ran (A_\Theta - \lambda_\pm) = \mathcal{H}$ for some $\lambda_\pm \in \mathbb{C}_\pm$.
According to Theorem~\ref{theorem_krein_abstract}~(ii) this is true 
if $\ran \gamma(\overline{\lambda_\pm})^* \subset \ran(\Theta - M(\lambda_\pm))$. Furthermore, if $\lambda \in \rho(A_\Theta)\not=\emptyset$ then the resolvent formula
in Theorem~\ref{theorem_krein_abstract}~(iii) holds for all $f\in\mathcal H$, that is, for $\lambda\in\rho(A_\Theta)\cap\rho(A_0)$ we have
 \begin{equation*}
      (A_\Theta - \lambda)^{-1}  = (A_0 - \lambda)^{-1}  
          + \gamma(\lambda) (\Theta - M(\lambda))^{-1} \gamma(\overline{\lambda})^*.
    \end{equation*}

Note that, if $\{ \mathcal{G}, \Gamma_0, \Gamma_1 \}$ is a quasi boundary triple for $T \subset S^*$, then Theorem~\ref{theorem_krein_abstract}
shows how the eigenvalues of self-adjoint extensions of $S$, that are contained in $\rho(A_0)$, can be characterized
by the Weyl function $M$. If the symmetric operator $S$ is simple, then {\em all} 
eigenvalues can be characterized with the help of $M$, in particular, also those that are embedded in $\sigma(A_0)$, compare~\cite[Corollary~3.4]{BR15_2}.
Note that there are also similar characterizations for the other types of the spectrum available in \cite{BR15_2},
but in our applications we restrict ourselves to find the eigenvalues.

\begin{prop} \label{proposition_Birman_Schwinger_simple}
  Let $S$ be a densely defined, closed, symmetric operator in $\mathcal{H}$, 
  let $\{ \mathcal{G}, \Gamma_0, \Gamma_1 \}$ be a quasi boundary triple for $T \subset S^*$,
  and let $M$ be the associated Weyl function.
  Moreover, let $\Theta$ be a bounded and self-adjoint operator in $\mathcal{G}$ and assume that the associated operator 
  $A_\Theta$ defined by ~\eqref{def_A_theta_abstract} is self-adjoint. Assume, in addition, that $S$ is simple.
  Then $\ran(M(\lambda) - \Theta)$ is independent of $\lambda \in \mathbb{C} \setminus \mathbb{R}$, and $\lambda \in \mathbb{R}$ is an eigenvalue of $A_\Theta$ if and only if there exists $\varphi \in \ran (M(\lambda+i \varepsilon) - \Theta)$
  such that
  \begin{equation*}
    \lim_{\varepsilon \searrow 0} i \varepsilon \big( M(\lambda + i \varepsilon) - \Theta \big)^{-1} \varphi \neq 0.
  \end{equation*}
\end{prop}
\begin{proof}
  Define the boundary mappings $\Gamma_0^\Theta, \Gamma_1^\Theta: \dom T \rightarrow \mathcal{G}$ by
  \begin{equation*} 
    \Gamma_0^\Theta f := \Gamma_1 f - \Theta \Gamma_0 f \quad \text{and} \quad \Gamma_1^\Theta f = - \Gamma_0 f,
    \quad f \in \dom T.
  \end{equation*}
  We claim that $\{ \mathcal{G}, \Gamma_0^\Theta, \Gamma_1^\Theta \}$ is a quasi boundary triple for $T \subset S^*$
  with the additional property $T \upharpoonright \ker \Gamma_0^\Theta = A_\Theta$.
  In fact, using that $\Theta$ is bounded and self-adjoint we deduce from the abstract Green's identity for 
  $\{ \mathcal{G}, \Gamma_0, \Gamma_1 \}$ and for $f, g \in \dom T$ that
  \begin{equation*} 
    \begin{split}
      (T f, g)_{\mathcal{H}} - (f, T g)_{\mathcal{H}}
        &= (\Gamma_1 f, \Gamma_0 g)_{\mathcal{G}} - (\Gamma_0 f, \Gamma_1 g)_{\mathcal{G}}
        - (\Theta \Gamma_0 f, \Gamma_0 g)_{\mathcal{G}} + (\Gamma_0 f, \Theta \Gamma_0 g)_{\mathcal{G}} \\
      &= \big(-\Gamma_0 f, (\Gamma_1 - \Theta \Gamma_0) g\big)_{\mathcal{G}}
          - \big((\Gamma_1 - \Theta \Gamma_0) f, -\Gamma_0 g\big)_{\mathcal{G}} \\
      &= \big(\Gamma_1^\Theta f, \Gamma_0^\Theta g\big)_{\mathcal{G}} 
          - \big(\Gamma_0^\Theta f, \Gamma_1^\Theta g\big)_{\mathcal{G}},
    \end{split}
  \end{equation*}
  and hence the abstract Green's identity holds also for the triple $\{ \mathcal{G}, \Gamma_0^\Theta, \Gamma_1^\Theta \}$.
  
  Next, the definition of $\Gamma_0^\Theta, \Gamma_1^\Theta$ can be written equivalently as 
  \begin{equation*}
    \begin{pmatrix} \Gamma_0^\Theta \\ \Gamma_1^\Theta \end{pmatrix} 
      = B \begin{pmatrix} \Gamma_0 \\ \Gamma_1 \end{pmatrix}, \qquad
      B := \begin{pmatrix} -\Theta & 1 \\ -1 & 0 \end{pmatrix}.
  \end{equation*}
  Since $\Theta$ is bounded, it follows that $B$ is boundedly invertible with
  \begin{equation*}
    B^{-1} = \begin{pmatrix} 0 & -1 \\ 1 & -\Theta \end{pmatrix}.
  \end{equation*}
  Since $\ran (\Gamma_0, \Gamma_1)$ is dense in 
  $\mathcal{G} \times \mathcal{G}$, also 
  $\ran(\Gamma_0^\Theta, \Gamma_1^\Theta)$ is dense. Finally, the restriction
  $T \upharpoonright \ker(\Gamma_0^\Theta) = T \upharpoonright \ker (\Gamma_1 - \Theta \Gamma_0) = A_\Theta$
  is self-adjoint by assumption. Therefore $\{ \mathcal{G}, \Gamma_0^\Theta, \Gamma_1^\Theta \}$
  is a quasi boundary triple for $S^*$.
  
  Next, we compute on $\mathbb{C} \setminus \mathbb{R}$ the Weyl function $M_\Theta$ corresponding to the triple 
  $\{ \mathcal{G}, \Gamma_0^\Theta, \Gamma_1^\Theta \}$. For a fixed $\lambda \in \mathbb{C} \setminus \mathbb{R}$
  this is the mapping which is determined uniquely by the relation
  $M_\Theta(\lambda) \Gamma_0^\Theta f_\lambda = \Gamma_1^\Theta f_\lambda$ 
  for $f_\lambda \in \ker (T - \lambda)$. For such an $f_\lambda$ we compute
  \begin{equation*} 
    \begin{split}
      \Gamma_0^\Theta f_\lambda &= ( \Gamma_1 - \Theta \Gamma_0 ) f_\lambda
         = \big( M(\lambda) - \Theta \big) \Gamma_0  f_\lambda
         = - \big( M(\lambda) - \Theta \big) \Gamma_1^\Theta  f_\lambda.
    \end{split}
  \end{equation*}
  Note that $M(\lambda) - \Theta$ is invertible by Theorem~\ref{theorem_krein_abstract},
  as otherwise the self-adjoint operator $A_\Theta$ would have the non-real eigenvalue $\lambda$.
  Thus, we conclude
  \begin{equation*}
    M_\Theta(\lambda) = -(M(\lambda) - \Theta )^{-1}.
  \end{equation*}
  In particular, this implies that $\dom M_\Theta(\lambda) = \ran (M(\lambda) - \Theta) = \ran \Gamma_0^\Theta$ is independent of $\lambda \in \rho(A_\Theta)$.
  
  After all these preparations the claim of the proposition follows from \cite[Corollary~3.4]{BR15_2}
  applied to the quasi boundary triple $\{ \mathcal{G}, \Gamma_0^\Theta, \Gamma_1^\Theta \}$, as $S$ is simple.
\end{proof}

Let $\{ \mathcal{G}, \Gamma_0, \Gamma_1 \}$ be a quasi boundary triple for $T \subset S^*$ and let $B$ be a symmetric operator in $\mathcal{G}$. In some applications it is more convenient to consider 
\begin{equation} \label{def_A_B_abstract}
  A_{[B]} := T \upharpoonright \ker (B \Gamma_1 - \Gamma_0)
\end{equation}
instead of~\eqref{def_A_theta_abstract}. Similarly as above, a symmetric operator $B$ leads to a symmetric extension 
$A_{[B]}$ of $S$, but $B$ being self-adjoint does not imply that also $A_{[B]}$ is self-adjoint. But we have the 
following counterpart of Theorem~\ref{theorem_krein_abstract}, which gives in item~(ii) similarly as above an efficient tool to check the self-adjointness of
$A_{[B]}$; cf. \cite[Theorem~2.6]{BLL13_2}.

\begin{thm} \label{theorem_krein_abstract_B}
  Let $S$ be a densely defined, closed, symmetric operator in $\mathcal{H}$, 
  let $\{ \mathcal{G}, \Gamma_0, \Gamma_1 \}$ be a quasi boundary triple for $T \subset S^*$,
  set $A_0 := T \upharpoonright \ker \Gamma_0$,
  and let $\gamma$ and $M$ be the associated $\gamma$-field and Weyl function, respectively.
  Moreover, let $B$ be a symmetric operator in $\mathcal{G}$ and let the associated operator 
  $A_{[B]}$ be defined by~\eqref{def_A_B_abstract}. Then the following statements hold
  for $\lambda \in \rho(A_0)$:
  \begin{itemize}
    \item[$\textup{(i)}$] $\lambda \in \sigma_{\textup{p}}(A_{[B]})$ if and only if 
    $1 \in \sigma_{\textup{p}} (B M(\lambda))$. Furthermore, one has
    \begin{equation*}
      \ker(A_{[B]} - \lambda) = \bigl\{ \gamma(\lambda) \varphi: \varphi \in \ker (I - B M(\lambda)) \bigr\}.
    \end{equation*}
    \item[$\textup{(ii)}$] If $\lambda \notin \sigma_{\textup{p}}(A_{[B]})$ and $B \gamma(\overline{\lambda})^* f \in \ran ( I - B M(\lambda))$, then $f \in \ran (A_{[B]} - \lambda)$.
    \item[$\textup{(iii)}$] If $\lambda \notin \sigma_{\textup{p}}(A_{[B]})$, then 
    \begin{equation*}
      (A_{[B]} - \lambda)^{-1} f = (A_0 - \lambda)^{-1} f 
          + \gamma(\lambda) (I - B M(\lambda))^{-1} B \gamma(\overline{\lambda})^* f
    \end{equation*}
    for all $f \in \ran (A_{[B]} - \lambda)$.
  \end{itemize}
\end{thm} 

Note that for $\lambda\in\rho(A_{[B]})\cap\rho(A_0)$ the resolvent formula in Theorem~\ref{theorem_krein_abstract_B}~(iii) reads as
\begin{equation*}
      (A_{[B]} - \lambda)^{-1}  = (A_0 - \lambda)^{-1}  
          + \gamma(\lambda) (I - B M(\lambda))^{-1} B \gamma(\overline{\lambda})^*.
    \end{equation*}

Finally, we state the counterpart of Proposition~\ref{proposition_Birman_Schwinger_simple} for extensions 
$A_{[B]}$ given by~\eqref{def_A_B_abstract} to detect all eigenvalues of $A_{[B]}$.

\begin{prop} \label{proposition_Birman_Schwinger_simple_B}
  Let $S$ be a densely defined, closed, symmetric operator in $\mathcal{H}$, 
  let $\{ \mathcal{G}, \Gamma_0, \Gamma_1 \}$ be a quasi boundary triple for $T \subset S^*$,
  and let $M$ be the associated Weyl function.
  Moreover, let $B$ be a bounded and self-adjoint operator in $\mathcal{G}$ and assume that the associated operator 
  $A_{[B]}$ defined by ~\eqref{def_A_B_abstract} is self-adjoint. Assume, in addition, that $S$ is simple.
  Then $\ran (I - B M(\lambda))$ is independent of $\lambda \in \mathbb{C} \setminus \mathbb{R}$, and $\lambda \in \mathbb{R}$ is an eigenvalue of $A_{[B]}$ if and only if there exists $\varphi \in \ran (I - B M(\lambda + i \varepsilon))$
  such that
  \begin{equation*}
    \lim_{\varepsilon \searrow 0} i \varepsilon M(\lambda + i \varepsilon) \big( I - B M(\lambda + i \varepsilon) \big)^{-1} \varphi \neq 0.
  \end{equation*}
\end{prop}
\begin{proof}
  The proof is very similar as the one of Proposition~\ref{proposition_Birman_Schwinger_simple}, so we only sketch the main differences here.
  Define the boundary mappings $\Gamma_0^{[B]}, \Gamma_1^{[B]}: \dom T \rightarrow \mathcal{G}$ by
  \begin{equation*} 
    \Gamma_0^{[B]} f := \Gamma_0 f - B \Gamma_1 f \quad \text{and} \quad \Gamma_1^{[B]} f = \Gamma_1 f,
    \quad f \in \dom T.
  \end{equation*}
  Then one verifies with the same arguments as in the proof of Proposition~\ref{proposition_Birman_Schwinger_simple} that $\{ \mathcal{G}, \Gamma_0^{[B]}, \Gamma_1^{[B]} \}$ is a quasi boundary triple for $T \subset S^*$ with $T \upharpoonright \Gamma_0^{[B]} = A_{[B]}$.
  
  If we are able to compute the Weyl function $M_{[B]}(\lambda)$ corresponding to the triple $\{ \mathcal{G}, \Gamma_0^{[B]}, \Gamma_1^{[B]} \}$ for $\lambda \in \mathbb{C} \setminus \mathbb{R}$, then we can apply again \cite[Corollary~3.4]{BR15_2} to characterize all eigenvalues of $A_{[B]}$. Let $\lambda \in \mathbb{C} \setminus \mathbb{R}$ and $f_\lambda \in \ker (T - \lambda)$ be fixed. Note that $M(\lambda)$ is invertible, as otherwise the symmetric operator $T \upharpoonright \ker \Gamma_1$ would have the non-real eigenvalue $\lambda$, cf. Theorem~\ref{theorem_krein_abstract}~(i).  
  Hence $M(\lambda) \Gamma_0 f_\lambda = \Gamma_1 f_\lambda$ implies $\Gamma_0 f_\lambda = M(\lambda)^{-1} \Gamma_1 f_\lambda$, which yields
  \begin{equation*} 
    \begin{split}
      \Gamma_0^{[B]} f_\lambda &= ( \Gamma_0 - B \Gamma_1 ) f_\lambda
         = ( I - B M(\lambda) ) M(\lambda)^{-1} \Gamma_1^{[B]} f_\lambda.
    \end{split}
  \end{equation*}
  Note that $I - B M(\lambda)$ is invertible by Theorem~\ref{theorem_krein_abstract_B},
  as otherwise the self-adjoint operator $A_{[B]}$ would have the non-real eigenvalue $\lambda$.
  Thus, we conclude
  \begin{equation*}
    M_{[B]}(\lambda) = M(\lambda) (I - B M(\lambda) )^{-1}.
  \end{equation*}
  This implies, in particular, that $\dom M_{[B]}(\lambda) = \ran (I - B M(\lambda)) = \ran \Gamma_0^{[B]}$ is independent of $\lambda \in \rho(A_{[B]})$.
  
  After all these preparations the claim of the proposition follows from \cite[Corollary~3.4]{BR15_2}
  applied to the quasi boundary triple $\{ \mathcal{G}, \Gamma_0^{[B]}, \Gamma_1^{[B]} \}$, as $S$ is simple.
\end{proof}

\section{The minimal, the maximal, the MIT bag, and some associated integral operators} \label{section_preliminary_ops}

In this section we provide some facts on Dirac operators and associated integral operators. First, we collect some properties of 
the minimal and the maximal realization of the Dirac operator on a domain $\Omega \subset \mathbb{R}^3$. 
Then we introduce and discuss the MIT bag operator, which is a distinguished self-adjoint realization of the Dirac operator 
in $\Omega$, and which serves as a reference operator later. Finally, we introduce several families of integral operators which will play 
a crucial role in Section~\ref{section_boundary_triples_domain} and Section~\ref{section_Dirac_domain} in the proofs of the main results of this paper.
Throughout this section let $\Omega$ be a $C^2$-domain in $\mathbb{R}^3$ with compact boundary, 
that is, $\Omega$ is either a bounded $C^2$-domain
or the complement of the closure of such a set. The unit normal vector field at $\partial \Omega$ pointing outwards $\Omega$ is denoted by $\nu$.

\subsection{The minimal and the maximal Dirac operator}

We are going to study the following two operators
acting in $L^2(\Omega; \mathbb{C}^4)$:
The {\it maximal Dirac operator}
\begin{equation} \label{def_maximal_op}
  \begin{split}
    T_{\textup{max}} f &= -i \alpha \cdot \nabla f + m \beta f, \\
    \quad \dom T_{\textup{max}} &= \big\{ f \in L^2(\Omega; \mathbb{C}^4): 
    \alpha \cdot \nabla f \in L^2(\Omega; \mathbb{C}^4) \big\},
  \end{split}
\end{equation}
where the derivatives are understood in the distributional sense,
and the {\it minimal Dirac operator} $T_{\textup{min}} = T_{\textup{max}} \upharpoonright H^1_0(\Omega; \mathbb{C}^4)$,
which is given in a  more explicit form by
\begin{equation} \label{def_minimal_op}
  T_{\textup{min}} f = -i \alpha \cdot \nabla f + m \beta f, \quad \dom T_{\textup{min}} = H^1_0(\Omega; \mathbb{R}^3).
\end{equation}
Some basic and well-known properties of $T_{\textup{min}}$ and $T_{\textup{max}}$ are collected in the following lemma; cf. 
\cite[Proposition~3.1]{BH17}, \cite[Lemma~2.1]{BFSB17_1}, and \cite[Proposition~2.10 and Proposition~2.12]{OV17}.

\begin{lem} \label{lemma_minimal_maximal_op}
  Let $\Omega \subset \mathbb{R}^3$ be a $C^2$-domain with compact boundary and
  let $T_{\textup{max}}$ and $T_{\textup{min}}$ be defined by \eqref{def_maximal_op} and 
  \eqref{def_minimal_op}, respectively. Then $T_{\textup{min}}$ is a densely defined, closed, symmetric operator 
  in $L^2(\Omega; \mathbb{C}^4)$ and we have
  $$T_{\textup{min}}^* = T_{\textup{max}}\qquad\text{and}\qquad T_{\textup{min}} = T_{\textup{max}}^*.$$
  Moreover, $C^\infty(\overline{\Omega}; \mathbb{C}^4)$ is dense in 
  $\dom T_{\textup{max}}$ with respect to the graph norm.
\end{lem}

Observe that for 
$f,g\in H^1(\Omega; \mathbb{C}^4)$ the integration by parts formula 
\begin{equation*}
\big( \alpha \cdot \nabla  f, g \big)_{\Omega}
          + \big( f,  \alpha \cdot \nabla  g \big)_{\Omega} 
      = \big( (\alpha \cdot \nu) f|_{\partial \Omega}, g|_{\partial \Omega}\big)_{\partial \Omega}
\end{equation*}
implies the identity
\begin{equation} \label{integration_by_parts}
 \big( (-i \alpha \cdot \nabla + m \beta) f, g \big)_\Omega - \big( f, (-i \alpha \cdot \nabla + m \beta)  g \big)_\Omega
 = \big( - i (\alpha \cdot \nu) f|_{\partial \Omega}, g|_{\partial \Omega}\big)_{\partial \Omega}.
\end{equation}

In the next proposition we verify that $T_\textup{min}$ is a simple symmetric operator, that is, there exists no nontrivial invariant 
subspace for $T_\textup{min}$ in $L^2(\Omega; \mathbb{C}^4)$ on which $T_\textup{min}$ reduces to a self-adjoint operator. 
The simplicity of $T_\textup{min}$ is essential in Proposition~\ref{proposition_basic_spectral_properties_bounded_domain} 
and Proposition~\ref{proposition_basic_spectral_properties_bounded_domain_A_omega}
for the characterization of eigenvalues of self-adjoint extensions of $T_\textup{min}$ which are embedded in the spectrum of 
the MIT bag operator.

\begin{prop} \label{proposition_T_min_simple}
  The operator $T_\textup{min}$ in \eqref{def_minimal_op} is simple.
\end{prop}
\begin{proof}
  Assume that $T_{\textup{min}} = T_1 \oplus T_2$, where 
  $T_j$ acts in an invariant subspace $\mathcal{H}_j \subset L^2(\Omega; \mathbb{C}^4)$ for $T_{\textup{min}}$, 
  $j \in \{ 1, 2\}$, and that $T_1 = T_1^*$. 
  We prove that $\mathcal{H}_1=\{0\}$. For that, note $(T_{\textup{min}})^2 = T_1^2 \oplus T_2^2$ and $T_1^2 = (T_1^2)^*$
  in $\mathcal{H}_1$ by the spectral theorem.
  Since $T_1^2$ is closed, we have $\overline{(T_{\textup{min}})^2} = T_1^2 \oplus \overline{T_2^2}$. 
  Let us show that 
  $\overline{(T_{\textup{min}})^2}$ is simple. We define the operator
  \begin{equation*}
    A^\Omega f = (-i \alpha \cdot \nabla + m \beta)^2 f = (-\Delta + m^2) f, \quad \dom A^\Omega = H^2(\Omega; \mathbb{C}^4),
  \end{equation*}
  and we claim that $A^\Omega \subset ( (T_{\textup{min}})^2 )^*$. In fact, consider arbitrary $f \in \dom A^\Omega$ and 
   let $g \in \dom (T_{\textup{min}})^2$. Then $g, (-i \alpha \cdot \nabla + m \beta) g \in H^1_0(\Omega; \mathbb{C}^4)$ and the identity~\eqref{integration_by_parts}
   shows
  \begin{equation*}
    \begin{split}
      \big( A^\Omega f, g \big)_\Omega &= \big( (-i \alpha \cdot \nabla + m \beta)^2 f, g \big)_\Omega \\
      &= \big( (-i \alpha \cdot \nabla + m \beta) f, (-i \alpha \cdot \nabla + m \beta) g \big)_\Omega  \\
      &= \big( f, (-i \alpha \cdot \nabla + m \beta)^2 g \big)_\Omega
          \\ &= \big( f, (T_{\textup{min}})^2 g \big)_\Omega,
    \end{split}
  \end{equation*}
  which implies $A^\Omega \subset ( (T_{\textup{min}})^2 )^*$.
  Now we use that  
  \begin{equation*} 
  L^2(\Omega; \mathbb{C}^4)=\overline{\text{span} \big\{ \ker(A^\Omega - \lambda) : \lambda \in \mathbb{C} \setminus \mathbb{R} \big\} };
\end{equation*}
  if $\Omega$ is bounded this is essentially a consequence of unique continuation (for details see \cite[Section~8.3]{BHS19})
  and if $\Omega$ is unbounded this fact can be found in \cite[Proposition~2.2]{BR15}.  As $A^\Omega \subset ( (T_{\textup{min}})^2 )^*$
  and $( (T_{\textup{min}})^2 )^*=(\overline{(T_{\textup{min}})^2})^*$ we have 
  $$\ker(A^\Omega - \lambda) \subset \ker\bigl(\bigl(\,\overline{(T_{\textup{min}})^2}\,\bigr)^*-\lambda\bigr),\quad \lambda \in \mathbb{C} \setminus \mathbb{R},
  $$ 
  and we conclude
  \begin{equation*} 
  L^2(\Omega; \mathbb{C}^4)=\overline{\text{span} \big\{ \ker\bigl(\bigl(\,\overline{(T_{\textup{min}})^2}\,\bigr)^* - \lambda\bigr) : 
  \lambda \in \mathbb{C} \setminus \mathbb{R} \big\} }.
\end{equation*}
  This implies that $\overline{(T_{\textup{min}})^2}$ is simple (see~\eqref{equation_check_simple}).
  Therefore, $\mathcal{H}_1=\{0\}$ and hence $T_{\textup{min}}$ is simple.
\end{proof}

\subsection{The MIT bag operator} \label{section_MIT}

In this subsection we discuss the MIT bag Dirac operator in $\Omega$ which will 
often play the role of a self-adjoint reference operator in this paper. 
The MIT bag operator is the partial differential operator in $L^2(\Omega; \mathbb{C}^4)$ defined by
\begin{equation} \label{def_MIT_op}
  \begin{split}
    T_{\textup{MIT}} f &= (-i \alpha \cdot \nabla + m \beta) f, \\
    \dom T_{\textup{MIT}} &= \bigl\{ f \in H^1(\Omega; \mathbb{C}^4): 
        f|_{\partial \Omega} = -i \beta (\alpha \cdot \nu) f|_{\partial \Omega} \bigr\}.
  \end{split}
\end{equation}
In the following proposition we summarize the basic properties of $T_{\textup{MIT}}$.
For some further results on $T_\text{MIT}$, as, e.g., 
symmetry relations of the spectrum or asymptotics of 
eigenvalues for large masses~$m$ we refer to \cite{ALTR17}. 
Moreover, we note that the orthogonal sum of the MIT bag operator in $\Omega$ and $\mathbb R^3\setminus\overline\Omega$ 
is a Dirac operator with a Lorentz scalar $\delta$-shell interaction, 
see Proposition~\ref{proposition_domain_confinement_A_omega}. Using this, one can show even 
some further properties of $T_\text{MIT}$; cf. \eqref{mitformel}.

\begin{prop} \label{proposition_MIT_bag_operator}
  The operator $T_{\textup{MIT}}$ defined by~\eqref{def_MIT_op}  is self-adjoint in $L^2(\Omega; \mathbb{C}^4)$ 
  and the following statements hold:
  \begin{itemize}
    \item[$\textup{(i)}$] $(-m, m) \subset \rho(T_{\textup{MIT}})$.
    \item[$\textup{(ii)}$] If $\Omega$ is bounded, then $\sigma(T_{\textup{MIT}}) = \sigma_{\textup{disc}}(T_{\textup{MIT}})=\sigma_{\rm p}(T_{\textup{MIT}})$.
    \item[$\textup{(iii)}$] If $\Omega$ is unbounded, 
    then $\sigma(T_{\textup{MIT}}) = \sigma_{\textup{ess}}(T_{\textup{MIT}}) = (-\infty, - m] \cup [m, \infty)$.
  \end{itemize}
\end{prop}
\begin{proof}
  First, the self-adjointness of $T_{\textup{MIT}}$ is shown in \cite[Theorem~3.2]{OV17}. 
  The proof of assertion~(i) follows similar considerations in \cite[Theorem~1.5]{ALTR17} for $C^3$-domains, but the 
  arguments are basically independent of the smoothness of $\partial \Omega$. Indeed, one can show for $f \in \dom T_{\textup{MIT}}$ 
  with the help of~\eqref{integration_by_parts} and~\eqref{eq_commutation} that
  \begin{equation*}
    \begin{split}
      \| T_{\textup{MIT}} f \|^2_\Omega 
          &= \big( (-i \alpha \cdot \nabla + m \beta) f, (-i \alpha \cdot \nabla + m \beta) f \big)_\Omega \\
      &= \| \alpha \cdot \nabla f \|^2_\Omega + m^2 \| f \|_\Omega^2 + ( -i \alpha \cdot \nabla f, m \beta f )_\Omega + (m \beta f, -i \alpha \cdot \nabla f)_\Omega \\
      &= \| \alpha \cdot \nabla f \|^2_\Omega + m^2 \| f \|_\Omega^2 + m \big( -i \beta (\alpha \cdot \nu) f|_{\partial \Omega}, f|_{\partial \Omega} \big)_{\partial \Omega} \\
      &=\| \alpha \cdot \nabla f \|_\Omega^2 + m^2 \| f \|_\Omega^2 + m \| f |_{\partial \Omega} \|_{\partial \Omega}^2
    \end{split}
  \end{equation*}
  holds, where the boundary condition for $f \in \dom T_\text{MIT}$ was used in the last step. Hence we have 
  $\| T_{\textup{MIT}} f \|_\Omega \geq m \| f \|_\Omega$ for all $f \in \dom T_{\textup{MIT}}$, which shows that  $\sigma(T_{\textup{MIT}}) \cap (-m, m) = \emptyset$. 
  
  To verify item~(ii) we note that $\dom T_{\textup{MIT}} \subset H^1(\Omega; \mathbb{C}^4)$ is compactly embedded 
  in $L^2(\Omega; \mathbb{C}^4)$, as $\Omega$ is a bounded $C^2$-domain. Hence $\sigma(T_{\textup{MIT}})$ is purely discrete.
  
  It remains to show point~(iii). By~(i) we have
  $\sigma(T_{\textup{MIT}}) \subset (-\infty, -m] \cup [m, \infty)$.
  To prove the other inclusion, fix some $\lambda \in (-\infty, -m] \cup [m, \infty)$,
  a number $R>0$ such that $\mathbb{R}^3 \setminus B(0, R) \subset \Omega$,
  a vector $\zeta \in \mathbb{C}^4$ such that $\left( \sqrt{\lambda^2 - m^2} \alpha_1 + m \beta + \lambda I_4 \right) \zeta \neq 0$, a cutoff-function $\chi \in C^\infty_0(\mathbb{R})$ with $\chi(r) = 1$ for $r < \frac{1}{2}$ and $\chi(r) = 0$ for $r > 1$ and set $x_n := (R + n^2, 0, 0)^\top$, $n \in \mathbb{N}$. Then we define the function $\psi_n^\lambda$ by
  \begin{equation*} 
    \psi_n^\lambda(x) := \frac{1}{n^{3/2}} \chi\left( \frac{1}{n} |x - x_n| \right) e^{i \sqrt{\lambda^2 - m^2} x\cdot e_1}
    \left( \sqrt{\lambda^2 - m^2} \alpha_1 + m \beta + \lambda I_4 \right)  \zeta.
  \end{equation*}
  Then one verifies in the same way as in \cite[Theorem~5.7]{BH17} that $\psi_n^\lambda \in \dom T_{\textup{MIT}}$, 
  that $\psi_n^\lambda$ converge weakly to zero, that
  \begin{equation*}
    \| \psi_n^\lambda \|_{\Omega} = \text{const.} > 0 \quad \text{and} \quad
    (T_{\textup{MIT}}^{\Omega} - \lambda) \psi_n^\lambda \rightarrow 0, 
        \quad \text{as } n \rightarrow \infty.
  \end{equation*}
  Thus $(\psi_n^\lambda)_n$ is a singular sequence for $T_{\textup{MIT}}$ and 
  $\lambda$,
  which shows $\lambda \in \sigma_{\textup{ess}}(T_{\textup{MIT}})$. This finishes the proof of this proposition.
\end{proof}

In a similar fashion as for the MIT bag model we also state some basic properties of another distinguished 
self-adjoint realization of the Dirac operator on $\Omega$. 
This operator has similar boundary conditions as $T_{\textup{MIT}}$,
but with opposite sign, and is given by
\begin{equation} \label{def_A_infty}
  \begin{split}
      T_{-\textup{MIT}} f &= (-i \alpha \cdot \nabla + m \beta) f, \\
      \dom T_{-\textup{MIT}} &= \bigl\{ f \in H^1(\Omega; \mathbb{C}^4): 
          f|_{\partial \Omega} = i \beta (\alpha \cdot \nu) f|_{\partial \Omega} \bigr\}.
    \end{split}
\end{equation}
The next proposition is the counterpart of Proposition~\ref{proposition_MIT_bag_operator} for $T_{-\textup{MIT}}$.
However, in contrast to Proposition~\ref{proposition_MIT_bag_operator}~(i) the interval $(-m, m)$ may also contain
spectrum; cf. Theorem~\ref{theorem_basic_spectral_properties_unbounded_domain} and Proposition~\ref{proposition_basic_spectral_properties_bounded_domain}.

\begin{prop} \label{prop_A_infty}
  The operator $T_{-\textup{MIT}}$ defined by~\eqref{def_A_infty}  is self-adjoint in $L^2(\Omega; \mathbb{C}^4)$ and the following statements hold:
  \begin{itemize}
    \item[$\textup{(i)}$] If $\Omega$ is bounded, then $\sigma(T_{-\textup{MIT}}) = \sigma_{\textup{disc}}(T_{-\textup{MIT}})=\sigma_{\rm p}(T_{-\textup{MIT}})$.
    \item[$\textup{(ii)}$] If $\Omega$ is unbounded, 
    then $(-\infty, - m] \cup [m, \infty) \subset \sigma_{\textup{ess}}(T_{-\textup{MIT}})$.
  \end{itemize}
\end{prop}

\begin{remark}
  We will show later in Theorem~\ref{theorem_basic_spectral_properties_unbounded_domain}~(i) that the inclusion in item~(ii) of the above proposition is in fact an equality, i.e.
  \begin{equation*}
    \sigma_{\textup{ess}}(T_{-\textup{MIT}}) = (-\infty, - m] \cup [m, \infty).
  \end{equation*}
  This holds, as the operator $T_{-\text{MIT}}$ corresponds to $ T \upharpoonright \ker \Gamma_1$ defined as in~\eqref{def_A_tau_domain} with the parameter $\vartheta = 0$. But for our next considerations the above inclusion is sufficient.
\end{remark}

\begin{proof}[Proof of Proposition~\ref{prop_A_infty}]
  Define the auxiliary operator
  \begin{equation*} 
  \begin{split}
      \widetilde{T} f &= (-i \alpha \cdot \nabla - m \beta) f, \quad
      \dom \widetilde{T} = \bigl\{ f \in H^1(\Omega; \mathbb{C}^4): 
          f|_{\partial \Omega} = i \beta (\alpha \cdot \nu) f|_{\partial \Omega} \bigr\},
    \end{split}
  \end{equation*}
  and consider the unitary and self-adjoint matrix
  \begin{equation*}
    \gamma_5 = \begin{pmatrix} 0 & I_2 \\ I_2 & 0 \end{pmatrix}.
  \end{equation*}
  We claim that
  \begin{equation} \label{equation_unitary_equivalence}
    T_\text{MIT} = \gamma_5 \widetilde{T} \gamma_5.
  \end{equation}
  To see this we note that  $\gamma_5 \beta = -\beta \gamma_5$ and $(\alpha \cdot x) \gamma_5 = \gamma_5 (\alpha \cdot x)$ for all $x \in \mathbb{R}^3$. This implies 
  \begin{equation*}
    (-i \alpha \cdot \nabla + m \beta)  f = \gamma_5 (-i \alpha \cdot \nabla - m \beta)\gamma_5 f,\quad f \in H^1(\Omega; \mathbb{C}^4).
  \end{equation*}
  Furthermore, we have $f \in \dom T_\text{MIT}$ if and only if $\gamma_5 f \in H^1(\Omega; \mathbb{C}^4)$ and
  \begin{equation*}
    i \beta (\alpha \cdot \nu) \gamma_5 f|_{\partial \Omega} = -\gamma_5 i \beta (\alpha \cdot \nu) f|_{\partial \Omega} = 
    \gamma_5  f|_{\partial \Omega},
  \end{equation*}
  so that $ f \in \dom T_\text{MIT}$ if and only if $\gamma_5 f\in \dom \widetilde{T}$.
  Hence, we have shown~\eqref{equation_unitary_equivalence}. In particular, this implies together with Proposition~\ref{proposition_MIT_bag_operator} 
  that $\widetilde{T}$ is self-adjoint. Since multiplication by $m \beta$ is a bounded and self-adjoint operator in $L^2(\Omega; \mathbb{C}^4)$ we 
  conclude that also $T_{-\text{MIT}} = \widetilde{T} + 2 m \beta$ is self-adjoint.
  
  Eventually, assertions~(i) and~(ii) can be shown in exactly the same way as Proposition~\ref{proposition_MIT_bag_operator}~(ii) and~(iii); 
  in particular, for the proof of item~(ii) the same singular sequence as in Proposition~\ref{proposition_MIT_bag_operator}~(iii) can be used.
\end{proof}

\subsection{Integral operators} \label{section_integral_op}

In this section we introduce several families of integral operators 
that will play an important role in the analysis 
of Dirac operators on domains. We also summarize some of their well-known properties.
Define for $\lambda \in \mathbb{C} \setminus ( (-\infty, -m] \cup [m, \infty) )$ the function
\begin{equation} \label{def_G_lambda}
  \begin{split}
    G_\lambda(x) = \left( \lambda I_4 + m \beta 
                 + \left( 1 - i \sqrt{\lambda^2 - m^2} |x| \right) \frac{i}{|x|^2} (\alpha \cdot x) \right)
                 \cdot \frac{e^{i \sqrt{\lambda^2 - m^2} |x|}}{4 \pi |x|}&.
  \end{split}
\end{equation}
Recall that $G_\lambda$ is the integral kernel
of the resolvent of the free Dirac operator in~$\mathbb{R}^3$, see \cite[Section~1.E]{T92}.
We introduce the operators
$\Phi_\lambda: L^2({\partial \Omega}; \mathbb{C}^4) \rightarrow L^2(\Omega; \mathbb{C}^4)$,
\begin{equation} \label{def_Phi_lambda}
  \Phi_\lambda \varphi(x) = \int_{{\partial \Omega}} G_\lambda(x - y) \varphi(y) \mathrm{d} \sigma(y), 
  \quad x \in \Omega,\, \varphi \in L^2({\partial \Omega}; \mathbb{C}^4),
\end{equation}
and $\mathcal{C}_\lambda: L^2({\partial \Omega}; \mathbb{C}^4) \rightarrow L^2({\partial \Omega}; \mathbb{C}^4)$
\begin{equation} \label{def_C_lambda}
  \mathcal{C}_\lambda \varphi(x) = \lim_{\varepsilon \searrow 0 }\int_{{\partial \Omega} \setminus B(x, \varepsilon)} 
      G_\lambda(x - y) \varphi(y) \mathrm{d} \sigma(y), 
  \quad x \in {\partial \Omega},\, \varphi \in L^2({\partial \Omega}; \mathbb{C}^4).
\end{equation}
It is well-known that $\Phi_\lambda$ and $\mathcal{C}_\lambda$ are bounded and everywhere defined and that
\begin{equation} \label{adjoint_C_lambda}
  \mathcal{C}_\lambda^* = \mathcal{C}_{\overline{\lambda}}
\end{equation}
holds;
cf. \cite[Lemmas~2.1 and~3.3]{AMV14} or \cite[Proposition~3.4]{BEHL18}. Furthermore, the adjoint of $\Phi_\lambda$ is given by
    $\Phi_\lambda^*: L^2(\Omega; \mathbb{C}^4) \rightarrow L^2({\partial \Omega}; \mathbb{C}^4),$
    \begin{equation} \label{def_Phi_lambda_adjoint}
      \begin{split}
        \Phi_\lambda^* f(x) &= \int_{\Omega} G_{\overline{\lambda}}(x-y) f(y) \mathrm{d} y,
        \quad x \in {\partial \Omega},\, f \in L^2(\Omega; \mathbb{C}^4),
      \end{split}
    \end{equation}
    and this operator is also bounded when viewed as an operator from $L^2(\Omega; \mathbb{C}^4)$ to $H^{1/2}(\partial \Omega; \mathbb{C}^4)$;
    cf. \cite[equation~(2.12) and the discussion below]{BEHL19_1}.
    Hence, we can define the anti-dual of $\Phi_\lambda^*$
    \begin{equation} \label{def_Phi_lambda_minus}
      \Phi_{\lambda, -1/2} := (\Phi_\lambda^*)': H^{-1/2}(\partial \Omega; \mathbb{C}^4) \rightarrow L^2(\Omega; \mathbb{C}^4).
    \end{equation}
    Since we have for $\varphi \in L^2(\partial \Omega; \mathbb{C}^4)$ and $f \in L^2(\Omega; \mathbb{C}^4)$
    \begin{equation*}
      \begin{split}
        (\Phi_{\lambda, -1/2} \varphi, f)_\Omega &= (\varphi, \Phi_\lambda^* f)_{H^{-1/2}(\partial \Omega; \mathbb{C}^4) \times H^{1/2}(\partial \Omega; \mathbb{C}^4)} \\
        &= (\varphi, \Phi_\lambda^* f)_{\partial \Omega} = (\Phi_{\lambda} \varphi, f)_\Omega,
      \end{split}
    \end{equation*}
    $\Phi_{\lambda, -1/2}$ is an extension of $\Phi_\lambda$.    
Some further properties of~$\Phi_\lambda$ are summarized in the following proposition. 

\begin{prop} \label{proposition_Phi_lambda}
  Let $\lambda \in \mathbb{C} \setminus ( (-\infty, -m] \cup [m, \infty))$ and let
  $\Phi_\lambda$  be the operator
  in~\eqref{def_Phi_lambda}. 
  Then the following statements hold:
  \begin{itemize}
    \item[$\textup{(i)}$] For any $s \in [-\frac{1}{2}, \frac{1}{2} ]$ the operator  $\Phi_\lambda$ gives rise 
    to a bounded and everywhere defined operator 
    $\Phi_{\lambda, s}: H^s(\partial \Omega; \mathbb{C}^4) \rightarrow H^{s+1/2}(\Omega; \mathbb{C}^4)$. 
    \item[$\textup{(ii)}$] $\ran \Phi_{\lambda, -1/2} = \ker (T_\textup{max} - \lambda)$.
  \end{itemize}
\end{prop}

\begin{remark} \label{remark_Phi_lambda}
  For $s \in (-\frac{1}{2}, \frac{1}{2}]$ the map $\Phi_{\lambda, s}$ is the restriction of $\Phi_{\lambda, -1/2}$ onto $H^s(\partial \Omega; \mathbb{C}^4)$, i.e. $\Phi_{\lambda, s} \varphi$ is for $\varphi \in H^s(\partial \Omega; \mathbb{C}^4)$ the uniquely determined $L^2$-function $\Phi_{\lambda, -1/2} \varphi$.
\end{remark}

\begin{proof}[Proof of Proposition~\ref{proposition_Phi_lambda}]
  The claim of statement~(i) for $s=-\frac{1}{2}$ follows from the definition of $\Phi_{\lambda, -1/2}$ in \eqref{def_Phi_lambda_minus} and it is contained for $s = \frac{1}{2} $ and $\Omega = \mathbb{R}^3 \setminus \Sigma$ for a $C^2$-surface 
  $\Sigma$ in \cite[Proposition~4.2]{BH17} (see also Remark~\ref{remark_Phi_lambda}); for that one has to note that the map $\gamma(\lambda)$ in \cite{BH17} coincides with $\Phi_{\lambda, 1/2}$. The claim for $\Omega$ follows from this by restriction  and
  for intermediate $s \in ( -\frac{1}{2}, \frac{1}{2})$ by an interpolation argument.
  Assertion~(ii) follows immediately from \cite[Propositions~4.4 and~2.6]{BH17} by noting that $\Phi_{\lambda, -1/2}$ coincides with $\widetilde{\gamma}(\lambda)$ in \cite{BH17}.
\end{proof}

In the next proposition we collect some additional properties of 
$\mathcal{C}_\lambda$ that will be useful in the sequel.

\begin{prop} \label{proposition_C_lambda}
Let $\lambda \in \mathbb{C} \setminus ( (-\infty, -m] \cup [m, \infty) )$ and let
  $\mathcal{C}_\lambda$ be the operator
  in~\eqref{def_C_lambda}. Then the following statements hold:
  \begin{itemize}
    \item[$\textup{(i)}$] For any $s \in [-\frac{1}{2}, \frac{1}{2} ]$ the operator 
    $\mathcal{C}_\lambda$ gives rise 
    to a bounded and everywhere defined operator 
    $\mathcal{C}_{\lambda, s}: H^s(\partial \Omega; \mathbb{C}^4) \rightarrow H^s(\partial \Omega; \mathbb{C}^4)$ and for the anti-dual of $\mathcal{C}_{\lambda, s}$ 
    one has $\mathcal{C}_{\lambda, s}' = \mathcal{C}_{\overline{\lambda}, -s}$.
    \item[$\textup{(ii)}$]  If $\lambda \in (-m, m)$ and $s \in [-\frac{1}{2}, \frac{1}{2} ]$, then the operator $\mathcal{C}_{\lambda, s}$ 
    is invertible and 
    \begin{equation}\label{guteformel}
    -4 \bigl(\mathcal{C}_{\lambda, s} (\alpha \cdot \nu)\bigr)^2 
    = -4 \bigl((\alpha \cdot \nu) \mathcal{C}_{\lambda, s}\bigr)^2 = I_4.
    \end{equation}
  \end{itemize}
\end{prop}
\begin{proof}
  First, by \cite[Proposition~4.2]{BH17} the restriction $\mathcal{C}_{\lambda,1/2} := \mathcal{C}_\lambda \upharpoonright H^{1/2}(\partial \Omega; \mathbb{C}^4)$ is bounded in $H^{1/2}(\partial \Omega; \mathbb{C}^4)$.
  Moreover, by \cite[Proposition~4.4]{BH17} the mapping $\mathcal{C}_{\lambda, 1/2}$ can be extended by continuity to a bounded operator $\mathcal{C}_{\lambda, -1/2}$ in $H^{-1/2}(\partial \Omega; \mathbb{C}^4)$
  and it is also shown there that
  \begin{equation*} 
    (\mathcal{C}_{\lambda, 1/2} \varphi, \psi)_{H^{1/2}(\partial \Omega; \mathbb{C}^4) \times H^{-1/2}(\partial \Omega; \mathbb{C}^4)} = (\varphi, \mathcal{C}_{\overline{\lambda}, -1/2} \psi)_{H^{1/2}(\partial \Omega; \mathbb{C}^4) \times H^{-1/2}(\partial \Omega; \mathbb{C}^4)}
  \end{equation*}
  holds for all $\varphi \in H^{1/2}(\partial \Omega; \mathbb{C}^4)$ and $\psi \in H^{-1/2}(\partial \Omega; \mathbb{C}^4)$;
  for this one just has to note that the operators $M(\lambda)$ and $\widetilde{M}(\lambda)$ in \cite{BH17} coincide with 
  $\mathcal{C}_{\lambda, 1/2}$ and $\mathcal{C}_{\lambda, -1/2}$, respectively. Hence, assertion~(i) holds for $s=\pm\frac{1}{2}$.  
  Now the continuity claim in item~(i) for the restriction $\mathcal{C}_{\lambda, s} := \mathcal{C}_{\lambda, -1/2} \upharpoonright H^s(\partial \Omega; \mathbb{C}^4)$ for intermediate $s \in ( -\frac{1}{2}, \frac{1}{2} )$ follows via interpolation. 
  Moreover, for $s \in (-\frac{1}{2}, 0)$ the map $\mathcal{C}_{\lambda, s}$ is the anti-dual of $\mathcal{C}_{\overline{\lambda}, -s}$. 
  To see the last claim, we note that for $\varphi \in L^2(\partial \Omega; \mathbb{C}^4)$ and $\psi \in H^{-s}(\partial \Omega; \mathbb{C}^4)$ one has
  \begin{equation*}
    \begin{split}
      (\mathcal{C}_{\lambda, s} \varphi, \psi)_{H^{s}(\partial \Omega; \mathbb{C}^4) \times H^{-s}(\partial \Omega; \mathbb{C}^4)} 
      &= (\mathcal{C}_{\lambda} \varphi, \psi)_{\partial \Omega}
      = (\varphi, \mathcal{C}_{\overline{\lambda}} \psi)_{\partial \Omega}\\
      &= (\varphi, \mathcal{C}_{\overline{\lambda}, -s} \psi)_{H^{s}(\partial \Omega; \mathbb{C}^4) \times H^{-s}(\partial \Omega; \mathbb{C}^4)},
    \end{split}
  \end{equation*}
  where~\eqref{adjoint_C_lambda} was used. By density this can be extended for all $\varphi \in H^s(\partial \Omega; \mathbb{C}^4)$, which implies that indeed $\mathcal{C}_{\lambda, s}' = \mathcal{C}_{\overline{\lambda}, -s}$.

  The identity \eqref{guteformel} in (ii) is shown for $\lambda = s = 0$ in \cite[Lemma~3.3]{AMV14} and can be shown for $\lambda \neq 0$ and $s=0$
  in a similar way. Clearly, this also implies the invertibility of $\mathcal{C}_{\lambda}$. For $s \in (0, \frac{1}{2}]$ the claim follows from 
  $\mathcal{C}_{\lambda, s} = \mathcal{C}_\lambda \upharpoonright H^s(\partial \Omega; \mathbb{C}^4)$. This and $\mathcal{C}_{\lambda, s}' = \mathcal{C}_{\overline{\lambda}, -s}$
  imply~\eqref{guteformel} also for $s \in [-\frac{1}{2}, 0)$. 
\end{proof}

For $\varphi \in H^{1/2}({\partial \Omega}; \mathbb{C}^4)$ 
the trace of $\Phi_{\lambda,1/2} \varphi$ and the function $\mathcal{C}_{\lambda,1/2} \varphi$ are closely related. The formula 
in the next lemma will be useful in the next sections.

\begin{lem} \label{lem_Phi_lambda}
  Let $\lambda \in \mathbb{C} \setminus ( (-\infty, -m] \cup [m, \infty))$ and let
  $\Phi_{\lambda,1/2}$ and $\mathcal{C}_{\lambda,1/2}$ be the operators 
  in~Proposition~\ref{proposition_Phi_lambda} and Proposition~\ref{proposition_C_lambda}, respectively. 
  Then one has 
    \begin{equation*}
      (\Phi_{\lambda,1/2} \varphi)\vert_{\partial\Omega} = \mathcal{C}_{\lambda,1/2} \varphi - \frac{i}{2} (\alpha \cdot \nu) \varphi,
      \qquad \varphi \in H^{1/2}({\partial \Omega}; \mathbb{C}^4).
    \end{equation*}
\end{lem}
\begin{proof}
 The formula is shown for $\lambda = 0$ in \cite[Lemma~3.3]{AMV14} in terms of the non-tangential limit of $\Phi_{\lambda,1/2} \varphi\in H^{1}(\Omega; \mathbb{C}^4)$.
 Since the non-tangential limit and the trace coincide (see, e.g. \cite[Lemma 3.1]{BGM20}) the claim follows for $\lambda = 0$; the claim for general $\lambda$
  can be deduced in the same way.
\end{proof}

In the following proposition we show that the commutator of the singular integral operator $\mathcal{C}_{\lambda}$
and a H\"older continuous function of order $a> 0$ is bounded from $L^2(\partial \Omega; \mathbb{C}^4)$ to $H^s(\partial \Omega; \mathbb{C}^4)$, $s \in [0, a)$, and hence 
compact in $H^s(\partial \Omega; \mathbb{C}^4)$ for $a>s\geq0$.
This has important consequences for the analysis of self-adjoint Dirac operators on domains 
and will be used in the proofs of many of the main results in this paper.
The proof of the next result relies on the properties of integral operators established in Appendix~\ref{appendix_commutator}.

\begin{prop} \label{proposition_commutator_C_lambda}
  Let $\lambda \in \mathbb{C} \setminus ( (-\infty, - m] \cup [m, \infty) )$, let $\mathcal{C}_\lambda$ be the operator
  in~\eqref{def_C_lambda}, and assume that  
  $\vartheta \in \textup{Lip}_a({\partial \Omega})$
  for some $a\in (0,1]$. Then the commutator 
  $$\varphi\mapsto\mathcal{C}_{\lambda}(\vartheta\varphi)
  -\vartheta\,\mathcal{C}_{\lambda}\varphi$$ 
  is bounded from $L^2(\partial \Omega; \mathbb{C}^4)$ to
  $H^{s}(\partial \Omega; \mathbb{C}^4)$ for any $s \in [0, a)$.
  In particular, the restriction $\mathcal{C}_{\lambda, s} \vartheta
  -\vartheta\,\mathcal{C}_{\lambda, s}$ is compact in $H^{s}(\partial \Omega; \mathbb{C}^4)$ for every $s \in [0, a)$.
\end{prop}

\begin{proof}
  To prove the claimed mapping properties we show that each component $k_{j l}$, $j, l\in\{1,\dots,4\}$, 
  of the matrix-valued integral kernel
  \begin{equation*}
    K(x, y) := G_\lambda(x-y) (\vartheta(y) - \vartheta(x))
  \end{equation*}
  of $\mathcal{C}_{\lambda} \vartheta - \vartheta \mathcal{C}_{\lambda}$ satisfies the estimates in~\eqref{kernel growth}. 
  As $a>s$ the claim of this proposition follows from Theorem~\ref{thm bounded} applied to the scalar integral operators with integral kernels $k_{j l}$, $j, l\in\{1,\dots,4\}$. 
  Since the embedding $\iota_s: H^s(\partial \Omega; \mathbb{C}^4) \rightarrow L^2(\partial \Omega; \mathbb{C}^4)$ is compact for $s\in (0,a)$, this implies then also the compactness of 
  $\mathcal{C}_{\lambda, s} \vartheta - \vartheta \mathcal{C}_{\lambda, s}$ in $H^s(\partial \Omega; \mathbb{C}^4)$. Indeed, for $s > 0$ it follows that $\mathcal{C}_{\lambda, s} \vartheta - \vartheta \mathcal{C}_{\lambda, s} = (\mathcal{C}_{\lambda} \vartheta - \vartheta \mathcal{C}_{\lambda}) \iota_s$ is compact in $H^s(\partial \Omega; \mathbb{C}^4)$, and for $s=0$ one can choose $r \in (0,a)$ and sees with the result of this proposition that $\mathcal{C}_{\lambda} \vartheta - \vartheta \mathcal{C}_{\lambda}: L^2(\partial \Omega; \mathbb{C}^4) \rightarrow H^r(\partial \Omega; \mathbb{C}^4)$ is bounded, and hence $\mathcal{C}_{\lambda} \vartheta - \vartheta \mathcal{C}_{\lambda} = \iota_r (\mathcal{C}_{\lambda} \vartheta - \vartheta \mathcal{C}_{\lambda})$ is compact in $L^2(\partial \Omega; \mathbb{C}^4)$.
  
  In the sequel we use the matrix norm 
  \begin{equation*}
    |M| := \max_{1 \leq k, l \leq 4} |m_{k l}|,\qquad M=(m_{kl})_{k,l=1}^4 \in \mathbb{C}^{4 \times 4}.
  \end{equation*}
  Throughout the proof of this proposition $C$ is a generic constant with different values on different places. First, due to the H\"older continuity of $\vartheta$ we conclude immediately from the definition of $G_\lambda$ in~\eqref{def_G_lambda} that
  \begin{equation*}
    |K(x, y)| = |G_\lambda(x-y)| \cdot |\vartheta(y) - \vartheta(x)| \leq C \frac{|y-x|^a }{|x - y|^2} = \frac{C}{|x - y|^{2-a}}
  \end{equation*}
  holds for all $x,y\in\partial\Omega$, $x \neq y$. Hence, the first estimate in \eqref{kernel growth} is satisfied. To show the second one, we take $x, y, z \in \partial \Omega$ with $|x - y| \leq \frac{1}{4} |x - z|$, write 
  \begin{equation} \label{equation_split_commutator}
    \begin{split}
      K(x, \,&z) - K(y, z) \\
      &= G_\lambda(x - z) (\vartheta(y) - \vartheta(x)) + \big( G_\lambda(x-z) - G_\lambda(y-z) \big) (\vartheta(z) - \vartheta(y)),
    \end{split}
  \end{equation}
  and show that each of the two terms on the right hand side of~\eqref{equation_split_commutator} fulfills this growth condition. Clearly, using the H\"older continuity of $\vartheta$ and the definition of $G_\lambda$ from~\eqref{def_G_lambda} we find first that
  \begin{equation} \label{equation_split_commutator1}
    \big| G_\lambda(x - z) (\vartheta(y) - \vartheta(x)) \big| \leq C \frac{|x - y|^a}{|x - z|^2}.
  \end{equation}
  To get an estimate for the second term in~\eqref{equation_split_commutator} we note first that the H\"older continuity of $\vartheta$, the triangle inequality, and $|x - y| \leq \frac{1}{4} |x - z|$ yield
  \begin{equation} \label{equation_split_commutator2}
    |\vartheta(z) - \vartheta(y)| \leq C |y - z|^a \leq C \big( |x - y| + |x -z| \big)^a \leq C |x - z|^a.
  \end{equation}
  In the next calculation we use for $\xi, \zeta \in \mathbb{R}^3 \setminus \{ 0 \}$ the notation 
  \begin{equation*}
    \nabla G_\lambda(\xi) \cdot \zeta := \partial_1 G_\lambda(\xi) \cdot \zeta_1 + \partial_2 G_\lambda(\xi) \cdot \zeta_2 + \partial_3 G_\lambda(\xi) \cdot \zeta_3.
  \end{equation*}
  Then we deduce with the main theorem of calculus applied to each entry of the matrix $G_\lambda(x-z) - G_\lambda(y-z)$
  \begin{equation} \label{equation_split_commutator3}
    \begin{split}
      \big| G_\lambda(x-z) - G_\lambda(y-z) \big| &= \left| \int_0^1 \frac{\text{d}}{\text{d} t} G_\lambda(t (x - y) + y - z) \text{d} t \right| \\
      &\leq \int_0^1 \big| \nabla G_\lambda(t (x - y) + y - z) \cdot (x - y) \big| \text{d} t.  \\
    \end{split}
  \end{equation}
  Using~\eqref{def_G_lambda} we find
  \begin{equation*}
    \begin{split}
      \partial_j G_\lambda(x) &= \Bigg[\!\left( \! \lambda I_4 + m \beta 
        + \left( 1 - i \sqrt{\lambda^2 - m^2} |x| \right) \frac{i(\alpha \cdot x )}{|x|^2} \!\right)\!
        \left( i \sqrt{\lambda^2 - m^2} - \frac{1}{|x|} \right) \partial_j |x| \\
      &\qquad +\left( 1 - i \sqrt{\lambda^2 - m^2} |x| \right) \frac{i}{|x|^2} \left( \alpha_j
        - \frac{2 \alpha\cdot x}{|x|} \partial_j |x| \right) \\
      &\qquad + \sqrt{\lambda^2 - m^2} \, \frac{\alpha \cdot x}{|x|^2}\, \partial_j |x| \Bigg]
        \frac{e^{i \sqrt{\lambda^2 - m^2} |x|}}{4 \pi |x|},
    \end{split}
  \end{equation*}
  which implies $|\partial_j G_\lambda(x)| \leq C |x|^{-3}$. Substituting this in~\eqref{equation_split_commutator3} yields
  \begin{equation} \label{equation_split_commutator4}
    \big| G_\lambda(x-z) - G_\lambda(y-z) \big| \leq C \int_0^1 |t (x - y) + y - z|^{-3} \text{d} t \cdot |x - y|. 
  \end{equation}
  Since $|x - y| \leq \frac{1}{4} |x - z|$, we can estimate for $t \in [0, 1]$
  \begin{equation*}
    |t (x - y) + y - z| \geq |y - z| - t |x - y| \geq |x - z| - (1 + t) |x - y| \geq \frac{1}{2} |x - z|
  \end{equation*}
  and 
  \begin{equation*}
    |x - y| = |x - y|^a \cdot |x - y|^{1-a} \leq C |x - y|^a \cdot |x - z|^{1-a}.
  \end{equation*}
  Using these two observations in~\eqref{equation_split_commutator4} we conclude
  \begin{equation*} 
    \big| G_\lambda(x-z) - G_\lambda(y-z) \big| \leq C \frac{|x - y|^a \cdot |x - z|^{1-a}}{|x - z|^3} = C \frac{|x - y|^a}{|x - z|^{2 + a}}. 
  \end{equation*}
  Together with~\eqref{equation_split_commutator2} this leads to
  \begin{equation*}
    \big| \big(G_\lambda(x-z) - G_\lambda(y-z) \big) (\vartheta(z) - \vartheta(y)) \big|
        \leq C \frac{|x - y|^a}{|x - z|^2}.
  \end{equation*}
  It follows now easily from \eqref{equation_split_commutator}, the triangle inequality, \eqref{equation_split_commutator1}, and the last estimate that the components 
  of $K$ also satisfy the second condition in~\eqref{kernel growth}. This completes the proof of this proposition.
\end{proof}

We now provide some useful anti-commutator properties of $\mathcal{C}_\lambda$ and the Dirac matrices.
These facts are also main ingredients to prove the self-adjointness of Dirac operators on domains later.

\begin{prop} \label{proposition_commutator}
  Let $\lambda \in \mathbb{C} \setminus ( (-\infty, - m] \cup [m, \infty) )$ and let $\mathcal{C}_\lambda$ be 
  the operator in~\eqref{def_C_lambda}.
  Then the following statements hold:
  \begin{itemize}
    \item[$\textup{(i)}$] The mapping 
    $\mathcal{A}_\lambda := \mathcal{C}_\lambda^2 - \frac{1}{4} I_4$
    can be extended to a bounded operator
    \begin{equation*}
      \widetilde{\mathcal{A}}_\lambda: H^{-1/2}({\partial \Omega}; \mathbb{C}^4) \rightarrow H^{1/2}({\partial \Omega}; \mathbb{C}^4).
    \end{equation*}
    \item[$\textup{(ii)}$] The anti-commutator 
    $\mathcal{B}_\lambda := \mathcal{C}_\lambda \beta + \beta \mathcal{C}_\lambda$
    can be extended to a bounded operator
    \begin{equation*}
      \widetilde{\mathcal{B}}_\lambda: H^{-1/2}({\partial \Omega}; \mathbb{C}^4) \rightarrow H^{1/2}({\partial \Omega}; \mathbb{C}^4).
    \end{equation*}
  \end{itemize}
  In particular, the restrictions 
  $$\mathcal{A}_{\lambda,1/2}=\mathcal{C}_{\lambda,1/2}^2 - \frac{1}{4} I_4\quad\text{and}\quad\mathcal{B}_{\lambda,1/2}=\mathcal{C}_{\lambda,1/2} \beta + \beta \mathcal{C}_{\lambda,1/2}$$ 
  onto $H^{1/2}(\partial \Omega; \mathbb{C}^4)$
  are both compact operators in $H^{1/2}(\partial \Omega; \mathbb{C}^4)$.
\end{prop}
\begin{proof}
  The proof of item~(i) can be  found in \cite[Proposition~4.4]{BH17} and \cite[Proposition~2.8]{OV17}.
  It remains to show statement~(ii). Using the anti-commutation relation~\eqref{eq_commutation}
  we see that $\mathcal{B}_\lambda$ is an integral operator with kernel
  \begin{equation*}
    b_\lambda(x-y) = 2 \left( \lambda \beta + m I_4 \right) \frac{e^{i\sqrt{\lambda^2 - m^2}|x-y|}}{4 \pi |x-y|}
  \end{equation*}
  and thus $\mathcal{B}_\lambda = 2 \big( \lambda \beta + m I_4 \big) \mathrm{SL}_{\lambda^2 - m^2}$,
  where $\mathrm{SL}_\mu$ denotes the single layer boundary integral operator for $-\Delta -\mu$, see, e.g., \cite[equation~(9.15)]{M00}.
  It is well known that $\mathrm{SL}_{\lambda^2 - m^2}$ gives rise to a bounded operator from 
  $H^{-1/2}({\partial \Omega}; \mathbb{C})$ to $H^{1/2}({\partial \Omega}; \mathbb{C})$, see for instance \cite[Theorem~6.11]{M00}. This yields assertion (ii).
\end{proof}

Next, we state a result on the invertibility of $\pm \frac{1}{2} \beta + \mathcal{C}_{\lambda, s}$ which will be used 
in the construction of the $\gamma$-field and the Weyl function associated to a quasi boundary triple for a
Dirac operator. To formulate the result we recall the definitions of $T_\text{MIT}$ and $T_{-\text{MIT}}$ from~\eqref{def_MIT_op} and~\eqref{def_A_infty}, 
respectively, and denote by $T_\text{MIT}^{\Omega^c}$ and $T_{-\textup{MIT}}^{\Omega^c}$ the Dirac operator in 
$L^2(\mathbb R^3\setminus \overline{\Omega}; \mathbb{C}^4)$ with the same boundary conditions on $\partial \Omega$ as $T_\text{MIT}$ and $T_{-\text{MIT}}$, respectively.

\begin{prop} \label{proposition_C_lambda_inverse}
  Let $s \in [-\frac{1}{2}, \frac{1}{2} ]$ and let $\mathcal{C}_{\lambda,s}$ be the operator in Proposition~\ref{proposition_C_lambda}.
  Then the following statements hold:
  \begin{itemize}
    \item[$\textup{(i)}$] For all $\lambda \not\in (-\infty, -m] \cup [m, \infty)$ 
    the operator $\frac{1}{2} \beta + \mathcal{C}_{\lambda, s}$ admits a bounded and everywhere defined inverse  in $H^s(\partial \Omega; \mathbb{C}^4)$.
    \item[$\textup{(ii)}$] For all $\lambda \notin (-\infty, -m] \cup[m, \infty) \cup \sigma_{\textup{p}}(T_{-\textup{MIT}}) 
    \cup \sigma_{\textup{p}}(T_{-\textup{MIT}}^{\Omega^c})$
    the
    operator $-\frac{1}{2} \beta + \mathcal{C}_{\lambda, s}$ admits a bounded and everywhere defined inverse 
    in $H^s(\partial \Omega; \mathbb{C}^4)$.
    \end{itemize}
\end{prop}
\begin{proof}
  (i) We prove the statement for $s=\frac{1}{2}$ and all $\lambda \in \mathbb{C} \setminus ( (-\infty, -m] \cup [m, \infty))$. For this it suffices to verify that 
   $\frac{1}{2} \beta + \mathcal{C}_{\lambda, 1/2}$ is bijective in $H^{1/2}(\partial \Omega; \mathbb{C}^4)$.
  Then the claim for $s=-\frac{1}{2}$ follows by duality, as $\mathcal{C}_{\overline{\lambda}, -1/2} = \mathcal{C}_{\lambda, 1/2}'$, and for 
  $s \in ( -\frac{1}{2}, \frac{1}{2})$ by interpolation.
  
Let us verify that  $\frac{1}{2} \beta + \mathcal{C}_{\lambda, 1/2}$ is injective. Indeed, assume that $\big(\frac{1}{2} \beta + \mathcal{C}_{\lambda, 1/2}\big) \varphi = 0$ for 
  some
  $\varphi \in H^{1/2}(\partial \Omega; \mathbb{C}^4)$, $\varphi\not=0$. 
  We claim that then $\lambda$ is an eigenvalue of the MIT bag operator in $\Omega$ or in $\Omega^c$, which is not possible by 
  Proposition~\ref{proposition_MIT_bag_operator}~(i). Define the functions 
  \begin{equation} \label{def_f_1}
    f_\Omega(x) := \Phi_{\lambda, 1/2} \varphi(x) = \int_{\partial \Omega} G_\lambda(x-y) \varphi(y) \textup{d} \sigma(y), \qquad x \in \Omega,
  \end{equation}
  and 
  \begin{equation} \label{def_f_2}
    f_{\Omega^c}(x) := \int_{\partial \Omega} G_\lambda(x-y) \varphi(y) \textup{d} \sigma(y), \qquad x \in \mathbb R^3\setminus\overline\Omega. 
  \end{equation}
  We claim first that $f_\Omega \in \dom T_\textup{MIT}$ and $f_{\Omega^c} \in \dom T_\textup{MIT}^{\Omega^c}$. This is shown for $f_\Omega$, 
  the proof for $f_{\Omega^c}$ is similar using that the unit normal vector field for $\Omega^c$ is $-\nu$. First, 
  since $\varphi \in H^{1/2}(\partial \Omega; \mathbb{C}^4)$, one has $f_\Omega \in H^1(\Omega; \mathbb{C}^4)$ by Proposition~\ref{proposition_Phi_lambda}~(i). 
  Moreover, according to  Lemma~\ref{lem_Phi_lambda} the trace of $f_\Omega$ is 
  \begin{equation*}
    f_\Omega|_{\partial \Omega} = -\frac{i}{2} (\alpha \cdot \nu) \varphi + \mathcal{C}_{\lambda, 1/2} \varphi,
  \end{equation*}
   and, because of $(\alpha \cdot \nu)^2  = I_4$ (which follows from~\eqref{eq_commutation}), it satisfies
  \begin{equation*}
    \begin{split}
      \big(I_4 + i \beta (\alpha \cdot \nu)\big) f_\Omega|_{\partial \Omega} 
        &= \big(I_4 + i \beta (\alpha \cdot \nu)\big) \left( -\frac{i}{2} (\alpha \cdot \nu) \varphi + \mathcal{C}_{\lambda, 1/2} \varphi \right) \\
        &= \left( \frac{1}{2} \beta + \mathcal{C}_{\lambda, 1/2} \right) \varphi + i \beta (\alpha \cdot \nu) \left( \frac{1}{2} \beta + \mathcal{C}_{\lambda, 1/2} \right) \varphi\\
        &= 0,
    \end{split}
  \end{equation*}
  i.e. $f_\Omega \in \dom T_\textup{MIT}$.
  Since the unit normal vector field of $\Omega^c$ is $-\nu$, one concludes from Proposition~\ref{proposition_Phi_lambda}~(iv) that the trace of $f_{\Omega^c}$ is
  \begin{equation*}
    f_{\Omega^c}|_{\partial \Omega} = \frac{i}{2} (\alpha \cdot \nu) \varphi + \mathcal{C}_{\lambda, 1/2} \varphi,
  \end{equation*}
  and hence,
  \begin{equation*}
    f_\Omega|_{\partial \Omega} - f_{\Omega^c}|_{\partial \Omega} = -i (\alpha \cdot \nu) \varphi \neq 0,
  \end{equation*}
  which shows that at least one of the functions $f_\Omega$ and $f_{\Omega^c}$ is not trivial. From Proposition~\ref{proposition_Phi_lambda}~(ii) 
  we get $(T_\textup{MIT} - \lambda) f_\Omega = 0$ and in the same way one concludes 
  $(T_\textup{MIT}^{\Omega^c} - \lambda) f_{\Omega^c} = 0$, i.e. $\lambda \in \mathbb{C} \setminus ( (-\infty, -m] \cup [m, \infty) )$ 
  is an eigenvalue of $T_\textup{MIT}$ or $T_\textup{MIT}^{\Omega^c}$, which leads to a contradiction. 
  Therefore, $\frac{1}{2} \beta + \mathcal{C}_{\lambda, 1/2}$ is injective.
  
  It remains to prove that $\frac{1}{2} \beta + \mathcal{C}_{\lambda, 1/2}$ is surjective. For that we note first that
   \begin{equation*} 
      \ran \big( \tfrac{1}{2} \beta + \mathcal{C}_{\lambda, 1/2} \big) \supset \ran \big( \tfrac{1}{2} \beta + \mathcal{C}_{\lambda, 1/2} \big)^2 
  \end{equation*}
  and from Proposition~\ref{proposition_commutator} we obtain
  \begin{equation*}\begin{split}
  \left( \frac{1}{2} \beta + \mathcal{C}_{\lambda, 1/2} \right)^2&=
  \frac{1}{2} I_4 + \frac{1}{2} \big(  \mathcal{C}_{\lambda, 1/2} \beta + \beta \mathcal{C}_{\lambda, 1/2}  \big) + (\mathcal{C}_{\lambda, 1/2})^2-\frac{1}{4} I_4\\
  &=\frac{1}{2} I_4 + \frac{1}{2}\mathcal{B}_{\lambda,1/2}+ \mathcal{A}_{\lambda,1/2},
  \end{split}\end{equation*}
  where $\frac{1}{2}\mathcal{B}_{\lambda,1/2}+ \mathcal{A}_{\lambda,1/2}$ is a compact operator in $H^{1/2}(\partial \Omega; \mathbb{C}^4)$. Therefore, 
  as the operator $\big( \tfrac{1}{2} \beta + \mathcal{C}_{\lambda, 1/2} \big)^2$ is injective (since $\tfrac{1}{2} \beta + \mathcal{C}_{\lambda, 1/2}$ is injective)
  Fredholm's alternative shows that  $\big( \tfrac{1}{2} \beta + \mathcal{C}_{\lambda, 1/2} \big)^2$, and hence also $\tfrac{1}{2} \beta + \mathcal{C}_{\lambda, 1/2}$, is surjective. This finishes the proof of item~(i).

  The proof of the item (ii) follows the lines of assertion (i), one just has to note that for $\varphi \in \ker \big( -\frac{1}{2} \beta + \mathcal{C}_{\lambda, 1/2} \big)$ the functions $f_\Omega$ and $f_{\Omega^c}$ defined by~\eqref{def_f_1} and~\eqref{def_f_2} belong to $\dom T_{-\textup{MIT}}$ and $\dom T_{-\textup{MIT}}^{\Omega^c}$, respectively.
\end{proof}

We now discuss that the derivatives of the integral operators $\Phi_\lambda$ and $\mathcal{C}_\lambda$
belong to certain (weak) Schatten-von Neumann ideals. For that we
use the following result on operators with range in the Sobolev space $H^s({\partial \Omega}; \mathbb{C})$. 
Its proof follows word-by-word the one of \cite[Proposition~2.4]{BEHL19_2}, hence we omit it here.

\begin{prop} \label{proposition_singular_values}
  Let $l \in \mathbb{N}$, let ${\partial \Omega} \subset \mathbb{R}^3$ be the boundary of a compact $C^l$-smooth domain, 
  and let $k \in \{ 1, \dots, 2 l - 1\}$. Let $\mathcal{H}$ be a separable Hilbert space and assume that 
  $A: \mathcal{H} \rightarrow L^2({\partial \Omega}; \mathbb{C}^4)$ is continuous with $\ran A \subset H^{k/2}({\partial \Omega}; \mathbb{C}^4)$.
  Then $A \in \mathfrak{S}_{4/k, \infty}\big(\mathcal{H}, L^2({\partial \Omega}; \mathbb{C}^4) \big)$.
\end{prop}

With the help of Proposition~\ref{proposition_singular_values} one can show in exactly the same way as in
\cite[Lemma~4.5]{BEHL18} the following result; note that the operators $\gamma(\lambda)$ and $M(\lambda)$ in \cite{BEHL18} coincide 
with $\Phi_\lambda$ and $\mathcal{C}_\lambda$, respectively.

\begin{lem} \label{lemma_Schatten_von_Neumann}
  Let $l \in \mathbb{N}$ with $l \geq 2$ and let ${\partial \Omega} \subset \mathbb{R}^3$ be the boundary of a compact $C^l$-smooth domain.
  Moreover, let $\lambda \in \mathbb{C} \setminus ( (-\infty, -m] \cup [m, \infty))$ and let $\Phi_\lambda$ and $\mathcal{C}_\lambda$ be the operators 
  in \eqref{def_Phi_lambda} and
  \eqref{def_C_lambda}, respectively. 
  Then the following statements hold:
  \begin{itemize}
    \item[$\textup{(i)}$] The operator-valued functions $\lambda \mapsto \Phi_\lambda$ and $\lambda \mapsto \Phi_{\overline{\lambda}}^*$ 
    are holomorphic and, for any $k \in \{ 0, 1, \dots, l-1 \}$, 
    \begin{equation*}
      \frac{\textup{d}^k}{\textup{d} \lambda^k} \Phi_\lambda \in \mathfrak{S}_{4/(2k+1), \infty}\big( L^2(\partial \Omega; \mathbb{C}^4), L^2(\Omega; \mathbb{C}^4) \big) 
    \end{equation*}
    and
    \begin{equation*}
      \frac{\textup{d}^k}{\textup{d} \lambda^k} \Phi_{\overline{\lambda}}^* \in \mathfrak{S}_{4/(2k+1), \infty}\big( L^2(\Omega; \mathbb{C}^4), L^2(\partial \Omega; \mathbb{C}^4) \big).
    \end{equation*}
    In particular, $\Phi_\lambda$ and $\Phi_\lambda^*$ are compact.
    \item[$\textup{(ii)}$] The mapping $\lambda \mapsto \mathcal{C}_\lambda$  
    is holomorphic, $\frac{\textup{d}}{\textup{d} \lambda} \mathcal{C}_\lambda = \Phi_{\overline{\lambda}}^* \Phi_\lambda$
    and, for any integer $k \in \{ 1, \dots, l\}$,
    \begin{equation*}
      \frac{\textup{d}^k}{\textup{d} \lambda^k} \mathcal{C}_\lambda \in \mathfrak{S}_{2/k, \infty}\big( L^2(\partial \Omega; \mathbb{C}^4) \big).
    \end{equation*}
  \end{itemize}
\end{lem}

\section{A quasi boundary triple for Dirac operators on domains} \label{section_boundary_triples_domain}

In this section we introduce a quasi boundary triple which is useful to define self-adjoint Dirac operators on domains 
via suitable boundary conditions on $\partial \Omega$.
Throughout this section let $\Omega$ be a bounded or unbounded domain in $\mathbb{R}^3$ with a compact $C^2$-smooth boundary. 
As before,
we denote the normal vector field at $\partial \Omega$ pointing outwards of $\Omega$ by $\nu$. In the following the operators
\begin{equation}\label{def_P_pm}
 P_\pm: L^2(\partial \Omega; \mathbb{C}^4)\rightarrow L^2(\partial \Omega; \mathbb{C}^4),\quad \varphi\mapsto \frac{1}{2} \big( I_4 \pm i \beta (\alpha \cdot \nu)\big)\varphi,
\end{equation}
will play an important role. The relation $P_- = I - P_+$ is clear. Furthermore, using the anti-commutation relation~\eqref{eq_commutation} 
it is easy to see that $P_\pm^2=P_\pm$ and
$P_+ P_- = P_- P_+ = 0$, that is, $P_\pm$ are orthogonal projections in $L^2(\partial \Omega; \mathbb{C}^4)$. We also note that~\eqref{eq_commutation} implies
 \begin{equation} \label{beta_P_pm}
   P_+\beta=\beta P_-.
\end{equation}
We shall make use of the spaces
\begin{equation} \label{def_G_Omega}
  \mathcal{G}_\Omega^s := P_+\bigl(H^s(\partial \Omega; \mathbb{C}^4)\bigr),\qquad s \in \bigl[0, \tfrac{1}{2} \bigr].
\end{equation}
For convenience of notation, we simply write $\mathcal{G}_\Omega := \mathcal{G}_\Omega^0$.
As $P_+$ is an orthogonal projection in $L^2(\partial \Omega; \mathbb{C}^4)$ the space $\mathcal{G}_\Omega$ is a Hilbert space.
Moreover, since $\partial \Omega$ is $C^2$-smooth the normal vector field $\nu$ is Lipschitz continuous and hence it follows from Lemma~\ref{lemma_mult_op} 
that $\mathcal{G}_\Omega^s \subset H^s(\partial \Omega; \mathbb{C}^4)$ for $s\in [0, \frac{1}{2}]$. Furthermore, 
it is not difficult to check that $\mathcal{G}_\Omega^s$ is a closed subspace of $H^s(\partial \Omega; \mathbb{C}^4)$ for $s\in [0, \frac{1}{2}]$.
In the sequel the spaces $\mathcal{G}_\Omega^s$ are always equipped with the norm of $H^s(\partial \Omega; \mathbb{C}^4)$.

Next, we define the operator $T$ in $L^2(\Omega; \mathbb{C}^4)$ by
\begin{equation} \label{def_T_domain}
  \begin{split}
    T f = (-i \alpha \cdot \nabla + m \beta) f, \qquad
    \dom T = H^1(\Omega; \mathbb{C}^4),
  \end{split}
\end{equation}
and the mappings $\Gamma_0, \Gamma_1: \dom T \rightarrow \mathcal{G}_\Omega$ acting as 
\begin{equation} \label{def_Gamma_0_Gamma_1_domain}
  \Gamma_0 f = P_+ f|_{\partial \Omega}  \quad \text{and} \quad
  \Gamma_1 f = P_+ \beta f|_{\partial \Omega}, \quad f \in \dom T.
\end{equation}
In the following theorem we show that $\{ \mathcal{G}_\Omega, \Gamma_0, \Gamma_1 \}$
is a quasi boundary triple and that $\overline{T}$ coincides with the maximal Dirac operator
$T_{\textup{max}}$ from \eqref{def_maximal_op}.
Moreover, it turns out that the reference operator $T \upharpoonright \ker \Gamma_0$
is the MIT bag operator studied in Section~\ref{section_MIT}.

\begin{thm} \label{theorem_quasi_triple_domain}
  Let $T_{\textup{min}}$ be the closed symmetric operator from~\eqref{def_minimal_op}, 
  let $\mathcal{G}_\Omega$ be given by~\eqref{def_G_Omega}, and
  let $T,\, \Gamma_0$, and $\Gamma_1$ be as in~\eqref{def_T_domain}
  and~\eqref{def_Gamma_0_Gamma_1_domain}, respectively.
  Then 
  $\overline{T} = T_{\textup{min}}^*= T_{\textup{max}}$
  and $\{ \mathcal{G}_\Omega, \Gamma_0, \Gamma_1 \}$ is a quasi boundary triple for 
  $T \subset T_{\textup{max}}$ such that 
  \begin{equation}\label{aberjadoch}
  T_{\textup{MIT}}=T \upharpoonright \ker \Gamma_0\quad\text{and}\quad T_{-\textup{MIT}}=T \upharpoonright \ker \Gamma_1.
  \end{equation}
  Moreover, 
  \begin{equation} \label{ran_Gamma_domain}
    \ran (\Gamma_0 \upharpoonright \ker \Gamma_1) 
        = \ran (\Gamma_1 \upharpoonright \ker \Gamma_0) 
        = \mathcal{G}_\Omega^{1/2},
  \end{equation}
  and hence, in particular, $\ran (\Gamma_0, \Gamma_1)^\top 
      = \mathcal{G}_\Omega^{1/2} \times \mathcal{G}_\Omega^{1/2}$.
\end{thm}
\begin{proof}
  First, we have   
  $T_{\textup{min}}^* = T_{\textup{max}}$ by Lemma~\ref{lemma_minimal_maximal_op}. Furthermore, Lemma~\ref{lemma_minimal_maximal_op} also implies that the closure of $T$ coincides with
  $T_{\textup{max}}$ , as
  $C^\infty(\overline{\Omega};\mathbb{C}^4) \subset  \dom T$ is dense in $\dom T_{\textup{max}}$ equipped with the graph norm.
  
  Now we verify that the abstract Green's identity is valid. For this consider
  $f, g \in \dom T = H^1(\Omega; \mathbb{C}^4)$.
  Then \eqref{integration_by_parts} and the self-adjointness of $\alpha \cdot \nu$ yield
  \begin{equation*}
    \begin{split}
       ( T f, g )_{\Omega} - (f, T g)_\Omega
          &= \big( (-i \alpha \cdot \nabla + m \beta) f, g \big)_{\Omega}
          - \big( f, (-i \alpha \cdot \nabla + m \beta) g \big)_{\Omega} \\
      &= \big( - i (\alpha \cdot \nu) f|_{\partial \Omega}, g|_{\partial \Omega}\big)_{\partial \Omega} \\
      &= \frac{1}{2} \big( -i (\alpha \cdot \nu) f|_{\partial \Omega}, g|_{\partial \Omega}\big)_{\partial \Omega}
        - \frac{1}{2} \big( f|_{\partial \Omega}, -i (\alpha \cdot \nu) g|_{\partial \Omega}\big)_{\partial \Omega}.
    \end{split}
  \end{equation*}
  Using that $\beta$ is unitary and self-adjoint 
  we see that the last expression is equal to
  \begin{equation*}
    \begin{split}
      \frac{1}{2} \big(& -i \beta (\alpha \cdot \nu) f|_{\partial \Omega}, 
          \beta g|_{\partial \Omega}\big)_{\partial \Omega}
          - \frac{1}{2} \big( \beta f|_{\partial \Omega}, 
          -i \beta (\alpha \cdot \nu) g|_{\partial \Omega}\big)_{\partial \Omega} \\
      &= \frac{1}{2} \big(  \beta f|_{\partial \Omega}, 
          g|_{\partial \Omega} + i \beta (\alpha \cdot \nu) g|_{\partial \Omega}\big)_{\partial \Omega} 
       - \frac{1}{2} \big( f|_{\partial \Omega} + i \beta (\alpha \cdot \nu) f|_{\partial \Omega}, 
          \beta g|_{\partial \Omega}\big)_{\partial \Omega} \\
      &= ( \beta f|_{\partial \Omega}, P_+ g|_{\partial \Omega} )_{\partial \Omega}
          - ( P_+ f|_{\partial \Omega}, \beta g|_{\partial \Omega})_{\partial \Omega}.
    \end{split}
  \end{equation*}
  Since $P_+$ is an orthogonal projection in $L^2(\partial \Omega; \mathbb{C}^4)$ we conclude
  \begin{equation*}
    \begin{split}
      (T f, g)_\Omega - (f, T g)_\Omega 
          &= (P_+ \beta f|_{\partial \Omega}, P_+ g|_{\partial \Omega} )_{\partial \Omega}
          - (P_+ f|_{\partial \Omega}, P_+ \beta g|_{\partial \Omega})_{\partial \Omega} \\
      &=  ( \Gamma_1 f, \Gamma_0 g )_{\partial \Omega}
          - ( \Gamma_0 f, \Gamma_1 g)_{\partial \Omega},
    \end{split}
  \end{equation*}
  which is the abstract Green's identity.

  Next, we check the range property~\eqref{ran_Gamma_domain}. 
  Clearly, by the definition of $\Gamma_0$ and $\Gamma_1$, $\dom \Gamma_0 = \dom \Gamma_1 = H^1(\Omega; \mathbb{C}^4)$, and by
  standard properties of the trace map and Lemma~\ref{lemma_mult_op} 
  one has $\ran \Gamma_0\subset \mathcal{G}_\Omega^{1/2}$ and $\ran \Gamma_1 \subset \mathcal{G}_\Omega^{1/2}$, and hence also
  \begin{equation}\label{erstmaldas}
    \ran (\Gamma_0 \upharpoonright \ker \Gamma_1)\subset \mathcal{G}_\Omega^{1/2}\quad\text{and}\quad
         \ran (\Gamma_1 \upharpoonright \ker \Gamma_0) \subset \mathcal{G}_\Omega^{1/2}.
  \end{equation}
  To prove $\mathcal{G}_\Omega^{1/2} \subset \ran(\Gamma_0 \upharpoonright \ker \Gamma_1)$
  let $\varphi \in \mathcal{G}_\Omega^{1/2}$ and choose 
  $f \in H^1(\Omega; \mathbb{C}^4)$ such that $f|_{\partial \Omega} = \varphi$.
  Since $\varphi \in \mathcal{G}_\Omega$ we have $P_+ \varphi=\varphi$ and $P_-\varphi=0$. 
  Hence,  
  \begin{equation*}
    \Gamma_0 f = P_+ f|_{\partial \Omega} = P_+ \varphi = \varphi
  \end{equation*}
  and using \eqref{beta_P_pm} we obtain
  \begin{equation*}
    \Gamma_1 f = P_+ \beta f|_{\partial \Omega} = P_+ \beta \varphi
    =   \beta P_- \varphi = 0,
  \end{equation*}
  that is, $\varphi \in \ran(\Gamma_0 \upharpoonright \ker \Gamma_1)$.
  To prove $\mathcal{G}_\Omega^{1/2} \subset \ran(\Gamma_1 \upharpoonright \ker \Gamma_0)$
  let $\psi \in \mathcal{G}_\Omega^{1/2}$ and choose $g \in H^1(\Omega; \mathbb{C}^4)$
  such that $g|_{\partial \Omega} = \beta \psi$. Since $\psi \in \mathcal{G}_\Omega$ we have $P_+ \psi=\psi$ and $P_- \psi=0$.
  Hence,  using \eqref{beta_P_pm} we obtain
  \begin{equation*}
    \Gamma_0 g =  P_+ g|_{\partial \Omega} = P_+ \beta \psi = \beta P_- \psi = 0
  \end{equation*}
   and 
  \begin{equation*}
    \Gamma_1 g = P_+ \beta g|_{\partial \Omega} = P_+ \beta^2 \psi = \psi,
  \end{equation*}
  that is, $\psi \in \ran(\Gamma_1 \upharpoonright \ker \Gamma_0)$.
  Together with \eqref{erstmaldas} we conclude \eqref{ran_Gamma_domain}.

  Finally, observe that
  \begin{equation*}
  \begin{split}
    \ker \Gamma_0 
      &= \{ f \in H^1(\Omega; \mathbb{C}^4): f|_{\partial \Omega} = -i \beta (\alpha \cdot \nu) f|_{\partial \Omega} \}
      = \dom T_{\textup{MIT}},\\
      \ker \Gamma_1 
      &= \{ f \in H^1(\Omega; \mathbb{C}^4): f|_{\partial \Omega} = i \beta (\alpha \cdot \nu) f|_{\partial \Omega} \}
      = \dom T_{-\textup{MIT}}.\\
  \end{split}
  \end{equation*}
  Therefore 
  $T \upharpoonright \ker \Gamma_0$ coincides with the MIT bag Dirac operator $T_{\textup{MIT}}$
  and $T \upharpoonright \ker \Gamma_1$ coincides with $T_{-\textup{MIT}}$, which shows \eqref{aberjadoch}.
  Note that both operators are self-adjoint; cf. Proposition~\ref{proposition_MIT_bag_operator} and Proposition~\ref{prop_A_infty}.
  Thus $\{ \mathcal{G}_\Omega, \Gamma_0, \Gamma_1 \}$ is a quasi boundary triple for 
  $T \subset T_{\textup{max}}$.
\end{proof}

Now we compute the $\gamma$-field and the Weyl function associated to the quasi boundary triple 
in Theorem~\ref{theorem_quasi_triple_domain}. It turns out that these operators are closely related to
the integral operators $\Phi_\lambda$ and $\mathcal{C}_\lambda$ defined in Section~\ref{section_integral_op}.
For the next proposition recall that $\frac{1}{2} \beta + \mathcal{C}_{\lambda, 1/2}$ admits a bounded and everywhere defined inverse 
in $H^{1/2}({\partial \Omega}; \mathbb{C}^4)$
for $\lambda \in \mathbb{C} \setminus ( (-\infty, -m] \cup [m, \infty))$, 
see Proposition~\ref{proposition_C_lambda_inverse}.

\begin{prop} \label{proposition_gamma_Weyl_quasi_domain}
  Let $\{ \mathcal{G}_\Omega, \Gamma_0, \Gamma_1 \}$ be the quasi boundary triple 
  from Theorem~\ref{theorem_quasi_triple_domain}, let 
  $\lambda \in \mathbb{C} \setminus ( (-\infty, -m] \cup [m, \infty) ) \subset \rho(T_{\textup{MIT}})$,
  and let $\Phi_{\lambda, 1/2}$ and $\mathcal{C}_{\lambda, 1/2}$ be the operators in Proposition~\ref{proposition_Phi_lambda} 
  and Proposition~\ref{proposition_C_lambda}, respectively.
  Then the following statements hold:
  \begin{itemize}
    \item[$\textup{(i)}$] The value of the $\gamma$-field corresponding to $\{ \mathcal{G}_\Omega, \Gamma_0, \Gamma_1 \}$ is given by
    \begin{equation}\label{gammaformel}
      \gamma(\lambda) 
          =  \Phi_{\lambda, 1/2} \left( \frac{1}{2} \beta + \mathcal{C}_{\lambda, 1/2} \right)^{-1},\quad \dom \gamma(\lambda) = \mathcal{G}_\Omega^{1/2}.
    \end{equation}
    Each $\gamma(\lambda)$ is a densely defined bounded operator from $\mathcal{G}_\Omega$
    to $L^2(\Omega; \mathbb{C}^4)$, and a bounded and everywhere defined operator from 
    $\mathcal{G}_\Omega^{1/2}$ to $H^1(\Omega; \mathbb{C}^4)$.
    \item[$\textup{(ii)}$] The adjoint $\gamma(\lambda)^*$ of the $\gamma$-field corresponding to $\{ \mathcal{G}_\Omega, \Gamma_0, \Gamma_1 \}$ is given by
    \begin{equation*}
      \gamma(\lambda)^*
          =  P_+ \left( \frac{1}{2} \beta + \mathcal{C}_{\overline{\lambda}, 1/2} \right)^{-1}  \Phi_{\lambda}^*.
    \end{equation*}
    Each $\gamma(\lambda)^*$ is bounded from $L^2(\Omega; \mathbb{C}^4)$
    to $\mathcal{G}_\Omega^{1/2}$, and, in particular, compact from $L^2(\Omega; \mathbb{C}^4)$ to $\mathcal{G}_\Omega$.
    \item[$\textup{(iii)}$] The value of the Weyl function corresponding to $\{ \mathcal{G}_\Omega, \Gamma_0, \Gamma_1 \}$ is given by
    \begin{equation}\label{mformel}
      M(\lambda) = - P_+ \left( \frac{1}{2} \beta + \mathcal{C}_{\lambda, 1/2} \right)^{-1} P_+,\quad \dom M(\lambda) = \mathcal{G}_\Omega^{1/2}.
    \end{equation}
    Each $M(\lambda)$ is a densely defined and bounded operator in $\mathcal{G}_\Omega$, and 
    a bounded and everywhere defined operator in $\mathcal{G}_\Omega^{1/2}$.    
  \end{itemize}
\end{prop}

\begin{proof}
  In the following let $\lambda \in \mathbb{C} \setminus ( (-\infty, -m] \cup [m, \infty))$ be fixed. From \eqref{ran_Gamma_domain}  and 
  the definition of the $\gamma$-field and Weyl function it follows that
  $$\dom \gamma(\lambda) = \dom M(\lambda) = \ran \Gamma_0 = \mathcal{G}_\Omega^{1/2}.$$
  
  For the proof of item~(i) consider $\varphi \in \ran \Gamma_0 = \mathcal{G}_\Omega^{1/2}$ and recall that $\gamma(\lambda) \varphi$
  is the unique solution of the boundary value problem
  \begin{equation} \label{equation_gamma_field_bvp_domain}
    (T - \lambda) f = 0 \quad \text{and} \quad \Gamma_0 f = \varphi.
  \end{equation}
  We set 
  \begin{equation*}
    f_\lambda =  \Phi_{\lambda, 1/2} \left( \frac{1}{2} \beta + \mathcal{C}_{\lambda, 1/2} \right)^{-1} \varphi.
  \end{equation*}
  Then, due to the mapping properties of $\Phi_{\lambda, 1/2}$ and $\big( \frac{1}{2} \beta + \mathcal{C}_{\lambda, 1/2} \big)^{-1}$,
  see Propositions~\ref{proposition_Phi_lambda} and Proposition~\ref{proposition_C_lambda_inverse},
  we have $f_\lambda \in H^1(\Omega; \mathbb{C}^4) = \dom T$. For \eqref{gammaformel} it suffices to check that $f_\lambda$
  solves the boundary value problem~\eqref{equation_gamma_field_bvp_domain}.
  In fact, by Proposition~\ref{proposition_Phi_lambda}~(ii) we have
  $(T - \lambda) f_\lambda = 0$ and using Lemma~\ref{lem_Phi_lambda} we get
  \begin{equation*}
    \begin{split}
      \Gamma_0 f_\lambda &= P_+ f_\lambda|_{\partial \Omega}
          = P_+ \left( -\frac{i}{2} (\alpha \cdot \nu) + \mathcal{C}_{\lambda,1/2} \right) 
          \left( \frac{1}{2} \beta + \mathcal{C}_{\lambda, 1/2} \right)^{-1} \varphi\\
      &= P_+ \left( -\frac{i}{2} (\alpha \cdot \nu) - \frac{1}{2} \beta + \frac{1}{2} \beta+ \mathcal{C}_{\lambda,1/2} \right)
          \left( \frac{1}{2} \beta + \mathcal{C}_{\lambda, 1/2} \right)^{-1} \varphi \\
      &= P_+ \left( -\frac{i}{2} (\alpha \cdot \nu) \beta - \frac{1}{2} I_4 \right) \beta 
          \left( \frac{1}{2} \beta + \mathcal{C}_{\lambda, 1/2} \right)^{-1} \varphi + P_+ \varphi.
    \end{split}
  \end{equation*}
  Using that $\varphi \in \mathcal{G}_\Omega$, \eqref{eq_commutation}, and $P_+P_-=0$
  we obtain then
  \begin{equation*}
    \begin{split}
      \Gamma_0 f_\lambda
          &= - P_+ \frac{1}{2}\bigl( I_4 - i \beta (\alpha \cdot \nu)   \bigr) \beta \left( \frac{1}{2} \beta 
          + \mathcal{C}_{\lambda, 1/2} \right)^{-1} \varphi +\varphi \\
      &= - P_+ P_- \beta \left( \frac{1}{2} \beta + \mathcal{C}_{\lambda, 1/2} \right)^{-1} \varphi + \varphi
        = \varphi.
    \end{split}
  \end{equation*}
  Hence, $f_\lambda$ is the unique solution of the boundary value problem~\eqref{equation_gamma_field_bvp_domain}. This implies $\gamma(\lambda) \varphi = f_\lambda$
  and leads to the representation \eqref{gammaformel}.

  It remains to check the mapping properties of $\gamma(\lambda)$ in (i). 
  From the definition of the $\gamma$-field it is clear that $\gamma(\lambda)$ is a densely defined bounded operator 
  from $\mathcal{G}_\Omega$ to $L^2(\Omega; \mathbb{C}^4)$.
  Moreover, from Proposition~\ref{proposition_Phi_lambda}~(i)
  and Proposition~\ref{proposition_C_lambda_inverse}~(i) it also follows that
  $\gamma(\lambda)$ is a bounded and everywhere defined operator from 
    $\mathcal{G}_\Omega^{1/2}$ to $H^1(\Omega; \mathbb{C}^4)$.

  Next we prove item (ii). Let $f \in L^2(\Omega; \mathbb{C}^4)$ and $\varphi \in \mathcal{G}_\Omega^{1/2} = \dom \gamma(\lambda)$ be fixed. Then, using~\eqref{adjoint_C_lambda} we find that
  \begin{equation*}
    \begin{split}
      \big(\varphi, \gamma(\lambda)^* f \big)_{\partial \Omega} 
          &= \big(\gamma(\lambda) \varphi, f \big)_{\Omega}
          = \left( \Phi_{\lambda, 1/2} \left( \frac{1}{2} \beta + \mathcal{C}_{\lambda, 1/2} \right)^{-1} \varphi, f \right)_{\Omega} \\
      &= \left( \left( \frac{1}{2} \beta + \mathcal{C}_{\lambda} \right)^{-1} \varphi, \Phi_\lambda^* f \right)_{\partial \Omega}
      = \left( \varphi, \left( \frac{1}{2} \beta + \mathcal{C}_{\overline{\lambda}} \right)^{-1} \Phi_\lambda^* f \right)_{\partial \Omega}.
    \end{split}
  \end{equation*}
  Since this holds for all $\varphi \in \mathcal{G}_\Omega^{1/2}$ and $\Phi_\lambda^*: L^2(\Omega; \mathbb{C}^4) \rightarrow H^{1/2}(\partial \Omega; \mathbb{C}^4)$ 
  by~\eqref{def_Phi_lambda_adjoint}, we find the claimed representation for $\gamma(\lambda)^*$. Moreover, since $(\frac{1}{2} \beta + \mathcal{C}_{\overline\lambda, 1/2})^{-1}$ is bounded in $H^{1/2}(\partial \Omega; \mathbb{C}^4)$ by Proposition~\ref{proposition_C_lambda_inverse}, we get that $\gamma(\lambda)^*$ is bounded from $L^2(\Omega; \mathbb{C}^4)$ to $\mathcal{G}_\Omega^{1/2}$ and compact from $L^2(\Omega; \mathbb{C}^4)$ to $\mathcal{G}_\Omega$, as $\mathcal{G}_\Omega^{1/2} = P_+ (H^{1/2}(\partial \Omega; \mathbb{C}^4))$ is compactly embedded in $\mathcal{G}_\Omega = P_+ (L^2(\partial \Omega; \mathbb{C}^4))$.

  Finally, we show assertion~(iii). For this we use \eqref{gammaformel}, Lemma~\ref{lem_Phi_lambda}, and compute for $\varphi \in \mathcal{G}_\Omega^{1/2}$
  \begin{equation*}
    \begin{split}
      M(\lambda) \varphi &= \Gamma_1 \gamma(\lambda) \varphi 
        = P_+ \beta \left( \Phi_{\lambda, 1/2}
        \left( \frac{1}{2} \beta + \mathcal{C}_{\lambda, 1/2} \right)^{-1} \varphi \right) \Bigg|_{\partial \Omega} \\
      &= P_+ \beta \left( -\frac{i}{2} (\alpha \cdot \nu) - \frac{1}{2} \beta + \frac{1}{2} \beta + \mathcal{C}_{\lambda,1/2} \right)
         \left( \frac{1}{2} \beta + \mathcal{C}_{\lambda, 1/2} \right)^{-1} \varphi \\
      &=  P_+ \left( -\frac{i}{2} \beta (\alpha \cdot \nu) - \frac{1}{2} I_4 \right) 
          \left( \frac{1}{2} \beta + \mathcal{C}_{\lambda, 1/2} \right)^{-1} \varphi 
          + P_+ \beta \varphi\\
      &= - P_+^2 \left( \frac{1}{2} \beta + \mathcal{C}_{\lambda, 1/2} \right)^{-1} \varphi + P_+ \beta \varphi. 
    \end{split}
  \end{equation*}
  Since $P_+^2 = P_+$ and $P_+ \beta \varphi=\beta P_-\varphi=0$ (see \eqref{beta_P_pm}) for $\varphi \in \mathcal{G}_\Omega^{1/2}$, the representation \eqref{mformel}
  for the Weyl function follows. 
  
  It is a consequence of mapping properties of $( \frac{1}{2} \beta + \mathcal{C}_{\lambda, 1/2})^{-1}$ from
  Proposition~\ref{proposition_C_lambda_inverse}~(i) that
  each $M(\lambda)$ is a densely defined and bounded operator in $\mathcal{G}_\Omega$, and 
  a bounded and everywhere defined operator in $\mathcal{G}_\Omega^{1/2}$.  
\end{proof}

In the next proposition we derive a useful formula for the inverse of $M(\lambda)$. 
Let $T_{-\text{MIT}}$ be given by~\eqref{def_A_infty} and let $T_{-\text{MIT}}^{\Omega^c}$ be the Dirac operator acting 
in $L^2(\mathbb R^3\setminus\overline\Omega; \mathbb{C}^4)$ with the same boundary conditions as $T_{-\text{MIT}}$.
Recall that by Proposition~\ref{prop_A_infty} we have 
$$\bigl(\rho(T_{-\textup{MIT}}) \cap \rho(T_{-\textup{MIT}}^{\Omega^c})\bigr) \subset \mathbb{C} \setminus \big( (-\infty, -m] \cup [m, \infty) \big),$$
that the latter set is contained in $\rho(T_{\textup{MIT}}) \cap \rho(T_{\textup{MIT}}^{\Omega^c})$,
and that by Proposition~\ref{proposition_C_lambda_inverse}~(ii) the operator $-\frac{1}{2} \beta + \mathcal{C}_{\lambda, 1/2}$ is boundedly 
invertible in $H^{1/2}(\partial \Omega; \mathbb{C}^4)$ for any number $\lambda \in \rho(T_{-\textup{MIT}}) \cap \rho(T_{-\textup{MIT}}^{\Omega^c})$.

\begin{prop} \label{proposition_M_inv_domain}
 Let $M$ be the Weyl function corresponding to the quasi boundary triple $\{ \mathcal{G}_\Omega, \Gamma_0, \Gamma_1 \}$,
  assume that $\lambda \in \rho(T_{-\textup{MIT}}) \cap \rho(T_{-\textup{MIT}}^{\Omega^c})$, and
  let $\mathcal{C}_{\lambda, 1/2}$ be the operator from Proposition~\ref{proposition_C_lambda}.
  Then $M(\lambda)$ admits a bounded and everywhere defined inverse in $\mathcal{G}_\Omega^{1/2}$
  which is given by
  \begin{equation}\label{inverse M(lambda)} 
    M(\lambda)^{-1} =  P_+ \beta \left( -\frac{1}{2} \beta + \mathcal{C}_{\lambda, 1/2} \right)^{-1} \beta P_+.
  \end{equation}
\end{prop}

\begin{proof}
Observe first that $\{ \mathcal{G}_\Omega, \widehat{\Gamma}_0, \widehat{\Gamma}_1 \}$, where 
  \begin{equation}\label{neuemaps}
    \widehat{\Gamma}_0 = \Gamma_1 \quad \text{and} \quad \widehat{\Gamma}_1 = -\Gamma_0,
  \end{equation}
  is a quasi boundary triple for $T\subset T_{\textup{max}}$ such that $T \upharpoonright \ker \widehat\Gamma_0 = T_{-\textup{MIT}}$.
  In fact, using that $\{ \mathcal{G}_\Omega, \Gamma_0, \Gamma_1 \}$ is quasi boundary triple it follows 
  that the abstract Green identity is satisfied by the boundary mappings in \eqref{neuemaps} and that the range of $(\widehat\Gamma_0,\widehat\Gamma_1)^\top$
  is dense. Moreover, $T_{-\textup{MIT}}= T \upharpoonright \ker \widehat\Gamma_0$ is a self-adjoint operator by Proposition~\ref{prop_A_infty}.
  Note that Weyl function $\widehat M$ corresponding to the quasi boundary triple $\{ \mathcal{G}_\Omega, \widehat{\Gamma}_0, \widehat{\Gamma}_1 \}$ is given for our choice of $\lambda \in \rho(T_{-\textup{MIT}}) \cap \rho(T_{-\textup{MIT}}^{\Omega^c}) \subset \rho(T_{\textup{MIT}}) \cap \rho(T_{-\textup{MIT}})$ by
  \begin{equation*}
    \widehat{M}(\lambda) 
        = \widehat{\Gamma}_1 \big( \widehat{\Gamma}_0 \upharpoonright \ker (T - \lambda) \big)^{-1} 
        = -(M(\lambda))^{-1}.
  \end{equation*} 
  
  Thus it remains to compute the value of the Weyl function $\widehat M(\lambda)$. 
  We first show the explicit formula 
  \begin{equation}\label{gammaformelhat}
      \widehat\gamma(\lambda) 
          =  \Phi_{\lambda, 1/2} \left( -\frac{1}{2} \beta + \mathcal{C}_{\lambda, 1/2} \right)^{-1}\beta ,\quad \dom \widehat\gamma(\lambda) = \mathcal{G}_\Omega^{1/2},
    \end{equation}
  for the $\gamma$-field corresponding to the quasi boundary triple $\{ \mathcal{G}_\Omega, \widehat{\Gamma}_0, \widehat{\Gamma}_1 \}$ using a similar argument 
  as in the proof of Proposition~\ref{proposition_gamma_Weyl_quasi_domain}. In fact, it is clear that $\dom \widehat\gamma(\lambda) =\ran\widehat\Gamma_0= \mathcal{G}_\Omega^{1/2}$. Next,  consider
  $\varphi \in \dom \widehat{\gamma}(\lambda)$  and recall that $\widehat\gamma(\lambda) \varphi$
  is the unique solution of the boundary value problem
  \begin{equation*}
    (T - \lambda) = 0 \quad \text{and} \quad \widehat{\Gamma}_0 f_\lambda = \varphi.
  \end{equation*}
  We set 
  \begin{equation*}
    f_\lambda = \Phi_{\lambda, 1/2} \left( -\frac{1}{2} \beta + \mathcal{C}_{\lambda, 1/2} \right)^{-1} \beta \varphi.
  \end{equation*}
  Then we have $(T-\lambda)f_\lambda=0$ and using $P_+P_-=0$ and $\varphi \in \mathcal{G}_\Omega$ 
  we obtain in a similar way as in the proof of  Proposition~\ref{proposition_gamma_Weyl_quasi_domain} that
  \begin{equation*}
    \begin{split}
      \widehat{\Gamma}_0 f_\lambda &= \Gamma_1 f_\lambda
        = P_+ \beta \left( \Phi_{\lambda, 1/2} 
        \left( -\frac{1}{2} \beta + \mathcal{C}_{\lambda, 1/2} \right)^{-1} \beta \varphi \right) \Bigg|_{\partial \Omega} \\
      &= P_+ \beta \left( -\frac{i}{2} (\alpha \cdot \nu) + \frac{1}{2} \beta 
        - \frac{1}{2} \beta + \mathcal{C}_{\lambda,1/2} \right)
        \left( -\frac{1}{2} \beta + \mathcal{C}_{\lambda, 1/2} \right)^{-1} \beta \varphi \\
      &=  P_+ P_- \left( -\frac{1}{2} \beta + \mathcal{C}_{\lambda, 1/2} \right)^{-1} \beta \varphi
        + P_+ \beta^2 \varphi\\
        &=\varphi,
    \end{split}
  \end{equation*}
  which leads to \eqref{gammaformelhat}.

  Let us now compute the Weyl function
  $\widehat{M}$. Using \eqref{gammaformelhat} and Lemma~\ref{lem_Phi_lambda} we find
  \begin{equation*}
    \begin{split}
      \widehat{M}(\lambda) \varphi &= \widehat{\Gamma}_1 \widehat{\gamma}(\lambda) \varphi
        = - \Gamma_0 \widehat{\gamma}(\lambda) \varphi \\
      &= -P_+ \left( \Phi_{\lambda, 1/2} 
        \left( -\frac{1}{2} \beta + \mathcal{C}_{\lambda, 1/2} \right)^{-1} \beta \varphi \right) \Bigg|_{\partial \Omega} \\
      &= -P_+ \left( -\frac{i}{2} (\alpha \cdot \nu) + \frac{1}{2} \beta 
        - \frac{1}{2} \beta + \mathcal{C}_{\lambda,1/2} \right)
        \left( -\frac{1}{2} \beta + \mathcal{C}_{\lambda, 1/2} \right)^{-1} \beta \varphi \\
      &= - P_+ \beta P_- \left( -\frac{1}{2} \beta + \mathcal{C}_{\lambda, 1/2} \right)^{-1} \beta \varphi
        - P_+ \beta \varphi.
    \end{split}
  \end{equation*}
  Thanks to \eqref{beta_P_pm}, $P_+^2=P_+$, and that $\varphi \in \mathcal{G}_\Omega$ we finally obtain
  \begin{equation*}
    \begin{split}
      \widehat{M}(\lambda) \varphi &
        = - P_+ \beta \left( -\frac{1}{2} \beta + \mathcal{C}_{\lambda, 1/2} \right)^{-1} \beta P_+ \varphi
        - \beta P_-\varphi \\
      &= - P_+ \beta \left( -\frac{1}{2} \beta + \mathcal{C}_{\lambda, 1/2} \right)^{-1} \beta P_+ \varphi,
    \end{split}
  \end{equation*}
  which gives \eqref{inverse M(lambda)}.
\end{proof}

Finally, we state a lemma on the invertibility of $\vartheta - M(\lambda)$ for a H\"older continuous 
function $\vartheta$. This result will be needed in the proofs of several results of this paper. Recall that the closure $\overline{M(\lambda)}$ of the Weyl function corresponding to the triple $\{ \mathcal{G}_\Omega, \Gamma_0, \Gamma_1 \}$ is bounded in $\mathcal{G}_\Omega$, see Proposition~\ref{proposition_gamma_Weyl_quasi_domain}.

\begin{lem} \label{lemma_M_inv}
Let $M$ be the Weyl function corresponding to the quasi boundary triple $\{ \mathcal{G}_\Omega, \Gamma_0, \Gamma_1 \}$ 
and let $\vartheta \in \textup{Lip}_a({\partial \Omega})$ for some $a \in(\frac{1}{2}, 1]$ be a real-valued function such that 
$|\vartheta(x)| \neq 1$ for all $x \in \partial \Omega$. 
Then the following statements hold:
\begin{itemize}
 \item [$\textup{(i)}$] For all $\lambda \in \mathbb{C} \setminus \mathbb{R}$ the operator $\vartheta - M(\lambda)$ has a bounded and everywhere defined inverse in $\mathcal{G}_\Omega^{1/2}$.
 \item [$\textup{(ii)}$] For all $\lambda \in \mathbb{C} \setminus \mathbb{R}$ the operator $\vartheta - \overline{M(\lambda)}$ has a bounded and everywhere defined inverse in $\mathcal{G}_\Omega$.
\end{itemize}
\end{lem}
\begin{proof}
  To prove (i) and (ii) some preparations are needed. 
  First, due to the explicit form of the operators $M(\lambda)$ and $M(\lambda)^{-1}$ from Proposition~\ref{proposition_gamma_Weyl_quasi_domain} 
  and Proposition~\ref{proposition_M_inv_domain} it is easy to see with the help of Proposition~\ref{proposition_C_lambda_inverse} that 
  $M(\lambda)$ and $M(\lambda)^{-1}$ have bounded extensions 
  onto $\mathcal{G}_\Omega$ given by
  \begin{equation*}
    \overline{M(\lambda)} = - P_+ \left( \frac{1}{2} \beta + \mathcal{C}_{\lambda} \right)^{-1} P_+ 
  \end{equation*}
  and
  \begin{equation*}
   \overline{M(\lambda)^{-1}} = P_+ \beta \left( -\frac{1}{2} \beta + \mathcal{C}_{\lambda} \right)^{-1} \beta P_+=
    P_+  \left( -\frac{1}{2} \beta + \beta\mathcal{C}_{\lambda}\beta \right)^{-1}  P_+,
  \end{equation*}
 respectively.
  We claim that the operator 
  \begin{equation} \label{equation_mapping_property}
    \mathcal{K}_\lambda = \vartheta \overline{M(\lambda)^{-1}} - \overline{M(\lambda)} \vartheta: \mathcal{G}_\Omega \rightarrow \mathcal{G}_\Omega^{1/2}
  \end{equation}
  is bounded. In particular, since $\mathcal{G}_\Omega^{1/2} = P_+(H^{1/2}(\partial \Omega; \mathbb{C}^4))$ is compactly embedded in 
  $\mathcal{G}_\Omega = P_+( L^2(\partial \Omega; \mathbb{C}^4))$, this implies that $\mathcal{K}_\lambda$ is a compact operator in $\mathcal{G}_\Omega$ and that 
   $\mathcal{K}_{\lambda, 1/2} = \mathcal{K}_\lambda \upharpoonright \mathcal{G}_\Omega^{1/2}$ is a compact operator
  in $\mathcal{G}_\Omega^{1/2}$.
  
  To verify the boundedness of $\mathcal{K}_\lambda$ in \eqref{equation_mapping_property} we note first that
  \begin{equation} \label{equation_mapping_property1}
    \begin{split}
      \vartheta \overline{M(\lambda)} - \overline{M(\lambda)} \vartheta
          = -P_+ \vartheta \left( \frac{1}{2} \beta + \mathcal{C}_{\lambda} \right)^{-1} P_+ + 
          P_+ \left( \frac{1}{2} \beta + \mathcal{C}_{\lambda} \right)^{-1} \vartheta P_+~~~&  \\
      =  -P_+ \left( \frac{1}{2} \beta + \mathcal{C}_{\lambda, 1/2} \right)^{-1} 
          \left( \mathcal{C}_{\lambda} \vartheta - \vartheta \mathcal{C}_{\lambda}  \right)
          \left( \frac{1}{2} \beta + \mathcal{C}_{\lambda} \right)^{-1} P_+,&
    \end{split}
  \end{equation}
  which, by Proposition~\ref{proposition_commutator_C_lambda} and Proposition~\ref{proposition_C_lambda_inverse}, is
  a bounded operator from $\mathcal{G}_\Omega$ to~$\mathcal{G}_\Omega^{1/2}$. Next, we have
  \begin{equation} \label{equation_mapping_property2}
    \begin{split}
      \overline{M(\lambda)} - \overline{M(\lambda)^{-1}}
          = -P_+ \left( \frac{1}{2} \beta + \mathcal{C}_{\lambda} \right)^{-1} P_+
          - P_+ \left( -\frac{1}{2} \beta + \beta \mathcal{C}_{\lambda} \beta \right)^{-1} P_+& \\
      = -P_+ \left( -\frac{1}{2} \beta + \beta \mathcal{C}_{\lambda, 1/2} \beta \right)^{-1} 
          \left( \mathcal{C}_{\lambda} \beta + \beta \mathcal{C}_{\lambda} \right) \beta
          \left( \frac{1}{2} \beta + \mathcal{C}_{\lambda} \right)^{-1} P_+&,
    \end{split}
  \end{equation}
  which, 
  by Proposition~\ref{proposition_commutator} and Proposition~\ref{proposition_C_lambda_inverse}, is
  also a bounded operator from~$\mathcal{G}_\Omega$ to~$\mathcal{G}_\Omega^{1/2}$, as $L^2(\partial \Omega; \mathbb{C}^4)$ is continuously embedded in $H^{-1/2}(\partial \Omega; \mathbb{C}^4)$. 
  Combining~\eqref{equation_mapping_property1} with~\eqref{equation_mapping_property2} and Lemma~\ref{lemma_mult_op} we conclude that the operator
  \begin{equation*}
    \mathcal{K}_\lambda = \vartheta \overline{M(\lambda)^{-1}} - \overline{M(\lambda)} \vartheta = \vartheta \big( \overline{M(\lambda)^{-1}} - \overline{M(\lambda)} \big) + \vartheta \overline{M(\lambda)} - \overline{M(\lambda)} \vartheta
  \end{equation*}
  in \eqref{equation_mapping_property} is bounded.

  For what follows it is important to note that the assumptions $|\vartheta(x)| \neq 1$ for all $x\in\partial\Omega$ and $a > \frac{1}{2}$ ensure that the functions 
  $(\vartheta^2 - 1), (\vartheta^2 - 1)^{-1} \in \text{Lip}_a(\partial \Omega)$ 
  give rise to bounded and boundedly invertible multiplication operators in $\mathcal{G}_\Omega$ and $\mathcal{G}_\Omega^{1/2}$, see Lemma~\ref{lemma_mult_op}.

  Let us now prove $\textup{(i)}$, i.e. that $\vartheta - M(\lambda)$ has a bounded inverse in $\mathcal{G}_\Omega^{1/2}$. Since $\vartheta - M(\lambda)$ 
  is bounded in $\mathcal{G}_\Omega^{1/2}$ by Lemma~\ref{lemma_mult_op} and Proposition~\ref{proposition_gamma_Weyl_quasi_domain}, 
  it suffices to show that this operator is bijective in $\mathcal{G}_\Omega^{1/2}$. Note that the operator $T \upharpoonright \ker(\Gamma_1 - \vartheta \Gamma_0)$
  is symmetric since $\vartheta$ is a real-valued function (this is an immediate consequence of the abstract Green's identity). Hence,
  $\vartheta - M(\lambda)$ is injective as otherwise the symmetric operator $T \upharpoonright \ker(\Gamma_1 - \vartheta \Gamma_0)$ 
  would have the non-real eigenvalue $\lambda$ by Theorem~\ref{theorem_krein_abstract}. Moreover, 
  we have 
  \begin{equation} \label{range_condition}
    \begin{split}
      \ran(\vartheta - M(\lambda)) 
          &\supset \ran \big[ (\vartheta - M(\lambda)) \big( \vartheta + M(\lambda)^{-1} \big) \big] \\
      &= \ran \big[ \vartheta^2 - 1 + \vartheta M(\lambda)^{-1} - M(\lambda) \vartheta \big].
    \end{split}
  \end{equation}
  The operators $\vartheta - M(\lambda)$ and $\vartheta + M(\lambda)^{-1} = (I_4 + \vartheta M(\lambda)) M(\lambda)^{-1}$ are both injective
  as otherwise one of the symmetric operators $T \upharpoonright \ker(\Gamma_1 - \vartheta \Gamma_0)$ and $T \upharpoonright \ker(\Gamma_0 + \vartheta \Gamma_1)$ 
  would have the non-real eigenvalue $\lambda$ by Theorem~\ref{theorem_krein_abstract} and Theorem~\ref{theorem_krein_abstract_B}, respectively. 
  Therefore
  \begin{equation}\label{nagut23}
    (\vartheta - M(\lambda)) \big( \vartheta + M(\lambda)^{-1} \big)
      = (\vartheta^2 - 1) \left[ 1 + \frac{1}{\vartheta^2 - 1} \mathcal{K}_{\lambda, 1/2} \right]
  \end{equation}
  is injective, and since $\mathcal{K}_{\lambda, 1/2} = \mathcal{K}_\lambda \upharpoonright \mathcal{G}_\Omega^{1/2}$ is a compact operator in $\mathcal{G}_\Omega^{1/2}$,
  we conclude from Fredholm's alternative and the bijectivity of $\vartheta^2 - 1$ in $\mathcal{G}_\Omega^{1/2}$ that 
  the operator \eqref{nagut23} is bijective in $\mathcal{G}_\Omega^{1/2}$. From \eqref{range_condition} 
  we conclude $\mathcal{G}_\Omega^{1/2} \subset \ran(\vartheta - M(\lambda))$ and hence we have shown that 
  $\vartheta - M(\lambda)$ is bijective in~$\mathcal{G}_\Omega^{1/2}$.
  
  Let us now focus on $\textup{(ii)}$. The proof that $\vartheta - \overline{M(\lambda)}$ has a bounded inverse in $\mathcal{G}_\Omega$ follows the same lines as above. 
  The only difference is in the argument that $\vartheta - \overline{M(\lambda)}$ and $\vartheta + \overline{M(\lambda)^{-1}}$ are injective. 
  To see this for, e.g., $\vartheta - \overline{M(\lambda)}$, assume that $\varphi \in \mathcal{G}_\Omega$ is such that 
  $(\vartheta - \overline{M(\lambda)}) \varphi = 0$. Then
  \begin{equation*}
    0 = \big( \vartheta + \overline{M(\lambda)^{-1}} \big)\big(\vartheta - \overline{M(\lambda)}\big) \varphi 
      = (\vartheta^2 - 1) \left[ 1 + \frac{1}{\vartheta^2 - 1} \widetilde{\mathcal{K}}_\lambda \right] \varphi
  \end{equation*}
  with 
  \begin{equation*}
    \widetilde{\mathcal{K}}_\lambda := \overline{M(\lambda)^{-1}} \vartheta - \vartheta \overline{M(\lambda)} = \big(\overline{M(\lambda)^{-1}} - \overline{M(\lambda)} \big) \vartheta + \overline{M(\lambda)} \vartheta - \vartheta \overline{M(\lambda)}.
  \end{equation*}
  From the second equality in the last line we conclude in the same way as in the proof of~\eqref{equation_mapping_property} from~\eqref{equation_mapping_property1} and~\eqref{equation_mapping_property2} that $\widetilde{\mathcal{K}}_\lambda$ maps $\mathcal{G}_\Omega$ into $\mathcal{G}_{\Omega}^{1/2}$.
  Using this and the bijectivity of $\vartheta^2 - 1$ in $\mathcal{G}_\Omega^{1/2}$, we get
  \begin{equation*}
    \varphi = -\frac{1}{\vartheta^2 - 1} \widetilde{\mathcal{K}}_\lambda \varphi \in \mathcal{G}_\Omega^{1/2},
  \end{equation*}
  that is, $\varphi \in \ker (\vartheta - M(\lambda))$. By the above considerations this implies $\varphi = 0$, i.e. $\vartheta - \overline{M(\lambda)}$ is injective. Similarly, one shows that also $\vartheta + \overline{M(\lambda)^{-1}}$ is injective. To show that $\vartheta - \overline{M(\lambda)}$ is surjective, one can use a similar argument as in~\eqref{range_condition} with $M(\lambda)$ and $M(\lambda)^{-1}$ replaced by $\overline{M(\lambda)}$ and $\overline{M(\lambda)^{-1}}$, respectively. The details are left to the reader.
\end{proof}

\section{Dirac operators on domains -- definition and basic spectral properties}
\label{section_Dirac_domain}

This section contains the main results of this paper. First, in Section~\ref{section_def_op} we introduce Dirac operators $A_\vartheta$ on $\Omega$ with 
boundary conditions of the form
\begin{equation} \label{boundary_condition}
  \vartheta P_+ f|_{\partial \Omega} = P_+ \beta f|_{\partial \Omega}
\end{equation}
for a real-valued H\"older continuous function $\vartheta: \partial \Omega \rightarrow \mathbb{R}$ of order $a > \frac{1}{2}$
and $P_+$ given by~\eqref{def_P_pm}. Using 
the quasi boundary triple $\{ \mathcal{G}_\Omega, \Gamma_0, \Gamma_1 \}$ from Theorem~\ref{theorem_quasi_triple_domain} we show that 
$A_\vartheta$ is self-adjoint if $\vert\vartheta(x)\vert \neq 1$ for all $x\in \partial \Omega$. We also obtain a 
Krein type resolvent formula and some qualitative spectral properties of $A_\vartheta$. In Section~\ref{ss bdy MIT} we sketch 
how Dirac operators $A_{[\omega]}$ with boundary conditions of the form
\begin{equation*}
  P_+ f|_{\partial \Omega} = \omega P_+ \beta f|_{\partial \Omega}
\end{equation*}
for a real-valued H\"older continuous function $\omega: \partial \Omega \rightarrow \mathbb{R}$ of order $a > \frac{1}{2}$ can be handled with similar arguments. 
Finally, in Section~\ref{section_confinement} we relate the operators $A_\vartheta$ to 
Dirac operators $B_{\eta, \tau}$ with singular $\delta$-shell potentials of the form \eqref{bet}. 
This relation allows to translate results for $B_{\eta, \tau}$ to $A_\vartheta$, and vice versa.

Throughout this section let $\Omega$ be a bounded or unbounded domain in $\mathbb{R}^3$ with a compact $C^2$-smooth boundary, 
and denote by $\nu$ the normal vector field at $\partial \Omega$ pointing outwards of $\Omega$.

\subsection{Self-adjointness and spectral properties of $A_\vartheta$} \label{section_def_op}

We start with the rigorous mathematical definition of the Dirac operator $A_\vartheta$ with boundary conditions~\eqref{boundary_condition}. 
We shall use the quasi boundary triple $\{ \mathcal{G}_\Omega, \Gamma_0, \Gamma_1 \}$ from  Theorem~\ref{theorem_quasi_triple_domain} in the next definition.

\begin{definition} \label{definition_Dirac_op_domain}
  Let $a \in (\frac{1}{2}, 1 ]$ and let $\vartheta \in \textup{Lip}_a({\partial \Omega})$ be real-valued.
  We define
  $A_\vartheta = T \upharpoonright \ker (\Gamma_1 - \vartheta \Gamma_0)$, which in a more explicit form is given by
  \begin{equation} \label{def_A_tau_domain}
    \begin{split}
      A_\vartheta f &= (-i \alpha \cdot \nabla + m \beta) f, \\
      \dom A_\vartheta &= \bigl\{ f \in H^1(\Omega; \mathbb{C}^4): 
          \vartheta P_+ f|_{\partial \Omega} = P_+ \beta f|_{\partial \Omega} \bigr\}.
    \end{split}
  \end{equation}
\end{definition}

\begin{remark}\label{rem_bc_klar}
The boundary conditions in \eqref{boundary_condition} are the 3D analogue of the boundary conditions used in
\cite{BFSB17_1}. In fact, let $\Omega \subset \mathbb{R}^2$ be a bounded $C^2$-domain. In \cite{BFSB17_1} the boundary conditions
\begin{equation} \label{bc_benguria}
  \big[ I_2 + i \sigma_3 (\sigma \cdot \nu) \cos \eta - \sin \eta \sigma_3 \big] u|_{\partial \Omega} = 0
\end{equation}
for $C^1$-functions $\eta: \partial \Omega \rightarrow \mathbb{R}$ with $\cos[\eta(x)] \notin \{ 0, 1 \}$ for all $x \in \partial \Omega$ are treated.
Here $\sigma = (\sigma_1, \sigma_2)$ and $\sigma_3$ are the Pauli spin matrices in \eqref{def_Pauli_matrices},
$\nu = (\nu_1, \nu_2)$ is the normal vector field at $\partial \Omega$,
and $\sigma \cdot \nu = \sigma_1 \nu_1 + \sigma_2 \nu_2$. To see that~\eqref{bc_benguria} is equivalent to the boundary conditions in \cite{BFSB17_1}, one has to note that $\sigma \cdot \mathbf{t} = -i \sigma_3 (\sigma \cdot \nu)$, where $\mathbf{t} = (-\nu_2, \nu_1)$ is the tangential vector at $\partial \Omega$.
We use the splitting 
\begin{equation*}
  u|_{\partial \Omega} = Q_+ u|_{\partial \Omega} + Q_- u|_{\partial \Omega},\qquad Q_\pm=\frac{1}{2} (I_2 \pm  i \sigma_3 (\sigma \cdot \nu)) ,
\end{equation*}
and remark that $Q_\pm$ is the 2D-analogue of $P_\pm$ from~\eqref{def_P_pm}. Hence, we can rewrite~\eqref{bc_benguria} as
\begin{equation*}
  \big[ I_2 + i \sigma_3 (\sigma \cdot \nu) \cos \eta - \sin \eta \sigma_3 \big] Q_+ u|_{\partial \Omega} = -\big[ I_2 + i \sigma_3 (\sigma \cdot \nu) \cos \eta - \sin \eta \sigma_3 \big] Q_- u|_{\partial \Omega}.
\end{equation*}
With the help of the relations
$i \sigma_3 (\sigma \cdot \nu) Q_\pm = \pm Q_\pm$ and $Q_- = \sigma_3 Q_+ \sigma_3$, we find that~\eqref{bc_benguria} is equivalent to
\begin{equation*} 
  \begin{split}
    \big[ I_2 + \cos \eta I_2 - \sin \eta \sigma_3 \big] Q_+ u|_{\partial \Omega}
    &= \big[ I_2 + i \sigma_3 (\sigma \cdot \nu) \cos \eta - \sin \eta \sigma_3 \big] Q_+ u|_{\partial \Omega}  \\
    &= -\big[ I_2 + i \sigma_3 (\sigma \cdot \nu) \cos \eta - \sin \eta \sigma_3 \big] Q_- u|_{\partial \Omega} \\
    &= -\big[ I_2 - \cos \eta I_2 - \sin \eta \sigma_3 \big] Q_- u|_{\partial \Omega} \\
    &= -\big[ (1-\cos \eta) \sigma_3 - \sin \eta I_2 \big] Q_+ \sigma_3 u|_{\partial \Omega}.
  \end{split}
\end{equation*}
By multiplying this equation with 
\begin{equation*}
  -\big[ (1-\cos \eta) \sigma_3 - \sin \eta I_2 \big]^{-1} = \frac{1}{2\cos \eta (1-\cos \eta)} \big[ (1-\cos \eta) \sigma_3 + \sin \eta I_2 \big],
\end{equation*}
which exists since $\cos [\eta(x)] \notin \{ 0, 1 \}$ is assumed,
we see that~\eqref{bc_benguria} is equivalent to
\begin{equation*}
  \frac{\sin (2 \eta)}{2 \cos \eta (1 - \cos \eta)} Q_+ u|_{\partial \Omega} = Q_+ \sigma_3 u|_{\partial \Omega}, 
\end{equation*}
which is the 2D analogue of the boundary conditions in~\eqref{def_A_tau_domain} for $\vartheta = \frac{\sin (2 \eta)}{2 \cos \eta (1 - \cos \eta)}$.
\end{remark}

It follows immediately from the abstract Green's identity 
that $A_\vartheta$ is symmetric 
for any real-valued function $\vartheta$. 
In order to prove self-adjointness we shall use Theorem~\ref{theorem_krein_abstract},
which also leads to a resolvent formula in terms of the resolvent of the MIT bag operator $T_\text{MIT}$ in~\eqref{def_MIT_op}
and the $\gamma$-field and Weyl function.
We note that in \eqref{mitformel} an explicit formula for $(T_\text{MIT} - \lambda)^{-1}$ is shown.

\begin{thm} \label{theorem_self_adjoint_noncritical}
  Let $a \in (\frac{1}{2}, 1 ]$ and let $\vartheta \in \textup{Lip}_a({\partial \Omega})$ be a real-valued function such that
  $|\vartheta(x)| \neq 1$ for all $x \in \partial \Omega$.   
  Moreover, let $\gamma$ and $M$ be as in Proposition~\ref{proposition_gamma_Weyl_quasi_domain}.
  Then the operator $A_{\vartheta}$ in \eqref{def_A_tau_domain} is self-adjoint in $L^2(\Omega; \mathbb{C}^4)$ and the resolvent formula
  \begin{equation*}
    (A_\vartheta-\lambda)^{-1} = \big(T_{\textup{MIT}} - \lambda\big)^{-1} 
        + \gamma(\lambda) \big( \vartheta - M(\lambda) \big)^{-1} \gamma(\overline{\lambda})^*
  \end{equation*}
   holds for all  $\lambda \in \rho(A_{\vartheta})\cap\rho(T_{\textup{MIT}})$.
\end{thm}
\begin{proof}
  As mentioned above it follows from the abstract
  Green's identity that the operator $A_\vartheta$ is symmetric. 
  Thus, for the self-adjointness it suffices to check that $\ran(A_\vartheta - \lambda) = L^2(\Omega; \mathbb{C}^4)$ 
  holds for some, and hence for all $\lambda \in \mathbb{C_\pm}$. 
  
  Let $f \in L^2(\Omega; \mathbb{C}^4)$ and $\lambda \in \mathbb{C} \setminus \mathbb{R}$. 
  According to Theorem~\ref{theorem_krein_abstract}~(ii) we would have $f \in \ran(A_\vartheta - \lambda)$ if we can show that 
  $\gamma(\overline{\lambda})^* f \in \ran(\vartheta - M(\lambda))$ holds.
  In fact, from $\gamma(\overline{\lambda})^* = \Gamma_1 \big(T_{\textup{MIT}} - \lambda\big)^{-1}$
  and $\dom T_{\textup{MIT}} \subset H^1(\Omega; \mathbb{C}^4)$ we obtain
  $\gamma(\overline{\lambda})^* f \in \mathcal{G}_\Omega^{1/2}$.
  Furthermore, by Lemma~\ref{lemma_M_inv}~(i) the operator $\vartheta - M(\lambda)$ is bijective in 
  $\mathcal{G}_\Omega^{1/2}$ 
  and hence $\gamma(\overline{\lambda})^* f \in \ran(\vartheta - M(\lambda))$, that is, $f \in \ran(A_\vartheta - \lambda)$.
  As $f$ was arbitrary we get  $\ran(A_\vartheta - \lambda) = L^2(\Omega; \mathbb{C}^4)$ for $\lambda \in \mathbb{C} \setminus \mathbb{R}$, so that 
  $A_\vartheta$ is self-adjoint in $L^2(\Omega; \mathbb{C}^4)$. 
  Finally, the formula for the resolvent of $A_\vartheta$ follows from Theorem~\ref{theorem_krein_abstract} and \eqref{aberjadoch}.
\end{proof}

Next, we discuss the basic spectral properties of the operator $A_\vartheta$.
Since these are of a very different nature whether $\Omega$ is bounded or unbounded the two cases are treated separately. 
Assume first that
$\Omega$ is an unbounded $C^2$-domain with compact boundary.
The proof of (ii) is based on the same argument as the proof of \cite[Proposition~3.9]{BHOP19}.

\begin{thm} \label{theorem_basic_spectral_properties_unbounded_domain}
  Let $\Omega$ be the complement of a bounded $C^2$-domain,
  let $a \in (\frac{1}{2}, 1]$, let $\vartheta \in \textup{Lip}_a({\partial \Omega})$ be a real-valued function such that
  $|\vartheta(x)| \neq 1$ for all $x \in \partial \Omega$, and let $A_\vartheta$ be defined by~\eqref{def_A_tau_domain}.
  Then the following statements hold:
  \begin{itemize}
    \item[$\textup{(i)}$] $\sigma_{\textup{ess}}(A_\vartheta) = (-\infty, - m] \cup [m, \infty)$.
    \item[$\textup{(ii)}$] The number of discrete eigenvalues of $A_\vartheta$ is finite.
    \item[$\textup{(iii)}$] $\lambda \in \sigma_{\textup{p}}(A_\vartheta) \cap (-m, m)$ if and only if $0 \in \sigma_\textup{p}(\vartheta - M(\lambda))$.
  \end{itemize}
\end{thm}

\begin{proof}
  We first deal with (i). Let $\gamma$ and $M$ be the $\gamma$-field and the Weyl function corresponding to the quasi boundary triple $\{ \mathcal{G}_\Omega, \Gamma_0, \Gamma_1 \}$, respectively,
  from Proposition~\ref{proposition_gamma_Weyl_quasi_domain}, and let 
  $\overline{\gamma(\lambda)} \in \mathcal{B}(\mathcal{G}_\Omega, L^2(\Omega; \mathbb{C}^4))$ and $\overline{M(\lambda)} \in \mathcal{B}(\mathcal{G}_\Omega)$ be the closures of $\gamma(\lambda)$ and $M(\lambda)$,
  $\lambda\in\rho(T_{\textup{MIT}})$. For $\lambda\in \rho(A_{\vartheta})\cap\rho(T_{\textup{MIT}})$ the resolvent formula in 
  Theorem~\ref{theorem_self_adjoint_noncritical} can be written in the form
  \begin{equation} \label{equation_resolvent_difference_domain1}
      (A_\vartheta-\lambda)^{-1} - \big(T_{\textup{MIT}} - \lambda\big)^{-1}    
      = \overline{\gamma(\lambda)} \big( \vartheta - \overline{M(\lambda)} \big)^{-1} \gamma(\overline{\lambda})^*.
  \end{equation}
   By Proposition~\ref{proposition_gamma_Weyl_quasi_domain}~(ii) the operator
   $\gamma(\overline\lambda)^*$ is compact from $L^2(\Omega; \mathbb{C}^4)$ to $\mathcal{G}_\Omega$.
   Furthermore, by Lemma~\ref{lemma_M_inv} the inverse $( \vartheta - \overline{M(\lambda)})^{-1}$ is bounded in $\mathcal{G}_\Omega$. 
  Since $\overline{\gamma(\lambda)}: \mathcal{G}_\Omega \rightarrow L^2(\Omega; \mathbb{C}^4)$ is bounded, we deduce that the right-hand side in~\eqref{equation_resolvent_difference_domain1}
  is compact in $L^2(\Omega; \mathbb{C}^4)$ and hence the same holds for the left-hand 
  side. Together with Proposition~\ref{proposition_MIT_bag_operator}~(iii) this implies
  \begin{equation*}
    \sigma_{\textup{ess}}(A_\vartheta) = \sigma_{\textup{ess}}(T_{\textup{MIT}}) = (-\infty, -m] \cup [m, \infty).
  \end{equation*}

  To verify assertion~(ii), consider the quadratic form
  \begin{equation*}
    \begin{split}
      \mathfrak{a}[f] &= \| A_\vartheta f \|_\Omega^2, \qquad \dom \mathfrak{a} = \dom A_\vartheta.
    \end{split}
  \end{equation*}
  Since $A_\vartheta$ is a self-adjoint operator it follows that $\mathfrak{a}$ is a closed, non-negative form and by \cite[Theorem~VI~2.1]{kato}
  the unique self-adjoint operator representing this form is $A_\vartheta^2$. Note that the number of eigenvalues (counted with multiplicities) of 
  $A_\vartheta$ in the gap of the essential spectrum $(-m, m)$ is equal to the number of 
  eigenvalues of $A_\vartheta^2$ below $m^2$ (counted with multiplicities). 
  
  To estimate the number of eigenvalues of $A_\vartheta^2$ with the help of the quadratic form $\mathfrak{a}$ let $0 < r < R$ 
  such that ${\partial \Omega} \subset B(0, r)$ and choose
  real-valued functions $g_1, g_2 \in C^\infty(\overline{\Omega}; \mathbb{C})$ with the properties
  \begin{equation*}
    0 \leq g_1, g_2 \leq 1,\,\, g_1 \!\upharpoonright\! (B(0, r) \cap \Omega) \equiv 1,\,\, g_2 \!\upharpoonright\! B(0, R)^c \equiv 1, \quad \text{and} \quad g_1^2 + g_2^2 \equiv 1.
  \end{equation*}
  Note that the properties of $g_1$ and $g_2$ imply that the mapping
  \begin{equation*}
    U: L^2(\Omega; \mathbb{C}^4) \rightarrow L^2\big( \Omega \cap B(0, R); \mathbb{C}^4 \big) \oplus L^2(\mathbb{R}^3 \setminus \overline{B(0, r)}; \mathbb{C}^4), \quad U f = g_1 f \oplus g_2 f,
  \end{equation*}
  is an isometry. Our next goal is to rewrite the form $\mathfrak{a}$ as a sesquilinear form in $L^2( \Omega \cap B(0, R); \mathbb{C}^4) \oplus L^2(\mathbb{R}^3 \setminus \overline{B(0, r)}; \mathbb{C}^4)$.
    For that we will often identify functions defined in $\Omega$ with their restrictions onto $\Omega \cap B(0,R)$ or onto $\mathbb{R}^3 \setminus \overline{B(0,r)}$ 
  and we also identify functions on $\Omega \cap B(0,R)$ or $\mathbb{R}^3 \setminus \overline{B(0,r)}$ with their extensions by zero onto~$\Omega$. 
  In both cases, we will use the same letters for the restrictions and the extended functions.
  
  Let $f \in \dom \mathfrak{a} = \dom A_\vartheta$ be fixed. Then also $g_1 f, g_2 f \in \dom \mathfrak{a}$. Using the relation
  \begin{equation*}
    A_\vartheta (g_j f) = g_j A_\vartheta f - i (\alpha \cdot \nabla g_j) f,\quad j=1,2,
  \end{equation*}
  we find that
  \begin{equation*}
    \begin{split}
      \mathfrak{a}[g_j f] &= \big( g_j A_\vartheta f -i (\alpha \cdot \nabla g_j) f, g_j A_\vartheta f -i (\alpha \cdot \nabla g_j) f \big)_\Omega \\
      &= \big( g_j^2 A_\vartheta f, A_\vartheta f \big)_\Omega + \big\| (\alpha \cdot \nabla g_j) f \big\|_\Omega^2 +  \text{Re}\,  \big( A_\vartheta f, (-i \alpha \cdot \nabla (g_j^2)) f \big)_\Omega.
    \end{split}
  \end{equation*}
  Note that~\eqref{eq_commutation} implies $(\alpha \cdot \nabla g_j)^2 = |\nabla g_j|^2 I_4$, which gives
  \begin{equation*}
    \big\| (\alpha \cdot \nabla g_j) f \big\|_\Omega^2 = \big( (\alpha \cdot \nabla g_j)^2 f, f \big)_\Omega = \big( |\nabla g_j|^2 f,f \big)_\Omega.
  \end{equation*}
  Moreover, since $g_1^2 + g_2^2 \equiv 1$ we have
  \begin{equation*}
    \big( A_\vartheta f, (\alpha \cdot \nabla (g_1^2)) f \big)_\Omega + \big( A_\vartheta f, (\alpha \cdot \nabla (g_2^2)) f \big)_\Omega = \big( A_\vartheta f, (\alpha \cdot \nabla (g_1^2 + g_2^2)) f \big)_\Omega = 0.
  \end{equation*}  
  We set $V=|\nabla g_1|^2 + |\nabla g_2|^2$ and 
  conclude 
  \begin{equation} \label{decomposition_form}
    \begin{split}
      \mathfrak{a}[f] &= \big( (g_1^2 + g_2^2) A_\vartheta f, A_\vartheta f \big)_\Omega \\
      &= \mathfrak{a}[g_1 f] - \big( |\nabla g_1|^2 f,f \big)_\Omega + \mathfrak{a}[g_2 f] - \big( |\nabla g_2|^2 f,f \big)_\Omega \\
      &\qquad \qquad - \text{Re}\, \big( A_\vartheta f, (\alpha \cdot \nabla (g_1^2 + g_2^2)) f \big)_\Omega \\
      &= \mathfrak{a}[g_1 f] - \big( V (g_1^2+g_2^2) f, f \big)_\Omega + \mathfrak{a}[g_2 f]  \\
      &= \mathfrak{a}[g_1 f] - \big( V g_1 f, g_1f \big)_\Omega + \mathfrak{a}[g_2 f] - \big( V g_2 f,g_2 f \big)_\Omega \\
      &=: \mathfrak{b}_1[g_1 f] + \mathfrak{b}_2[g_2 f],
    \end{split}
  \end{equation}
  where $\mathfrak{b}_1$ and $\mathfrak{b}_2$ are the semibounded sesquilinear forms in $L^2(\Omega \cap B(0, R); \mathbb{C}^4)$ 
  and $L^2(\mathbb{R}^3 \setminus \overline{B(0, r)}; \mathbb{C}^4)$ given by
  \begin{equation*}
    \begin{split}
      \mathfrak{b}_1[h] &= \mathfrak{a}[h] - \big( V h, h \big)_{\Omega \cap B(0, R)}, \quad
      \dom \mathfrak{b}_1 = \{ h \in \dom A_\vartheta: \text{supp}\, h \subset \overline{B(0, R)} \},
    \end{split}
  \end{equation*}
  and
  \begin{equation*}
    \begin{split}
      \mathfrak{b}_2[h] &= \mathfrak{a}[h] - \big( V h, h \big)_{\mathbb{R}^3 \setminus \overline{B(0,r)}}, \quad
      \dom \mathfrak{b}_2 = H^1_0(\mathbb{R}^3 \setminus \overline{B(0, r)}; \mathbb{C}^4),
    \end{split}
  \end{equation*}
  respectively.
  
  In the following let us have a closer look at $\mathfrak{b}_1$ and $\mathfrak{b}_2$.
  First we note that with the aid of~\eqref{integration_by_parts} and~\eqref{eq_commutation} 
  one has for $h \in C^\infty_0(\mathbb{R}^3 \setminus \overline{B(0, r)}; \mathbb{C}^4)$ that
  \begin{equation*}
    \begin{split}
      \mathfrak{b}_2[h] &= \| (-i \alpha \cdot \nabla + m \beta) h \|_{\mathbb{R}^3 \setminus \overline{B(0, r)}}^2 - ( V h, h )_{\mathbb{R}^3 \setminus \overline{B(0,r)}}\\
      &= \big( (-i \alpha \cdot \nabla + m \beta)^2 h, h \big)_{\mathbb{R}^3 \setminus \overline{B(0, r)}} - ( V h, h )_{\mathbb{R}^3 \setminus \overline{B(0,r)}}\\
      &= \big( (-\Delta + m^2) h, h \big)_{\mathbb{R}^3 \setminus \overline{B(0, r)}} - ( V h, h )_{\mathbb{R}^3 \setminus \overline{B(0,r)}}\\
      &= \| \nabla h \|_{\mathbb{R}^3 \setminus \overline{B(0, r)}} + m^2 \|h\|_{\mathbb{R}^3 \setminus \overline{B(0, r)}}^2 - ( V h, h)_{\mathbb{R}^3 \setminus \overline{B(0,r)}}. \\
    \end{split}
  \end{equation*}
  By density this extends to
  \begin{equation*}
    \mathfrak{b}_2[h] = \| \nabla h \|_{\mathbb{R}^3 \setminus \overline{B(0, r)}} + m^2 \|h\|_{\mathbb{R}^3 \setminus \overline{B(0, r)}}^2 - ( V h, h)_{\mathbb{R}^3 \setminus \overline{B(0,r)}}
  \end{equation*}
  for all $h \in H^1_0(\mathbb{R}^3 \setminus \overline{B(0, r)}; \mathbb{C}^4) = \dom \mathfrak{b}_2$,
  i.e. $\mathfrak{b}_2$ is the closed semibounded form associated to the self-adjoint operator 
  $B_2 := -\Delta^D + m^2 - V$, 
  where $-\Delta^D$ is the self-adjoint Dirichlet Laplacian in $\mathbb{R}^3 \setminus \overline{B(0, r)}$ and 
  $V = |\nabla g_1|^2 + |\nabla g_2|^2$ is compactly supported in $\overline{B(0, R)} \setminus B(0, r)$  
  due to the construction of $g_1$ and $g_2$. Thus, $B_2$ has only finitely many eigenvalues below $m^2$; for a proof see, e.g., \cite[Proof of Proposition~3.9]{BHOP19}.

  Next, we claim that $\mathfrak{b}_1$ is closed. In fact, let $(h_n)$ be a sequence in $\dom \mathfrak{b}_1$ and let 
  $h \in L^2(\Omega \cap B(0, R))$ such that
  \begin{equation*}
    \mathfrak{b}_1[h_n - h_m] \rightarrow 0 \quad \text{and} \quad \| h_n - h \|_{\Omega \cap B(0, R)} \rightarrow 0, \quad \text{as }m, n \rightarrow \infty.
  \end{equation*}
  By the definition of $\mathfrak{b}_1$ this implies that $\mathfrak{a}[h_n - h_m] \rightarrow 0$ and 
  $\| h_n - h \|_\Omega \rightarrow 0$, as $m, n \rightarrow \infty$. As $\mathfrak{a}$ is closed we have 
  $h \in \dom \mathfrak{a}= \dom A_\vartheta$ and $\mathfrak a[h-h_n]\rightarrow 0$ as $n\rightarrow\infty$. Moreover, it follows from $\| h_n - h \|_\Omega \rightarrow 0$ 
  that $\text{supp}\, h \subset  \overline{B(0, R)}$. Hence, $h \in \dom \mathfrak{b}_1$ and $\mathfrak b_1[h-h_n]\rightarrow 0$ as $n\rightarrow\infty$, thus, $\mathfrak{b}_1$ is closed.
  The semibounded self-adjoint operator $B_1$ associated to $\mathfrak{b}_1$ defined on 
  $\dom B_1 \subset \dom \mathfrak{b}_1 \subset H^1(\Omega \cap B(0, R); \mathbb{C}^4)$ has a compact resolvent 
  in $L^2(\Omega \cap B(0, R); \mathbb{C}^4)$, which implies that the spectrum of $B_1$ is purely discrete and accumulates only to~$\infty$.
  
  By combining the above considerations we are now prepared to show the claim of assertion~(ii). First, we have by~\eqref{decomposition_form} for $f \in \dom \mathfrak{a}$
  \begin{equation*}
    \frac{\mathfrak{a}[f]}{\| f \|_\Omega^2} = \frac{(\mathfrak{b}_1 \oplus \mathfrak{b}_2)[U f]}{\| U f \|_{L^2(\Omega \cap B(0, R); \mathbb{C}^4) \oplus L^2(\mathbb{R}^3 \setminus \overline{B(0, r)}; \mathbb{C}^4)}^2},
  \end{equation*}
  where it was used that $U$ is an isometry, and $U (\dom \mathfrak{a}) \subset \dom (\mathfrak{b}_1 \oplus \mathfrak{b}_2)$. Hence, 
  it follows from the min-max principle that the number of eigenvalues of $A_\vartheta^2$ below $m^2$ is less or equal to the number of 
  eigenvalues of the operator $B_1 \oplus B_2$ associated to $\mathfrak{b}_1 \oplus \mathfrak{b}_2$ below $m^2$. As we have seen above, 
  the number of eigenvalues of $B_1$ and $B_2$ below $m^2$ is finite. Hence, also the number of eigenvalues of $B_1 \oplus B_2$ below $m^2$ 
  is finite. This shows that the number of eigenvalues of $A_\vartheta^2$ below $m^2$ is finite, which yields the claimed result.

  Finally, item~(iii) is an immediate consequence of Theorem~\ref{theorem_krein_abstract}~(i).
\end{proof}

If $\Omega$ is a bounded $C^2$-domain, then $\dom A_\vartheta \subset H^1(\Omega; \mathbb{C}^4)$ is compactly embedded 
in $L^2(\Omega; \mathbb{C}^4)$ and hence the spectrum of
$A_\vartheta$ is purely discrete. It is clear that the Birman Schwinger principle from Theorem~\ref{theorem_krein_abstract} can be used
to detect discrete eigenvalues of $A_\vartheta$ that belong to $\rho(T_{\textup{MIT}})$. The next result, which is a direct consequence
of Proposition~\ref{proposition_Birman_Schwinger_simple} and Proposition~\ref{proposition_T_min_simple}, goes beyond the standard 
Birman Schwinger principle in two ways: First, it allows to detect eigenvalues of $A_\vartheta$ that may be eigenvalues of $T_{\textup{MIT}}$
at the same time, and, second, it enables to use the explicit expression for the values $M(\lambda)$ of the Weyl function 
in Proposition~\ref{proposition_gamma_Weyl_quasi_domain} in terms of integral operators (which we have available only for 
$\lambda \in \mathbb{C} \setminus ( (-\infty, -m] \cup [m, \infty) )$).

\begin{prop} \label{proposition_basic_spectral_properties_bounded_domain}
  Let $\Omega$ be a bounded $C^2$-smooth domain,
  let $a \in (\frac{1}{2}, 1 ]$, let $\vartheta \in \textup{Lip}_a({\partial \Omega})$ be a real-valued function such that
  $|\vartheta(x)| \neq 1$ for all $x \in \partial \Omega$, and let $A_\vartheta$ be defined by~\eqref{def_A_tau_domain}.
  Then $\sigma(A_\vartheta) = \sigma_{\textup{disc}}(A_\vartheta)$ and $\lambda$ is an eigenvalue of $A_\vartheta$
  if and only if there exists $\varphi \in \mathcal{G}_\Omega^{1/2}$ such that
  \begin{equation*}
    \lim_{\varepsilon \searrow 0} i \varepsilon \big( M(\lambda + i \varepsilon) - \vartheta \big)^{-1} \varphi \neq 0,
  \end{equation*}
  where $M$ is the Weyl function corresponding to the quasi boundary triple $\{ \mathcal{G}_\Omega, \Gamma_0, \Gamma_1 \}$ 
  from Proposition~\ref{proposition_gamma_Weyl_quasi_domain}.
\end{prop}

\begin{remark}
  In Theorem~\ref{theorem_basic_spectral_properties_unbounded_domain} we discussed the spectral properties of $A_\vartheta$ for 
  unbounded domains $\Omega$, but we did not address the question of eigenvalues which are embedded in 
  $\sigma_{\textup{ess}}(A_\vartheta) = (-\infty, -m] \cup [m, \infty)$. In fact, if the domain $\Omega$ is connected, 
  then it is not difficult to show that $A_\vartheta$ has no embedded eigenvalues in $\mathbb{R} \setminus [-m, m]$; 
  this can be done in the same way as in \cite[Theorem~3.7]{AMV15}, see also the discussion of this result. 
  If $\Omega$ is not connected, then there exists a bounded set $\Omega_1$ and an unbounded connected domain $\Omega_2$ such that $\Omega = \Omega_1 \cup \Omega_2$. This implies that also $A_\vartheta$ decomposes as 
  $A_\vartheta = A_{\vartheta, 1} \oplus A_{\vartheta, 2}$, where $A_{\vartheta, j}$ is a self-adjoint operator of 
  the form~\eqref{def_A_tau_domain} in $L^2(\Omega_j; \mathbb{C}^4)$, $j \in \{ 1, 2 \}$. By the same reasoning as 
  above $A_{\vartheta, 2}$ has no eigenvalues in $\mathbb{R} \setminus [-m, m]$. Therefore the embedded eigenvalues of 
  $A_\vartheta$ are those of $A_{\vartheta, 1}$, which can be found with the help of 
  Proposition~\ref{proposition_basic_spectral_properties_bounded_domain}.
\end{remark}

Next, we compare the differences of powers of the resolvents of 
$A_\vartheta$ and $T_{\textup{MIT}}$ and show that these operators belong to certain weak Schatten-von Neumann ideals. 
In the proof of this result we will make several times use of 
\begin{equation} \label{equation_product_Schatten_von_Neumann} 
  S T \in \mathfrak{S}_{r, \infty} \quad \text{for} \quad S \in \mathfrak{S}_{p, \infty}, \quad T \in \mathfrak{S}_{q, \infty},
  \quad \text{and} \quad \frac{1}{r} = \frac{1}{p} + \frac{1}{q}.
\end{equation}
Moreover, for holomorphic operator functions $A(\cdot), B(\cdot), C(\cdot)$ the formula
\begin{equation} \label{product_rule}
  \frac{\text{d}^m}{\text{d} \lambda^m} \big( A(\lambda) B(\lambda) C(\lambda) \big) = \sum_{p+q+r = m} \frac{m!}{p! q! r!} \frac{\text{d}^p}{\text{d} \lambda^p} A(\lambda) \frac{\text{d}^q}{\text{d} \lambda^q} B(\lambda) \frac{\text{d}^r}{\text{d} \lambda^r} C(\lambda)
\end{equation}
(see, e.g., \cite[equation~(2.7)]{BLL13_2}) will be employed several times.
Furthermore, if the operator function $A(\cdot)$ is holomorphic and invertible with bounded everywhere defined inverses, 
then also $A(\cdot)^{-1}$ is holomorphic and one has
\begin{equation} \label{chain_rule}
  \frac{\text{d}}{\text{d} \lambda} \big( A(\lambda)^{-1} \big) = -A(\lambda)^{-1} \left(\frac{\text{d}}{\text{d} \lambda} A(\lambda) \right) A(\lambda)^{-1};
\end{equation}
cf. \cite[equation~(2.8)]{BLL13_2}.
The proof of the following theorem is based on Lemma~\ref{lemma_Schatten_von_Neumann} and on the same strategy as in \cite{BLL13_2} or in \cite[Theorem~4.6]{BEHL18}. 
Hence, we have to assume some additional smoothness of $\partial \Omega$.

\begin{thm} \label{theorem_Schatten_von_Neumann_domain}
  Let $\Omega$ be a $C^2$-domain with compact boundary, let $T_{\textup{MIT}}$ be the MIT bag operator in \eqref{def_MIT_op}, 
  let $a \in (\frac{1}{2}, 1]$, let $\vartheta \in \textup{Lip}_a({\partial \Omega})$ be a real-valued function such that
  $|\vartheta(x)| \neq 1$ for all $x \in \partial \Omega$, and let $A_\vartheta$ be defined by~\eqref{def_A_tau_domain}.
  Moreover, let $l \in \mathbb{N}$ and, if $l>2$, assume that $\Omega$ has a $C^l$-smooth boundary.
  Then 
  \begin{equation*}
    (A_\vartheta - \lambda)^{-l} - (T_{\textup{MIT}} - \lambda)^{-l} \in \mathfrak{S}_{2/l, \infty}\bigl(L^2(\Omega; \mathbb{C}^4)\bigr)
  \end{equation*}
  holds for all $\lambda \in \mathbb{C} \setminus \mathbb{R}$.
\end{thm}

\begin{proof}
  Let $\lambda \in \mathbb{C} \setminus \mathbb{R}$ be fixed.
  By Proposition~\ref{proposition_gamma_Weyl_quasi_domain} we have
  \begin{equation*} 
    \gamma(\lambda) = 
        \Phi_{\lambda, 1/2} \left( \frac{1}{2} \beta + \mathcal{C}_{\lambda, 1/2} \right)^{-1}.
  \end{equation*}
  Hence, using Proposition~\ref{proposition_Phi_lambda} and Proposition~\ref{proposition_C_lambda_inverse} we find
  \begin{equation}\label{os240}
    \overline{\gamma(\lambda)} =  
        \Phi_\lambda \left( \frac{1}{2} \beta + \mathcal{C}_\lambda \right)^{-1}.
  \end{equation}
  In a similar way one gets
  \begin{equation} \label{closure_M}
    \overline{M(\lambda)} = - P_+ \left( \frac{1}{2} \beta + \mathcal{C}_\lambda \right)^{-1} P_+.
  \end{equation}
  With the resolvent formula from Theorem~\ref{theorem_self_adjoint_noncritical} and \eqref{product_rule} we obtain
  \begin{equation} \label{resolvent_power_difference}
    \begin{split}
      (A_\vartheta - \lambda)^{-l} - (&T_{\textup{MIT}} - \lambda)^{-l} 
          = \frac{1}{(l-1)!} \frac{\text{d}^{l-1}}{\text{d} \lambda^{l-1}} 
          \big( (A_\vartheta - \lambda)^{-1} - (T_{\textup{MIT}} - \lambda)^{-1}  \big) \\
      &= \frac{1}{(l-1)!} \frac{\text{d}^{l-1}}{\text{d} \lambda^{l-1}} 
          \big[ \overline{\gamma(\lambda)} \big( \vartheta - \overline{M(\lambda)} \big)^{-1} \gamma(\overline{\lambda})^*  \big] \\
      &= \sum_{p + q + r = l-1} \frac{1}{p! q! r!} \frac{\text{d}^p }{\text{d} \lambda^p} \overline{\gamma(\lambda)}
          \frac{\text{d}^q}{\text{d} \lambda^q} \big( \vartheta - \overline{M(\lambda)} \big)^{-1} 
          \frac{\text{d}^r}{\text{d} \lambda^r} \gamma(\overline{\lambda})^*.
    \end{split}
  \end{equation}
  
  We are going to study now all the terms on the right hand side of~\eqref{resolvent_power_difference} and show that they belong to certain Schatten-von Neumann ideals. 
  For this purpose we claim that
  \begin{equation} \label{equation_Schatten_inverse} 
    \frac{\text{d}^k}{\text{d} \lambda^k} \left( \frac{1}{2} \beta + \mathcal{C}_\lambda \right)^{-1} 
    \in \mathfrak{S}_{2/k, \infty}\bigl(L^2(\partial\Omega; \mathbb{C}^4)\bigr)
  \end{equation}
  for $k \in \{ 1, \dots, l-1 \}$. This will be shown by induction.
  First, for $k = 1$ we have by~\eqref{chain_rule}
  \begin{equation*} 
    \frac{\text{d}}{\text{d} \lambda} \left( \frac{1}{2} \beta + \mathcal{C}_\lambda \right)^{-1}  
      = -\left( \frac{1}{2} \beta + \mathcal{C}_\lambda \right)^{-1} 
        \frac{\text{d}}{\text{d} \lambda} \mathcal{C}_\lambda \left( \frac{1}{2} \beta + \mathcal{C}_\lambda \right)^{-1}.
  \end{equation*}
  Hence, the statement for $k=1$ holds by Lemma~\ref{lemma_Schatten_von_Neumann}
  and Proposition~\ref{proposition_C_lambda_inverse}. 
  Let us assume now that the statement holds for $k=1, \dots, q$ with $q<l-1$. With the aid of~\eqref{product_rule} we get
  \begin{equation*}
    \begin{split}
      \frac{\text{d}^{q+1}}{\text{d} \lambda^{q+1}} &\left(\frac{1}{2} \beta + \mathcal{C}_\lambda \right)^{-1}
          = \frac{\text{d}^q}{\text{d} \lambda^q} \left[ 
          \frac{\text{d}}{\text{d} \lambda} \left(\frac{1}{2} \beta + \mathcal{C}_\lambda \right)^{-1} \right] \\
      &= -\frac{\text{d}^q}{\text{d} \lambda^q} \left[ 
          \left(\frac{1}{2} \beta + \mathcal{C}_\lambda \right)^{-1} \frac{\text{d}}{\text{d} \lambda} \mathcal{C}_\lambda
          \left(\frac{1}{2} \beta + \mathcal{C}_\lambda \right)^{-1} \right] \\
      &= -\sum_{k+m+n = q} \frac{q!}{k! m! n!} \frac{\text{d}^k }{\text{d} \lambda^k} \left(\frac{1}{2} \beta + \mathcal{C}_\lambda \right)^{-1} \frac{\text{d}^{m+1}}{\text{d} \lambda^{m+1}}  \mathcal{C}_\lambda  
          \frac{\text{d}^n}{\text{d} \lambda^n} \left(\frac{1}{2} \beta + \mathcal{C}_\lambda \right)^{-1}.
    \end{split}
  \end{equation*}
  Now Lemma~\ref{lemma_Schatten_von_Neumann}, the induction hypothesis, and~\eqref{equation_product_Schatten_von_Neumann} show that $\frac{\text{d}^{q+1}}{\text{d} \lambda^{q+1}} \left(\frac{1}{2} \beta + \mathcal{C}_\lambda \right)^{-1}$ belongs to $\mathfrak{S}_{2/(q+1), \infty}$, and \eqref{equation_Schatten_inverse} is proved.
  
  Thanks to \eqref{closure_M} it is now easy to see that \eqref{equation_Schatten_inverse} implies 
  \begin{equation} \label{Schatten_M_domain}
    \frac{\text{d}^k}{\text{d} \lambda^k} \overline{M(\lambda)} \in \mathfrak{S}_{2/k, \infty}(\mathcal{G}_\Omega).
  \end{equation}
  Similarly, using \eqref{os240},~\eqref{product_rule}, Lemma~\ref{lemma_Schatten_von_Neumann}, and \eqref{equation_product_Schatten_von_Neumann} we obtain
  \begin{equation*} 
    \begin{split}
      \frac{\text{d}^k}{\text{d} \lambda^k} \overline{\gamma(\lambda)} 
          &= \sum_{s+t=k} \frac{k!}{s! t!} 
        \frac{\text{d}^s}{\text{d} \lambda^s} \Phi_{\lambda}
        \frac{\text{d}^t}{\text{d} \lambda^t} \left( \frac{1}{2} \beta + \mathcal{C}_\lambda \right)^{-1}
    \end{split}
  \end{equation*}
  and hence
  \begin{equation} \label{Schatten_gamma} 
    \begin{split}
      \frac{\text{d}^k}{\text{d} \lambda^k} \overline{\gamma(\lambda)} 
      \in \mathfrak{S}_{4/(2 k+1), \infty}\bigl(\mathcal{G}_\Omega,L^2(\partial\Omega; \mathbb{C}^4)\bigr).
    \end{split}
  \end{equation}
  By taking adjoints, this implies that also 
  \begin{equation}\label{os240-a}
  \frac{\text{d}^k}{\text{d} \lambda^k} \gamma(\overline{\lambda})^* \in \mathfrak{S}_{4/(2 k+1), \infty}\bigl(L^2(\partial\Omega; \mathbb{C}^4),\mathcal G_\Omega\bigr).
  \end{equation}
  Note that~\eqref{Schatten_M_domain} yields 
  \begin{equation} \label{Schatten_M_inv}
    \frac{\text{d}^k}{\text{d} \lambda^k} \big( \vartheta - \overline{M(\lambda)} \big)^{-1}
    \in \mathfrak{S}_{2/k, \infty}(\mathcal{G}_\Omega);
  \end{equation}  
  this can be shown in the same way as~\eqref{equation_Schatten_inverse}.
  Thus, using~\eqref{resolvent_power_difference}, \eqref{Schatten_gamma}, \eqref{os240-a}, \eqref{Schatten_M_inv}, and ~\eqref{equation_product_Schatten_von_Neumann},
  we finally get that 
  \begin{equation*}
    (A_\vartheta - \lambda)^{-l} - (T_{\textup{MIT}} - \lambda)^{-l}  \in \mathfrak{S}_{2/l, \infty}\bigl(L^2(\partial\Omega; \mathbb{C}^4)\bigr),
  \end{equation*}
  which is the claimed result.
\end{proof}

In the following corollary we discuss the special case $l=3$ in Theorem~\ref{theorem_Schatten_von_Neumann_domain}. Then the difference of the 
third powers of the resolvents of $A_\vartheta$ and $T_\text{MIT}$ belongs to the trace class ideal.
By \cite[Chapter~0, Theorem~8.2]{Y10} or \cite[Problem~25]{RS79} 
this implies that the wave operators for the scattering pair $\{ A_\vartheta, T_{\textup{MIT}} \}$ exist and are complete, and hence 
the absolutely continuous parts of $A_\vartheta$ and $T_{\textup{MIT}}$ are unitarily equivalent. 
Moreover, we state an explicit formula for the trace of $(A_\vartheta - \lambda)^{-3} - (T_{\textup{MIT}} - \lambda)^{-3}$ in terms of the Weyl 
function $M$; this formula can be shown in exactly the same way as in \cite[Theorem~4.6]{BEHL18}.

\begin{cor} \label{corollary_resolvent_power_difference}
  Assume that $\Omega$ has a $C^3$-smooth boundary,
  let $a \in (\frac{1}{2}, 1 ]$, and let $\vartheta \in \textup{Lip}_a({\partial \Omega})$ be a real-valued function such that
  $|\vartheta(x)| \neq 1$ for all $x \in \partial \Omega$.
  Let $A_\vartheta$ be defined by~\eqref{def_A_tau_domain} and let $T_{\textup{MIT}}$ be the MIT bag operator in \eqref{def_MIT_op}. 
  Then the operator $(A_\vartheta - \lambda)^{-3} - (T_{\textup{MIT}} - \lambda)^{-3}$
  belongs to the trace class ideal and 
  \begin{equation*} 
    \begin{split}
      \textup{tr} \big[ (A_\vartheta - \lambda)^{-3} &- (T_{\textup{MIT}} - \lambda)^{-3} \big]
        = \frac{1}{2} \textup{tr} \left[ \frac{\textup{d}^2}{\textup{d} \lambda^2} \left( 
        \big( \vartheta - \overline{M(\lambda)} \big)^{-1} 
          \frac{\textup{d}}{\textup{d} \lambda} \overline{M(\lambda)} \right) \right]
    \end{split}
  \end{equation*}
  holds for all $\lambda \in \mathbb{C} \setminus \mathbb{R}$.
  Moreover, the wave operators for the scattering system 
  $\{ A_\vartheta, T_{\textup{MIT}} \}$ exist and are complete, and the absolutely continuous parts of 
  $A_\vartheta$ and $T_{\textup{MIT}}$ are unitarily equivalent. 
\end{cor}

\subsection{Self-adjointness and spectral properties of $A_{[\omega]}$}\label{ss bdy MIT}
To complement the class of boundary conditions \eqref{boundary_condition} discussed in the previous section
now boundary conditions of the form
\begin{equation}\label{boundary conditions tau 2}
   P_+ f|_{\partial \Omega} = \omega P_+ \beta f|_{\partial \Omega}
\end{equation}
will be treated; here $\omega: \partial \Omega \rightarrow \mathbb{R}$ is H\"older continuous of order $a > \frac{1}{2}$, as before. 
In particular, \eqref{boundary conditions tau 2} for 
$\omega\equiv0$ leads to the MIT bag operator $T_{\textup{MIT}}$ introduced in \eqref{def_MIT_op}. Of course, if $\omega$ is invertible, then
\eqref{boundary_condition} and \eqref{boundary conditions tau 2} are equivalent by setting $\vartheta=\omega^{-1}$, but if $\omega = 0$ on some parts of 
$\partial \Omega$, then this correspondence is only formal.

More precisely, let  $\{ \mathcal{G}_\Omega, \Gamma_0, \Gamma_1 \}$ be the quasi boundary triple 
from Theorem~\ref{theorem_quasi_triple_domain} and assume that $\omega: \partial \Omega \rightarrow \mathbb{R}$ is H\"older continuous of order $a > \frac{1}{2}$. 
Then the Dirac operator $A_{[\omega]}$ acting in $L^2(\Omega; \mathbb{C}^4)$ with boundary
conditions~\eqref{boundary conditions tau 2} is defined by 
\begin{equation} \label{definition_A_omega}
  \begin{split}
    A_{[\omega]} f &:= (-i \alpha \cdot \nabla + m \beta) f, \\
    \dom A_{[\omega]} &:= \{ f \in H^1(\Omega; \mathbb{C}^4): 
        \Gamma_0 f = \omega \Gamma_1 f \},
  \end{split}
\end{equation}
i.e. \eqref{boundary conditions tau 2} corresponds to the abstract boundary conditions
$\Gamma_0 f - \omega \Gamma_1 f = 0$, see also~\eqref{def_A_B_abstract}. 
Employing Theorem~\ref{theorem_krein_abstract_B} instead of Theorem~\ref{theorem_krein_abstract} one can show in a similar
manner as in Theorem~\ref{theorem_self_adjoint_noncritical} the following result on the self-adjointness of $A_{[\omega]}$:

\begin{thm}\label{theorem_A_omega_self_adjoint}
  Let $a \in (\frac{1}{2}, 1 ]$ and let $\omega \in \textup{Lip}_a({\partial \Omega})$ be a real-valued function such that
  $|\omega(x)| \neq 1$ for all $x \in \partial \Omega$.   
  Moreover, let $\gamma$ and $M$ be as in Proposition~\ref{proposition_gamma_Weyl_quasi_domain}.
  Then the operator $A_{[\omega]}$ in \eqref{definition_A_omega} is self-adjoint in $L^2(\Omega; \mathbb{C}^4)$ and the resolvent formula
  \begin{equation*}
    (A_{[\omega]}-\lambda)^{-1} = \big(T_{\textup{MIT}} - \lambda\big)^{-1} 
        + \gamma(\lambda) \big( I_4 - \omega M(\lambda) \big)^{-1} \omega \gamma(\overline{\lambda})^*
  \end{equation*}
  holds for all $\lambda \in \rho(A_{\vartheta})\cap\rho(T_{\textup{MIT}})$.
\end{thm}

Similarly to Section~\ref{section_def_op}, one can prove now several results about 
the spectral properties of $A_{[\omega]}$. The following assertions follow from Theorem~\ref{theorem_krein_abstract_B} and the Krein 
type resolvent formula from Theorem~\ref{theorem_A_omega_self_adjoint} in the same way as in
Theorem~\ref{theorem_basic_spectral_properties_unbounded_domain}.

\begin{thm} \label{thm_basic_spectral_properties_unbounded_domain_A_omega}
  Let $\Omega$ be the complement of a bounded $C^2$-domain,
  let $a \in (\frac{1}{2}, 1]$, let $\omega \in \textup{Lip}_a({\partial \Omega})$ be a real-valued function such that
  $|\omega(x)| \neq 1$ for all $x \in \partial \Omega$, and let $A_{[\omega]}$ be defined by~\eqref{definition_A_omega}.
  Then the following statements hold:
  \begin{itemize}
    \item[$\textup{(i)}$] $\sigma_{\textup{ess}}(A_{[\omega]}) = (-\infty, - m] \cup [m, \infty)$.
    \item[$\textup{(ii)}$] The number of discrete eigenvalues of $A_{[\omega]}$ is finite.
    \item[$\textup{(iii)}$] $\lambda \in \sigma_{\textup{p}}(A_{[\omega]}) \cap (-m, m)$ if and only if $1 \in \sigma_\textup{p}(\omega M(\lambda))$.
  \end{itemize}
\end{thm}

Furthermore, like in Proposition~\ref{proposition_basic_spectral_properties_bounded_domain}, one can use the Weyl function $M$ also to detect all eigenvalues of $A_{[\omega]}$ in the case that $\Omega$ is a bounded domain. Here one has to use Proposition~\ref{proposition_Birman_Schwinger_simple_B} instead of Proposition~\ref{proposition_Birman_Schwinger_simple} to obtain the following result: 

\begin{prop} \label{proposition_basic_spectral_properties_bounded_domain_A_omega}
  Let $\Omega$ be a bounded $C^2$-smooth domain,
  let $a \in (\frac{1}{2}, 1 ]$, let $\omega \in \textup{Lip}_a({\partial \Omega})$ be real-valued such that
  $|\omega(x)| \neq 1$ for all $x \in \partial \Omega$, and let $A_{[\omega]}$ be defined by~\eqref{definition_A_omega}.
  Then $\sigma(A_{[\omega]}) = \sigma_{\textup{disc}}(A_{[\omega]})$ and $\lambda$ is an eigenvalue of $A_{[\omega]}$
  if and only if there exists $\varphi \in \mathcal{G}_\Omega^{1/2}$ such that
  \begin{equation*}
    \lim_{\varepsilon \searrow 0} i \varepsilon M(\lambda + i \varepsilon) \big( I_4 - \omega M(\lambda + i \varepsilon) \big)^{-1} \varphi \neq 0.
  \end{equation*}
\end{prop}

Finally, also the proof of Theorem~\ref{theorem_Schatten_von_Neumann_domain} can be adapted in a straightforward way to obtain a similar result 
for $A_{[\omega]}$. A summary of the counterpart of Theorem~\ref{theorem_Schatten_von_Neumann_domain} and Corollary~\ref{corollary_resolvent_power_difference} 
reads as follows:

\begin{thm} \label{theorem_Schatten_von_Neumann_domain_A_omega}
  Let $\Omega$ be a $C^2$-domain with compact boundary,
  let $T_{\textup{MIT}}$ be the MIT bag operator in \eqref{def_MIT_op}, let $a \in (\frac{1}{2}, 1 ]$, let $\omega \in \textup{Lip}_a({\partial \Omega})$ be a real-valued function such that
  $|\omega(x)| \neq 1$ for all $x \in \partial \Omega$,
  and let $A_{[\omega]}$ be defined by~\eqref{definition_A_omega}.
  Moreover, let $l \in \mathbb{N}$ and, if $l>2$, assume that $\Omega$ has a $C^l$-smooth boundary.
  Then 
  \begin{equation}\label{os240-b}
    (A_{[\omega]} - \lambda)^{-l} - (T_{\textup{MIT}} - \lambda)^{-l} \in \mathfrak{S}_{2/l, \infty}\bigl(L^2(\Omega; \mathbb{C}^4)\bigr)
  \end{equation}
  holds for all $\lambda \in \mathbb{C} \setminus \mathbb{R}$.
  In particular, for $l=3$ the operator in \eqref{os240-b} 
  belongs to the trace class ideal and 
  \begin{equation*} 
    \begin{split}
      \textup{tr} \big[ (A_{[\omega]} - \lambda)^{-3} &- (T_{\textup{MIT}} - \lambda)^{-3} \big]
        = \frac{1}{2} \textup{tr} \left[ \frac{\textup{d}^2}{\textup{d} \lambda^2} \left( 
        \big( I_4 - \omega M(\lambda) \big)^{-1} \omega
          \frac{\textup{d}}{\textup{d} \lambda} M(\lambda) \right) \right]
    \end{split}
  \end{equation*}
  holds for all $\lambda \in \mathbb{C} \setminus \mathbb{R}$.
  Moreover, the wave operators for the scattering system 
  $\{ A_{[\omega]}, T_{\textup{MIT}} \}$ exist and are complete, and the absolutely continuous parts of 
  $A_{[\omega]}$ and $T_{\textup{MIT}}$ are unitarily equivalent. 
\end{thm}

\subsection{On the connection of $A_\vartheta$, $A_{[\omega]}$, and Dirac operators with $\delta$-shell interactions}
\label{section_confinement}

In this section we assume that $\Omega_+ \subset \mathbb{R}^3$ is a bounded $C^2$-domain, we set 
$\Omega_- := \mathbb{R}^3 \setminus \overline{\Omega_+}$, and $\Sigma := \partial \Omega_\pm$.
By $\nu_\pm$ we denote the unit normal vector field pointing outwards of $\Omega_\pm$ and 
for functions $f \in L^2(\mathbb{R}^3; \mathbb{C}^4)$ we will use the notation
$f_\pm := f \upharpoonright \Omega_\pm$. 
Assume that $a \in (\frac{1}{2}, 1 ]$, let
$\eta, \tau \in \textup{Lip}_a({\Sigma})$ be real-valued functions on $\Sigma$, and consider
 the formal differential expression 
 $$-i \alpha \cdot \nabla + m \beta + (\eta I_4 + \tau \beta) \delta_\Sigma,$$
 where $\delta_\Sigma$ stands for the $\delta$-distribution supported on the interface $\Sigma$.
 In analogy to the case of constant interaction strengths in \cite[Section 3]{BEHL19_1} we introduce the associated Dirac operator in
$L^2(\mathbb{R}^3; \mathbb{C}^4)$ as
\begin{equation} \label{def_delta_op_main}
  \begin{split}
    B_{\eta, \tau} f &:= (-i \alpha \cdot \nabla + m \beta) f_+ \oplus
        (-i \alpha \cdot \nabla + m \beta) f_-, \\
    \dom B_{\eta, \tau} 
        &:= \bigg\{ f = f_+ \oplus f_- \in H^1(\Omega_+; \mathbb{C}^4) \oplus H^1(\Omega_-; \mathbb{C}^4): \\
    &\qquad  i (\alpha \cdot \nu_+) (f_+|_{\Sigma} - f_-|_{\Sigma}) 
        = -\frac{1}{2} (\eta I_4 + \tau \beta) (f_+|_{\Sigma} + f_-|_{\Sigma}) \bigg\}.
  \end{split}
\end{equation}
Note that so far singularly perturbed Dirac operators of this form with electrostatic and Lorentz scalar $\delta$-shell interactions have been studied only for constant
coefficients $\eta,\tau\in\mathbb R$, see, e.g., \cite{AMV14, AMV15, AMV16, BEHL18, BEHL19_1, BH17, BHOP19, HOP17, M17, MP18a, MP18b, OP19, OV17}.
In particular, it is known that for constant $\eta, \tau \in \mathbb{R}$ such that
$\eta^2 - \tau^2 = -4$ the operator $B_{\eta, \tau}$
decouples into two self-adjoint operators acting in $L^2(\Omega_\pm; \mathbb{C}^4)$ with certain boundary conditions; cf. \cite[Lemma~3.1~(ii)]{BEHL19_1} 
 and 
\cite[Section~V]{DES89}, \cite[Section~5]{AMV15}.
This phenomenon is referred to as {\it confinement}, since a particle which is located in 
$\Omega_\pm$ will stay in $\Omega_\pm$ for all times; in other words, the $\delta$-potential is impenetrable.

Define the projections
\begin{equation*}
  P_\pm^{\Omega_+} := \frac{1}{2} \big( I_4 \pm i \beta (\alpha \cdot \nu_+)\big) \quad \text{and} \quad
  P_\pm^{\Omega_-} := \frac{1}{2} \big( I_4 \pm i \beta (\alpha \cdot \nu_-)\big).
\end{equation*}
Moreover, for $a \in (\frac{1}{2}, 1 ]$ and real-valued $\vartheta\in \textup{Lip}_a({\Sigma})$ denote by 
$A_\vartheta^{\Omega_\pm}$ 
the self-adjoint operators defined in \eqref{def_A_tau_domain} acting in $L^2(\Omega_\pm; \mathbb{C}^4)$.
Our aim is to show that
\begin{equation}\label{schoen}
B_{\eta, \tau}=A_\vartheta^{\Omega_+} \oplus A_\vartheta^{\Omega_-}
\end{equation}
for suitable $\vartheta,\eta,\tau\in \textup{Lip}_a({\Sigma})$. With the help of this identity one can translate (under the appropriate assumptions on the 
interaction strengths $\vartheta,\eta,\tau$) results for Dirac operators with $\delta$-interactions to the operators $A_\vartheta^{\Omega_\pm}$
studied in this paper, and vice versa; cf. Lemma~\ref{lem_domain_confinement} and Theorem~\ref{finalthm} below for an illustration. 
The following preparatory lemma will be useful.

\begin{lem}\label{gutzuhaben}
  Let $\eta$ and $\tau$ be real-valued functions on $\Sigma$, assume that the identity $\eta^2(x) - \tau^2(x) = -4$ holds for all $x\in\Sigma$, 
  and let $f_\pm\in H^1(\Omega_\pm; \mathbb{C}^4)$. 
  Then the jump condition
  \begin{equation} \label{jussi1}
    i (\alpha \cdot \nu_+) (f_+|_\Sigma - f_-|_\Sigma) 
        = -\frac{1}{2} (\eta I_4 + \tau \beta) (f_+|_\Sigma + f_-|_\Sigma)
  \end{equation}
  is equivalent to the boundary conditions
 \begin{equation} \label{jussi2}
    \big( (2 + \tau) \beta - \eta I_4 \big) P_+^{\Omega_\pm} f_\pm|_{\Sigma} 
      = -\big( (2 - \tau) I_4 + \eta \beta \big) P_+^{\Omega_\pm} \beta f_\pm|_{\Sigma}.
  \end{equation}
  \end{lem}

 \begin{proof}
 Observe first that separating terms  with $f_\pm$ in \eqref{jussi1} leads to 
  \begin{equation} \label{jump_condition_delta_opjussi3}
    \left( i (\alpha \cdot \nu_+) + \frac{1}{2} (\eta I_4 + \tau \beta) \right) f_+|_\Sigma  
        = \left( i (\alpha \cdot \nu_+) -\frac{1}{2} (\eta I_4 + \tau \beta) \right) f_-|_\Sigma.
  \end{equation}
  We multiply \eqref{jump_condition_delta_opjussi3}
  by $\pm i (\alpha \cdot \nu_+) + \frac{1}{2}(\eta I_4 - \tau \beta)$ and  use $\beta^2 = I_4$, $\eta^2 - \tau^2 = -4$, \eqref{eq_commutation}, which implies
  $$i (\alpha \cdot \nu_+) (\eta I_4 + \tau \beta)= (\eta I_4 - \tau \beta)i (\alpha \cdot \nu_+),$$ 
  and  
  $\nu_- = - \nu_+$, and arrive at the following equivalent form of \eqref{jussi1}:
  \begin{equation} \label{boundary_condition_delta_op}
    \big( 2 I_4 - (\eta I_4 - \tau \beta) i(\alpha \cdot \nu_\pm)  \big) f_\pm|_\Sigma = 0.
  \end{equation}
  From $\beta^2 = I_4$ and~\eqref{beta_P_pm} we see that 
    $I_4 = P_+^{\Omega_\pm} + P_-^{\Omega_\pm} = P_+^{\Omega_\pm} + \beta P_+^{\Omega_\pm} \beta$,
  and thus \eqref{boundary_condition_delta_op}
  is equivalent to
  \begin{equation*}
    \big( 2 I_4 - (\eta I_4 - \tau \beta) i(\alpha \cdot \nu_\pm)  \big) P_+^{\Omega_\pm} f_\pm|_{\Sigma} 
      = -\big( 2 I_4 - (\eta I_4 - \tau \beta) i(\alpha \cdot \nu_\pm)  \big) \beta P_+^{\Omega_\pm} \beta f_\pm|_{\Sigma}.
  \end{equation*}
  Multiplying both sides by $\beta$ and using 
  $i \beta (\alpha \cdot \nu_\pm) P_+^{\Omega_\pm} = P_+^{\Omega_\pm}$  we conclude that the jump condition \eqref{jussi1} is equivalent to the boundary condition \eqref{jussi2}. 
 \end{proof}
 
The next proposition provides conditions on  $\vartheta,\eta,\tau\in \textup{Lip}_a({\Sigma})$ such that \eqref{schoen} holds.

\begin{prop} \label{proposition_domain_confinement}
  Let $a \in (\frac{1}{2}, 1]$ and let $\eta, \tau, \vartheta$ be real-valued functions on $\Sigma$. 
  Then the following statements hold: 
  \begin{itemize}
    \item[$\textup{(i)}$] Assume that $\eta,\tau\in \textup{Lip}_a(\Sigma)$, $\eta(x)^2 - \tau(x)^2 = -4$, and $\tau(x)\neq2$ for all $x \in \Sigma$,
    and define 
    \begin{equation}\label{nimmdas}
    \vartheta = \frac{\eta}{2 - \tau}.
    \end{equation}
    Then $\vartheta\in \textup{Lip}_a(\Sigma)$, $\vert \vartheta(x)\vert \neq 1$ for all $x\in\Sigma$, and \eqref{schoen} holds.
    \item[$\textup{(ii)}$] Assume that $\vartheta\in \textup{Lip}_a(\Sigma)$, $\vert \vartheta(x)\vert \neq 1$
    for all $x \in \Sigma$, and define
    \begin{equation}\label{formulas eta in terms of tau}
      \eta = \frac{4 \vartheta}{1 - \vartheta^2} \quad \text{and} \quad \tau = \frac{2 (1 + \vartheta^2)}{\vartheta^2 - 1}.
    \end{equation}
    Then $\eta,\tau\in \textup{Lip}_a(\Sigma)$, $\eta^2(x) - \tau^2(x) = -4$ for all $x\in\Sigma$, and \eqref{schoen} holds. 
  \end{itemize}
 In particular, in both situations \textup{(i)} and \textup{(ii)} the operator $B_{\eta,\tau}$ in \eqref{def_delta_op_main} is self-adjoint in $L^2(\mathbb{R}^3; \mathbb{C}^4)$. 
\end{prop}

\begin{proof}
The assertion \eqref{schoen} in  (i) and (ii) follows from Lemma~\ref{gutzuhaben}; the remaining assertions on $\eta, \tau, \vartheta$ in  (i) and (ii) 
are easy to check. 
In fact, to verify \eqref{schoen} in item (i) we multiply \eqref{jussi2} by 
the matrix $( (2 - \tau)I_4 - \eta \beta)$, which is invertible by our assumptions on the functions $\eta$ and $\tau$. Using 
$\eta^2 - \tau^2 = -4$ we obtain for the left-hand side of \eqref{jussi2}
  \begin{equation*}
    \big( (2 - \tau)I_4 - \eta \beta \big) \big( (2 + \tau) \beta - \eta I_4 \big)P_+^{\Omega_\pm} f_\pm|_{\Sigma} 
        = -4 \eta P_+^{\Omega_\pm} f_\pm|_{\Sigma} 
  \end{equation*}
  and for the right-hand side of \eqref{jussi2}
  \begin{equation*}
    -\big( (2 - \tau)I_4 - \eta \beta \big) \big( (2 - \tau)I_4 + \eta \beta \big)P_+^{\Omega_\pm} \beta f_\pm|_{\Sigma}
        = -(8-4 \tau) P_+^{\Omega_\pm} \beta f_\pm|_{\Sigma}.
  \end{equation*}
  Therefore, if $\{ \mathcal{G}_{\Omega_\pm}, \Gamma_0^{\Omega_\pm}, \Gamma_1^{\Omega_\pm} \}$ denote the quasi boundary triples 
  from Theorem~\ref{theorem_quasi_triple_domain} then \eqref{jussi2} is equivalent to 
  \begin{equation} \label{boundary_condition_confinement1}
    \frac{\eta}{2 - \tau} \Gamma_0^{\Omega_\pm} f_\pm = \Gamma_1^{\Omega_\pm} f_\pm.
  \end{equation}
With $\vartheta$ in \eqref{nimmdas} we now 
conclude from Lemma~\ref{gutzuhaben} that $f=f_+\oplus f_-\in\dom B_{\eta,\tau}$ if and only if $f_\pm\in\dom A_\vartheta^{\Omega_\pm}$, that is, 
the identity \eqref{schoen} is valid. 

To show \eqref{schoen} in item (ii) note first that \eqref{formulas eta in terms of tau} yields $\vartheta = \frac{\eta}{2 - \tau}$. The above observation that
\eqref{boundary_condition_confinement1} is equivalent to \eqref{jussi2}, and hence equivalent to \eqref{jussi1} by Lemma~\ref{gutzuhaben}, implies that 
$f_\pm\in\dom A_\vartheta^{\Omega_\pm}$ if and only if $f=f_+\oplus f_-\in\dom B_{\eta,\tau}$. Hence \eqref{schoen} holds.

Finally, note that under the assumptions on $\eta, \tau, \vartheta$ in  (i) and (ii) 
the operators $A_\vartheta^{\Omega_\pm}$ are both self-adjoint in $L^2(\Omega_\pm; \mathbb{C}^4)$ by Theorem~\ref{theorem_self_adjoint_noncritical}.
Hence it follows from \eqref{schoen} that the operator $B_{\eta,\tau}$ in \eqref{def_delta_op_main} is self-adjoint in $L^2(\mathbb{R}^3; \mathbb{C}^4)$. 
\end{proof}

For $a \in (\frac{1}{2}, 1 ]$ and real-valued $\omega\in \textup{Lip}_a({\Sigma})$ denote by 
$A_{[\omega]}^{\Omega_\pm}$ 
the self-adjoint operators defined in \eqref{definition_A_omega} acting in $L^2(\Omega_\pm; \mathbb{C}^4)$.
Now we verify 
\begin{equation}\label{schoen2}
B_{\eta, \tau}=A_{[\omega]}^{\Omega_+} \oplus A_{[\omega]}^{\Omega_-}
\end{equation}
for suitable $\omega,\eta,\tau\in \textup{Lip}_a({\Sigma})$, which is the counterpart of the identity \eqref{schoen}. As above one may use \eqref{schoen2}
to translate results for Dirac operators with $\delta$-interactions to the operators $A_{[\omega]}^{\Omega_\pm}$, and vice versa. The next proposition
provides the necessary relations between the functions $\omega,\eta,\tau$. The proof follows the same strategy as the proof of 
Proposition~\ref{proposition_domain_confinement}.

\begin{prop} \label{proposition_domain_confinement_A_omega}
  Let $a \in (\frac{1}{2}, 1]$ and let $\eta, \tau, \omega$ be real-valued functions on $\Sigma$. 
  Then the following statements hold:
  \begin{itemize}
    \item[$\textup{(i)}$] Assume that $\eta,\tau\in \textup{Lip}_a(\Sigma)$, $\eta(x)^2 - \tau(x)^2 = -4$, and $\tau(x)\neq -2$ for all $x \in \Sigma$,
    and define 
    \begin{equation}\label{nimmdas2}
    \omega = - \frac{\eta}{2 + \tau}.
    \end{equation}
    Then $\omega\in \textup{Lip}_a(\Sigma)$, $\vert \omega(x)\vert\neq 1$ for all $x\in\Sigma$, and \eqref{schoen2} holds.
    \item[$\textup{(ii)}$] Assume that $\omega\in \textup{Lip}_a(\Sigma)$, $\vert \omega(x)\vert  \neq 1$
    for all $x \in \Sigma$, and define
    \begin{equation}\label{formulas eta in terms of tau22}
      \eta = \frac{4 \omega}{\omega^2-1} \quad \text{and} \quad \tau = \frac{2 (1 + \omega^2)}{1-\omega^2}.
    \end{equation}
    Then $\eta,\tau\in \textup{Lip}_a(\Sigma)$, $\eta^2(x) - \tau^2(x) = -4$ for all $x\in\Sigma$, and \eqref{schoen2} holds. 
  \end{itemize}
 In particular, in both situations \textup{(i)} and \textup{(ii)} the operator $B_{\eta,\tau}$ in \eqref{def_delta_op_main} is self-adjoint in $L^2(\mathbb{R}^3; \mathbb{C}^4)$. 
\end{prop}

\begin{proof}
The assertion \eqref{schoen2} in  (i) and (ii) follows from Lemma~\ref{gutzuhaben}; the remaining assertions on $\eta, \tau, \omega$ in  (i) and (ii) 
are easy to check. 
In fact, to verify \eqref{schoen2} in item (i) we shall multiply the identity \eqref{jussi2} by 
\begin{equation*} 
    \big( (2 + \tau) \beta - \eta I_4 \big)^{-1}
    =\frac{1}{4(2 + \tau)}\big( (2 + \tau) \beta + \eta I_4 \big),
  \end{equation*}
where $\eta^2 - \tau^2 = -4$ was used. It follows that \eqref{jussi2} is equivalent to 
  \begin{equation*} 
  \begin{split}
  P_+^{\Omega_\pm} f_\pm|_{\Sigma} 
      &= -\frac{1}{4(2 + \tau)}\big( (2 + \tau) \beta + \eta I_4 \big)\big( (2 - \tau) I_4 + \eta \beta \big) P_+^{\Omega_\pm} \beta f_\pm|_{\Sigma}\\
&=      -\frac{\eta}{2 + \tau} P_+^{\Omega_\pm} \beta f_\pm|_{\Sigma}.
\end{split}
\end{equation*}
Using the quasi boundary triples $\{ \mathcal{G}_{\Omega_\pm}, \Gamma_0^{\Omega_\pm}, \Gamma_1^{\Omega_\pm} \}$ 
from Theorem~\ref{theorem_quasi_triple_domain} we find that \eqref{jussi2} is equivalent to 
  \begin{equation*} 
     \Gamma_0^{\Omega_\pm} f_\pm =
     -\frac{\eta}{2 + \tau} \Gamma_1^{\Omega_\pm} f_\pm.
  \end{equation*}
With $\omega$ in \eqref{nimmdas2} we now 
conclude from Lemma~\ref{gutzuhaben} that $f=f_+\oplus f_-\in\dom B_{\eta,\tau}$ if and only if $f_\pm\in\dom A_{[\omega]}^{\Omega_\pm}$, that is, 
the identity \eqref{schoen2} is valid. To show \eqref{schoen2} in (ii) one argues in the same way as in the proof of Proposition~\ref{proposition_domain_confinement}.
The details are left to the reader.
\end{proof}

A particularly interesting case in Proposition~\ref{proposition_domain_confinement_A_omega}~(ii) corresponds to 
the choice $\omega = 0$. Since $A_{[0]}^{\Omega_\pm}=T_{\textup{MIT}}^{\Omega_\pm}$ by  \eqref{definition_A_omega}
the identity \eqref{schoen2} reduces to 
\begin{equation*}
B_{0,2}=T_{\textup{MIT}}^{\Omega_+} \oplus T_{\textup{MIT}}^{\Omega_-},
\end{equation*} 
and hence identifies the orthogonal sum
of the
MIT bag operators with a Dirac operator with a purely Lorentz scalar $\delta$-shell potential; cf. \cite[Remark~2.1]{HOP17}.

Now we return to the identities \eqref{schoen} and \eqref{schoen2}, and illustrate how one can 
translate known results for Dirac operators with $\delta$-potentials to self-adjoint Dirac operators on domains studied in this paper.
Some preparation is necessary to formulate Theorem~\ref{finalthm} below. For 
a $C^2$-domain $\Omega \subset \mathbb{R}^3$ with compact boundary
we introduce the orthogonal projection
\begin{equation*}
  P_\Omega: L^2(\mathbb{R}^3; \mathbb{C}^4) \rightarrow L^2(\Omega; \mathbb{C}^4),
  \quad P_\Omega f = f \upharpoonright\Omega,
\end{equation*}
and the corresponding embedding
\begin{equation*}
  \iota_\Omega: L^2(\Omega; \mathbb{C}^4) \rightarrow L^2(\mathbb{R}^3; \mathbb{C}^4),
  \quad \iota_\Omega g = \begin{cases} g & \text{ in } \Omega, \\ 0 & \text{ in } \Omega^c. \end{cases}
\end{equation*}
Moreover, for $\lambda \in \mathbb{C} \setminus ( (-\infty, -m] \cup [m, \infty) )$ we consider the integral operator
$R_\lambda: L^2(\Omega; \mathbb{C}^4) \rightarrow L^2(\Omega; \mathbb{C}^4)$,
\begin{equation} \label{def_R_lambda}
  R_\lambda f(x) = \int_{\Omega} G_\lambda(x-y) f(y) \text{d} y,
  \quad x \in \Omega,\, f \in L^2(\Omega; \mathbb{C}^4),
\end{equation}
with $G_\lambda$ defined by~\eqref{def_G_lambda}. Note that $R_\lambda$ is the compression of the resolvent of the free 
Dirac operator $A$ in~$\mathbb{R}^3$, that is, 
\begin{equation*}
  R_\lambda = P_\Omega (A - \lambda)^{-1} \iota_\Omega;
\end{equation*}
cf.~\cite[Section~1.E]{T92}. The resolvent formulas in the next preparatory lemma are now an immediate consequence of \cite[Theorem~3.4]{BEHL19_1},
Proposition~\ref{proposition_domain_confinement}~(ii), and Proposition~\ref{proposition_domain_confinement_A_omega}~(ii).
We emphasize that in contrast to the previous discussion the coefficients are first assumed to be real constants since 
Dirac operators with electrostatic and Lorentz scalar $\delta$-shell interactions have been studied in this case only. To avoid confusion we
use a subindex here.

\begin{lem} \label{lem_domain_confinement}
  Let $\Omega \subset \mathbb{R}^3$ be a $C^2$-domain with compact boundary and let $R_\lambda$,  $\Phi_\lambda$, and $\mathcal{C}_\lambda$, 
  $\lambda\in\mathbb{C} \setminus \mathbb{R}$, be the operators 
  in~\eqref{def_R_lambda}, \eqref{def_Phi_lambda}, and~\eqref{def_C_lambda}, respectively. Then the following statements hold:
  \begin{itemize}
    \item[$\textup{(i)}$] Assume that $\vartheta_*\in\mathbb R\setminus\{\pm 1\}$ 
    and let    
    \begin{equation*}
      \eta_* = \frac{4 \vartheta_*}{1 - \vartheta_*^2} \quad \text{and} \quad 
      \tau_* = \frac{2 (1 + \vartheta_*^2)}{\vartheta_*^2 - 1}.
    \end{equation*}  
    Then the resolvent formula
    \begin{equation*}
      (A_{\vartheta_*} - \lambda)^{-1} = R_\lambda - \Phi_\lambda 
          \big( I_4 + (\eta_* I_4 + \tau_* \beta) \mathcal{C}_\lambda \big)^{-1} 
              (\eta_* I_4 + \tau_* \beta) \Phi_{\overline{\lambda}}^*
    \end{equation*}
    holds for all $\lambda \in \mathbb{C} \setminus \mathbb{R}$.

    \item[$\textup{(ii)}$] Assume that $\omega_*\in\mathbb R\setminus\{\pm 1\}$ 
    and let  
    \begin{equation*}
      \eta_* = \frac{4 \omega_*}{\omega_*^2 - 1} \quad \text{and} \quad \tau_* = \frac{2 (1 + \omega_*^2)}{1 - \omega_*^2}.
    \end{equation*}  
    Then the resolvent formula
    \begin{equation*}
      (A_{[\omega_*]} - \lambda)^{-1} = R_\lambda - \Phi_\lambda 
          \big( I_4 + (\eta_* I_4 + \tau_* \beta) \mathcal{C}_\lambda \big)^{-1} 
              (\eta_* I_4 + \tau_* \beta) \Phi_{\overline{\lambda}}^*
    \end{equation*}
    holds for all $\lambda \in \mathbb{C} \setminus \mathbb{R}$.
  \end{itemize}
\end{lem}

For the special choice $\omega_*=0$ the resolvent formula in Lemma~\ref{lem_domain_confinement}~(ii) has the form
\begin{equation}\label{mitformel}
      (T_{\textup{MIT}} - \lambda)^{-1} = R_\lambda - \Phi_\lambda 
          \big( I_4 + 2 \beta \mathcal{C}_\lambda \big)^{-1} 
              2 \beta \Phi_{\overline{\lambda}}^*,\quad\lambda\in \mathbb{C} \setminus \mathbb{R}. 
    \end{equation}
Using \eqref{mitformel} we now obtain a more explicit description of the resolvents of the self-adjoint operators $A_\vartheta$ and $A_{[\omega]}$ 
for general real-valued $\vartheta ,\omega \in \textup{Lip}_a({\partial \Omega})$ in terms of the compressed resolvent of the free Dirac operator and the operators
$\Phi_\lambda$  and $\mathcal{C}_\lambda$ in \eqref{def_Phi_lambda} and~\eqref{def_C_lambda}, respectively. 
The next theorem is an immediate consequence of the resolvent formulas in Theorem~\ref{theorem_self_adjoint_noncritical}
and Theorem~\ref{theorem_A_omega_self_adjoint} and of~\eqref{mitformel}.

\begin{thm}\label{finalthm}
Let $\Omega \subset \mathbb{R}^3$ be a $C^2$-domain with compact boundary, let $R_\lambda$,  $\Phi_\lambda$, and $\mathcal{C}_\lambda$, 
  $\lambda\in\mathbb{C} \setminus \mathbb{R}$, be the operators 
  in~\eqref{def_R_lambda}, \eqref{def_Phi_lambda}, and~\eqref{def_C_lambda}, respectively, and 
  let $\gamma$ and $M$ be as in Proposition~\ref{proposition_gamma_Weyl_quasi_domain}. Then the following assertions hold:
   \begin{itemize}
    \item[$\textup{(i)}$] If $a \in (\frac{1}{2}, 1]$ and
    $\vartheta \in \textup{Lip}_a({\partial \Omega})$ is a real-valued function such that $\vert\vartheta(x)\vert\not=1$ for all $x\in\partial\Omega$, 
    then the resolvent of the self-adjoint operator 
    $A_\vartheta$ in \eqref{def_A_tau_domain} admits the representation 
    \begin{equation*}
    (A_\vartheta-\lambda)^{-1} = R_\lambda - \Phi_\lambda 
          \big( I_4 + 2 \beta \mathcal{C}_\lambda \big)^{-1} 
              2 \beta \Phi_{\overline{\lambda}}^*
        + \gamma(\lambda) \big( \vartheta - M(\lambda) \big)^{-1} \gamma(\overline{\lambda})^*
  \end{equation*}
  for all $\lambda \in \mathbb{C} \setminus \mathbb{R}$.
    \item[$\textup{(ii)}$] If $a \in (\frac{1}{2}, 1]$ and
    $\omega \in \textup{Lip}_a({\partial \Omega})$ is a real-valued function such that $\vert\omega(x)\vert\not=1$ for all $x\in\partial\Omega$, then the resolvent of the self-adjoint operator 
    $A_{[\omega]}$ in \eqref{definition_A_omega} admits the representation
     \begin{equation*}
    (A_{[\omega]}-\lambda)^{-1} =  R_\lambda - \Phi_\lambda 
          \big( I_4 + 2 \beta \mathcal{C}_\lambda \big)^{-1} 
              2 \beta \Phi_{\overline{\lambda}}^*
        + \gamma(\lambda) \big( I_4 - \omega M(\lambda) \big)^{-1} \omega \gamma(\overline{\lambda})^*
  \end{equation*}
    for all $\lambda \in \mathbb{C} \setminus \mathbb{R}$.
      \end{itemize}
\end{thm}

Finally, we remark that the resolvent formulas above also hold for certain
$\lambda\in\mathbb R$ which belong to the resolvent sets of the involved Dirac operators. This straightforward, but slightly more technical, 
generalization is not pursued further here.

\begin{appendix}
\section{Mapping properties of integral operators between Sobolev spaces} \label{appendix_commutator}

Throughout this appendix let $\Omega$ be a $C^2$-domain in $\mathbb{R}^3$ with compact boundary~$\partial \Omega$.
The aim of this section is to provide results on integral operators acting between Sobolev spaces on the boundary $\partial \Omega$ which are applied in the main part of the paper to prove Proposition~\ref{proposition_commutator_C_lambda}.
Recall that the norm in $H^s(\partial \Omega; \mathbb{C})$ for $s \in (0,1)$ is given by
\begin{equation*}
\| \varphi \|_s^2:=\int_{{\partial \Omega}}|\varphi(x)|^2 \text{d} \sigma(x)
+\int_{{\partial \Omega}}\int_{{\partial \Omega}}
\frac{|\varphi(x)-\varphi(y)|^2}{|x-y|^{2+2s}} \text{d} \sigma(y) \text{d} \sigma(x)
\end{equation*}
for $\varphi \in H^s(\partial \Omega; \mathbb{C})$; cf.~\eqref{def_Sobolev_Slobodeckii}.
To prove the main results of this appendix we need some preliminary considerations.
First we recall a standard result on the growth of the integral of $|x-y|^{b-2}$
with respect to the surface measure $\sigma$.
A proof can be found, e.g., in \cite[Lemma~3.2~(b)]{KH15}.

\begin{lem}\label{integral estimate}
Given $b>0$ there exists $C>0$ such that
\begin{equation*}
\int_{|x-y|\leq \rho}|x-y|^{b-2} \textup{d} \sigma(y)\leq C \rho^{b}
\end{equation*}
for all $x\in{\partial \Omega}$ and $\rho > 0$. In particular, 
$\int_{\partial \Omega} |x-y|^{b-2} \textup{d} \sigma(y)\leq C$ uniformly in $x\in{\partial \Omega}$.
\end{lem}

Next, we show that the multiplication operator with a H\"older continuous
function $\vartheta\in\text{Lip}_a({\partial \Omega})$
is bounded in $H^s({\partial \Omega}; \mathbb{C})$ for any $0 \leq s < a$.

\begin{lem} \label{lemma_mult_op}
Let $0\leq s<a\leq1$ and $\vartheta\in\textup{Lip}_a({\partial \Omega})$. Then the operator given by the multiplication with $\vartheta$ is bounded in $H^s(\partial \Omega; \mathbb{C})$.
\end{lem}

\begin{proof}
Throughout the proof let $C$ be a generic constant, which changes its value several times, and let $\varphi\in H^s({\partial \Omega}; \mathbb{C})$ be fixed. In order to get the desired result we estimate in
\begin{equation} \label{Sobolev_norm_mult_op}
\|\vartheta \varphi\|_s^2 = \int_{{\partial \Omega}}|(\vartheta \varphi)(x)|^2 \text{d} \sigma(x)
+\int_{{\partial \Omega}}\int_{{\partial \Omega}}
\frac{|(\vartheta \varphi)(x) - (\vartheta \varphi)(y)|^2}{|x-y|^{2+2s}} \text{d} \sigma(y) \text{d} \sigma(x)
\end{equation}
both terms on the right hand side separately.
First, since $\vartheta\in\text{Lip}_a({\partial \Omega})$ and ${\partial \Omega}$ is bounded, we have 
$\|\vartheta\|_{L^\infty({\partial \Omega})} < \infty$. Therefore
\begin{equation}\label{prod Hs eq2}
\int_{\partial \Omega} |(\vartheta \varphi)(x)|^2 \textup{d} \sigma(x) \leq \|\vartheta\|_{L^\infty({\partial \Omega})}^2\int_{\partial \Omega} |\varphi(x)|^2 \textup{d} \sigma(x) \leq \|\vartheta\|_{L^\infty({\partial \Omega})}^2 \| \varphi \|_s^2.
\end{equation}
To find an upper bound for the second term in \eqref{Sobolev_norm_mult_op}, we note first that by $\vartheta\in\text{Lip}_a({\partial \Omega})$ 
\begin{equation*}
\begin{split}
|\vartheta(x)\varphi(x)-\vartheta(y)\varphi(y)|
&\leq|\vartheta(x)||\varphi(x)-\varphi(y)|+|\vartheta(x)-\vartheta(y)||\varphi(y)|\\
&\leq \|\vartheta\|_{L^\infty({\partial \Omega})}|\varphi(x)-\varphi(y)|+C|x-y|^a|\varphi(y)|.
\end{split}
\end{equation*}
This, Fubini's theorem, and Lemma 
\ref{integral estimate} applied with $b = 2 (a - s) > 0$ imply now
\begin{equation}\label{prod Hs eq3}
\begin{split}
\int_{\partial \Omega} \int_{\partial \Omega} &\frac{|\vartheta(x)\varphi(x)-\vartheta(y)\varphi(y)|^2}{|x-y|^{2+2s}}
 \textup{d} \sigma(y) \textup{d} \sigma(x)\\
&\leq C\int_{\partial \Omega} \int_{\partial \Omega} \frac{|\varphi(x)-\varphi(y)|^2}{|x-y|^{2+2s}}
 \textup{d} \sigma(y) \textup{d} \sigma(x) \\
&\qquad \qquad +C\int_{\partial \Omega}|\varphi(y)|^2\int_{\partial \Omega} |x-y|^{2(a-s)-2}
 \textup{d} \sigma(x) \textup{d} \sigma(y)\\
&\leq C\|\varphi\|^2_s.
\end{split}
\end{equation}
Using \eqref{prod Hs eq2} and \eqref{prod Hs eq3} in~\eqref{Sobolev_norm_mult_op}, we conclude that 
$\|\vartheta\varphi\|_s\leq C\|\varphi\|_s$ for $s<a$.
\end{proof}

Given now $0<a\leq1$ and an integral kernel $k:{\partial \Omega}\times{\partial \Omega}\to\mathbb{C}$ such that
\begin{equation}\label{kernel growth}
\begin{split}
&|k(x,y)|\leq \frac{C}{|x-y|^{2-a}}\quad\text{for all }x\neq y,\\
&|k(x,z)-k(y,z)|\leq C\frac{|x-y|^a}{|x-z|^2}
\quad\text{for all }|x-y|<\frac{1}{4}\,|x-z|,
\end{split}
\end{equation}
define
\begin{equation} \label{def_int_op}
T\varphi(x):=\int_{{\partial \Omega}}k(x,y)\varphi(y) \textup{d} \sigma(y)\quad\text{for }x\in{\partial \Omega}.
\end{equation}
It is well known that if the integral kernel satisfies~\eqref{kernel growth}, then $T$ is bounded in $L^2(\partial \Omega; \mathbb{C})$ and, in particular, the integral on the right hand side of~\eqref{def_int_op} exists a.e. on $\partial \Omega$, see, e.g., \cite[Proposition~3.10]{F95}.
In the following theorem, which is the main result of this appendix, we show that $T$ has even better mapping properties between Sobolev spaces on $\partial \Omega$:

\begin{thm}\label{thm bounded}
Let $0<s<a\leq1$. Then $T$ defined by~\eqref{def_int_op} gives rise to a bounded operator
\begin{equation*}
  T:L^2({\partial \Omega}; \mathbb{C})\to H^s({\partial \Omega}; \mathbb{C}).
\end{equation*}
\end{thm}

\begin{proof}
Throughout the proof let $C$ be a generic constant, which changes its value several times, and let $\varphi\in L^2({\partial \Omega}; \mathbb{C})$ be fixed.
Let us estimate the two terms in 
\begin{equation} \label{Sobolev_norm_int_op}
\| T \varphi\|_s^2 = \int_{{\partial \Omega}}|T \varphi(x)|^2 \text{d} \sigma(x)
+\int_{{\partial \Omega}}\int_{{\partial \Omega}}
\frac{|T \varphi(x) - T \varphi(y)|^2}{|x-y|^{2+2 s}} \text{d} \sigma(y) \text{d} \sigma(x)
\end{equation}
separately.
To control the first one we use \eqref{kernel growth}, the Cauchy-Schwarz inequality, Lemma \ref{integral estimate}, and Fubini's theorem and get
\begin{equation}\label{thm bounded control1}
\begin{split}
\int_{\partial \Omega}|T\varphi(x)|^2 \textup{d} \sigma(x)
&\leq C\int_{\partial \Omega}\bigg(\int_{\partial \Omega}\frac{|\varphi(y)|}{|x-y|^{2-a}}
 \textup{d} \sigma(y)\bigg)^2 \textup{d} \sigma(x)\\
&\leq C\int_{\partial \Omega}\bigg(\int_{\partial \Omega}
\frac{\textup{d} \sigma(y)}{|x-y|^{2-a}}
 \bigg)\bigg(\int_{\partial \Omega}\frac{|\varphi(y)|^2\textup{d} \sigma(y)}{|x-y|^{2-a}}
 \bigg) \textup{d} \sigma(x)\\
&\leq C\int_{\partial \Omega}|\varphi(y)|^2 \textup{d} \sigma(y). 
\end{split}
\end{equation}

Let us now focus on the second term in~\eqref{Sobolev_norm_int_op}.
We set
\begin{equation*}
\begin{split}
A_1&:=\{(x,y,z)\in{\partial \Omega}\times{\partial \Omega}\times{\partial \Omega}:\,4|x-y|<|x-z|\},\\
A_2&:=\{(x,y,z)\in{\partial \Omega}\times{\partial \Omega}\times{\partial \Omega}:\,4|x-y|<|y-z|\},\\
A_3&:=\{(x,y,z)\in{\partial \Omega}\times{\partial \Omega}\times{\partial \Omega}:\,4|x-y|\geq\max(|x-z|,|y-z|)\},
\end{split}
\end{equation*}
and
\begin{equation*}
K_j(x,y,z):=\frac{k(x,z)-k(y,z)}{|x-y|^{1+s}}\,\chi_{A_j}(x,y,z)\quad\text{for }j=1,2,3.
\end{equation*}
Then
\begin{equation}\label{thm bounded eq1}
\begin{split}
\int_{\partial \Omega} \int_{\partial \Omega}&\frac{|T\varphi(x)-T\varphi(y)|^2}{|x-y|^{2+2s}} \textup{d} \sigma(y) \textup{d} \sigma(x) \\
&=\int_{\partial \Omega} \int_{\partial \Omega}\bigg|\int_{\partial \Omega}\frac{k(x,z)-k(y,z)}{|x-y|^{1+s}}\,\varphi(z) \textup{d} \sigma(z)\bigg|^2 \textup{d} \sigma(y) \textup{d} \sigma(x)\\
&\leq C\sum_{j=1}^3\int_{\partial \Omega} \int_{\partial \Omega}\bigg(\int_{\partial \Omega} |K_j(x,y,z)\varphi(z)| \textup{d} \sigma(z)\bigg)^2 \textup{d} \sigma(y) \textup{d} \sigma(x).
\end{split}
\end{equation}
The three terms in the sum in the right hand side of \eqref{thm bounded eq1} are also estimated separately. In order to deal with the first one, let $\epsilon \in (0, a - s)$ be fixed. By \eqref{kernel growth}, the Cauchy-Schwarz inequality, and Lemma \ref{integral estimate} we get first
\begin{equation}\label{thm bounded eq2}
\begin{split}
\int_{\partial \Omega} \int_{\partial \Omega}& \bigg(\int_{\partial \Omega} |K_1(x,y,z)\varphi(z)| \textup{d} \sigma(z)\bigg)^2 \textup{d} \sigma(y) \textup{d} \sigma(x)\\
&\leq C\int_{\partial \Omega} \int_{\partial \Omega}\bigg(\int_{\partial \Omega} \frac{\chi_{A_1}(x,y,z)|\varphi(z)|}{|x-z|^2|x-y|^{1+s-a}}
 \textup{d} \sigma(z)\bigg)^2 \textup{d} \sigma(y) \textup{d} \sigma(x)\\
&\leq C\int_{\partial \Omega} \int_{\partial \Omega}\bigg(\int_{\partial \Omega} \frac{\textup{d}\sigma(z)}{|x-z|^{2(1-\epsilon)}}\bigg) \\
&\qquad \cdot \bigg(\int_{\partial \Omega} \frac{\chi_{A_1}(x,y,z)|\varphi(z)|^2}{|x-z|^{2(1+\epsilon)}|x-y|^{2-2(a-s)}}
 \textup{d} \sigma(z)\bigg) \textup{d} \sigma(y) \textup{d} \sigma(x)\\
&\leq C\int_{\partial \Omega} \int_{\partial \Omega} \int_{\partial \Omega}\frac{\chi_{A_1}(x,y,z)|\varphi(z)|^2}{|x-z|^{2(1+\epsilon)}|x-y|^{2-2(a-s)}}
 \textup{d} \sigma(z) \textup{d} \sigma(y) \textup{d} \sigma(x).
\end{split}
\end{equation}
Using now twice Lemma~\ref{integral estimate}, first for the integral with respect to $y$ with $b = 2 (a - s) > 0$ and $\rho = \frac{1}{4} |x-z|$ and then for the integral with respect to $x$ with $b = 2 (a - s - \epsilon) > 0$, we conclude that
\begin{equation*}\label{thm bounded eq3}
\begin{split}
\int_{\partial \Omega} \int_{\partial \Omega}&\frac{\chi_{A_1}(x,y,z)}{|x-z|^{2(1+\epsilon)}|x-y|^{2-2(a-s)}}
 \textup{d} \sigma(y) \textup{d} \sigma(x) \\
 &=\int_{\partial \Omega} \frac{1}{|x-z|^{2(1+\epsilon)}} \int_{|x-y| <  |x-z|/4} \frac{1}{|x-y|^{2-2(a-s)}} \textup{d} \sigma(y) \textup{d} \sigma(x) \\
& \leq C\int_{\partial \Omega}|x-z|^{2(a-s-\epsilon)-2} \textup{d} \sigma(x)\leq C.
\end{split}
\end{equation*}
Combining this with \eqref{thm bounded eq2} and Fubini's theorem we finally obtain
\begin{equation}\label{thm bounded K1}
\begin{split}
\int_{\partial \Omega} \int_{\partial \Omega}\bigg(\int_{\partial \Omega} |K_1(x,y,z)\varphi(z)| \textup{d} \sigma(z)\bigg)^2 &\textup{d} \sigma(y) \textup{d} \sigma(x) 
 \leq C \int_{\partial \Omega}|\varphi(z)|^2 \textup{d} \sigma(z).
\end{split}
\end{equation}

Regarding the term  in \eqref{thm bounded eq1} containing $K_2$, we simply note the symmetry relation
$K_2(x,y,z)=-K_1(y,x,z)$ and thus, Fubini's theorem and \eqref{thm bounded K1} give
\begin{equation}\label{thm bounded K2}
\int_{\partial \Omega} \int_{\partial \Omega}\bigg(\int_{\partial \Omega} |K_2(x,y,z)\varphi(z)| \textup{d} \sigma(z)\bigg)^2 \textup{d} \sigma(y) \textup{d} \sigma(x)
\leq C \int_{\partial \Omega}|\varphi(z)|^2 \textup{d} \sigma(z).
\end{equation}

Let us finally estimate the term  in \eqref{thm bounded eq1} containing $K_3$. For this purpose, set 
\begin{equation*}
A_0:=\{(x,y,z)\in{\partial \Omega}\times{\partial \Omega}\times{\partial \Omega}:\,4|x-y|\geq|x-z|\}
\end{equation*}
and
\begin{equation*}
K_0(x,y,z):=\frac{k(x,z)}{|x-y|^{1+s}}\,\chi_{A_0}(x,y,z).
\end{equation*}
Then, because of $A_3 \subset A_0$ we easily see that 
\begin{equation}\label{thm bounded split K3}
|K_3(x,y,z)|\leq|K_0(x,y,z)|+|K_0(y,x,z)|.
\end{equation}
Clearly, \eqref{thm bounded split K3}, the triangle inequality, and Fubini's theorem imply that
\begin{equation}\label{thm bounded K3 eq1}
\begin{split}
\int_{\partial \Omega} &\int_{\partial \Omega}\bigg(\int_{\partial \Omega}  |K_3(x,y,z)\varphi(z)|\text{d} \sigma(z)\bigg)^2 \textup{d} \sigma(y) \textup{d} \sigma(x)\\
&\leq C\int_{\partial \Omega} \int_{\partial \Omega}\bigg(\int_{\partial \Omega} |K_0(x,y,z)\varphi(z)| \textup{d} \sigma(z)\bigg)^2 \textup{d} \sigma(y)
 \textup{d} \sigma(x). \\
\end{split}
\end{equation}
To estimate the right hand side of~\eqref{thm bounded K3 eq1}, we perform similar estimates as in~\eqref{thm bounded eq2}. Choose $\epsilon \in (0, a - s)$. Then it follows from \eqref{kernel growth} and the Cauchy-Schwarz inequality that
\begin{equation}\label{thm bounded K0 eq3}
\begin{split}
\int_{\partial \Omega} \int_{\partial \Omega}&\bigg(\int_{\partial \Omega} |K_0(x,y,z)\varphi(z)| \textup{d} \sigma(z)\bigg)^2
 \textup{d} \sigma(y) \textup{d} \sigma(x)\\
&~\leq C\int_{\partial \Omega} \int_{\partial \Omega}\bigg(\int_{\partial \Omega}  \frac{\chi_{A_0}(x,y,z)}{|x-z|^{2-a}} \frac{|\varphi(z)|}{|x-y|^{1+s}} \textup{d} \sigma(z)\bigg)^2
 \textup{d} \sigma(y) \textup{d} \sigma(x)\\
 &~\leq C\int_{\partial \Omega} \int_{\partial \Omega} \bigg(\int_{\partial \Omega} \frac{|\varphi(z)|^2}{|x-y|^{2+2s} |x-z|^{2 (1 - \epsilon)}}\text{d} \sigma(z) \bigg) \\
 &\qquad \qquad \cdot \bigg(\int_{|x-z| \leq 4|x-y|} \frac{1}{|x-z|^{2(1-a + \varepsilon)}} \textup{d} \sigma(z) \bigg)
 \textup{d} \sigma(y) \textup{d} \sigma(x) \\
 &~\leq C\int_{\partial \Omega} \int_{\partial \Omega} \int_{\partial \Omega} \frac{|\varphi(z)|^2}{|x-y|^{2 (1 + s - a + \epsilon)} |x-z|^{2 (1 - \epsilon)}}\text{d} \sigma(z) 
 \textup{d} \sigma(y) \textup{d} \sigma(x),
\end{split}
\end{equation}
where Lemma~\ref{integral estimate} with $b = 2(a-\epsilon) > 0$ and $\rho = 4|x-y|$ was applied in the last step.
Applying now two more times Lemma~\ref{integral estimate}, first for the integral with respect to $y$ with $b = 2 (a - s - \epsilon) > 0$ and then for the integral with respect to $x$ with $b = 2 \epsilon$, we find that
\begin{equation*}
\begin{split}
\int_{\partial \Omega} \int_{\partial \Omega}& \frac{1}{|x-y|^{2 (1 + s - a + \epsilon)} |x-z|^{2 (1 - \epsilon)}}
 \textup{d} \sigma(y) \textup{d} \sigma(x) 
 \leq C \int_{\partial \Omega} \frac{1}{|x-z|^{2 (1 - \epsilon)}}
 \textup{d} \sigma(x) \leq C
\end{split}
\end{equation*}
holds independently of $z$. Using this in~\eqref{thm bounded K0 eq3} we conclude 
\begin{equation}\label{thm bounded K0 eq4}
\begin{split}
\int_{\partial \Omega} \int_{\partial \Omega}&\bigg(\int_{\partial \Omega} |K_0(x,y,z)\varphi(z)| \textup{d} \sigma(z)\bigg)^2
 \textup{d} \sigma(y) \textup{d} \sigma(x) \leq C \int_{\partial \Omega} |\varphi(z)|^2 \text{d} \sigma(z).
\end{split}
\end{equation}
Hence, it follows from \eqref{thm bounded K3 eq1} 
and \eqref{thm bounded K0 eq4} that
\begin{equation}\label{thm bounded K3}
\int_{\partial \Omega} \int_{\partial \Omega}\bigg(\int_{\partial \Omega} |K_3(x,y,z)\varphi(z)| \textup{d} \sigma(z)\bigg)^2 \textup{d} \sigma(y) \textup{d} \sigma(x)
\leq C \int_{\partial \Omega} |\varphi(z)|^2 \text{d} \sigma(z).
\end{equation}

Finally, a combination of \eqref{thm bounded eq1}, \eqref{thm bounded K1}, \eqref{thm bounded K2}, and \eqref{thm bounded K3} shows that
\begin{equation*}
\int_{\partial \Omega} \int_{\partial \Omega}\frac{|T\varphi(x)-T\varphi(y)|^2}{|x-y|^{2+2 s}} \textup{d} \sigma(y) \textup{d} \sigma(x)
\leq C \int_{\partial \Omega} |\varphi(z)|^2 \text{d} \sigma(z)
\end{equation*}
which, together with \eqref{thm bounded control1}, implies 
$\|T\varphi\|_s\leq C\|\varphi\|_\Omega$.
\end{proof}
\end{appendix}

\vskip 0.8cm
\noindent {\bf Acknowledgments.}  
Jussi Behrndt gratefully
acknowledges financial support
by the Austrian Science Fund (FWF): Project P~25162-N26.
Albert Mas was partially supported by
MTM2017-84214 and MTM2017-83499 projects of the MCINN (Spain),
2017-SGR-358 project of the AGAUR (Catalunya) and ERC-2014-ADG project 
HADE Id.\! 669689 (European Research Council).




\end{document}